\documentclass[11pt]{article}

\usepackage[pdfpagelabels]{hyperref} \hypersetup{colorlinks = true, linktoc = all, linkcolor = black, citecolor = black}
\usepackage[normalem]{ulem} \usepackage{algorithm}
\usepackage{algorithmicx} \usepackage{algpseudocode}
\usepackage{pgfplots} \usepackage[toc,page]{appendix}
\usepackage{upgreek} \usepackage{fullpage}
\usepackage[english]{babel} \usepackage[latin1]{inputenc}
\usepackage{amsmath}
\usepackage{amsfonts}
\usepackage{xcolor}
\usepackage{graphicx}
\usepackage{amsthm}
\usepackage{psfrag}
\usepackage{subfig}
\usepackage{amssymb}
\usepackage{tikz}
\usepackage{ifthen} \usepackage{xifthen} \usepackage{dsfont}
\usepackage{bm} \usepackage{accents}
\usepackage{xparse} %
\usepackage{mathtools} %
\usetikzlibrary{arrows,calc} \usetikzlibrary{decorations,arrows}
\usetikzlibrary{decorations.pathreplacing}

\usepackage[textsize=tiny]{todonotes}
\definecolor{jz}{rgb}{0.1,0.45,0.1}

\newcommand{\proofref}[1]{}
\newcommand{\D}{\mathrm{D}} \newcommand{\measi}{{\rho}}
\newcommand{\measii}{{\pi}}

\newcommand{\temp}[1]{}
\newcommand{\norm}[2][]{\| #2 \|_{#1}}

\newcommand{\hide}[1]{}

\newcommand{\normc}[2][]{\left\| #2 \right\|_{#1}}

\newcommand{\set}[2]{\{#1\,:\,#2\}} \newcommand{\setc}[2]{\left\{#1\,
    :\,#2\right\}} \usepackage{enumitem}

\newcommand{\dd}{\;\mathrm{d}}
\newcommand{\ddd}{\mathrm{d}}
\DeclareMathOperator{\divg}{div}
\DeclareMathOperator{\supp}{supp} 
\DeclareMathOperator*{\essinf}{ess\,inf} 
\DeclareMathOperator{\argmin}{arg \,min}
\DeclareMathOperator*{\esssup}{ess\,sup}

 \newcommand{\ii}{\mathrm{i}}

\newcommand{\dfn}{\vcentcolon=} \newcommand{\dfnn}{=\vcentcolon}

\newtheorem{theorem}{Theorem}[section]

\newtheorem{lemma}[theorem]{Lemma}

\newtheorem{proposition}[theorem]{Proposition}
\newtheorem{remark}[theorem]{Remark}

\newtheorem{assumption}[theorem]{Assumption}

\numberwithin{equation}{section}
\newcommand{\cA} {{\mathcal A}} \newcommand{\cB} {{\mathcal B}}
 \newcommand{\cE}{{\mathcal E}}

 \newcommand{\cN}{{\mathcal N}}

\newcommand{\C}{{\mathbb C}} \newcommand{\N}{{\mathbb N}}
\newcommand{\R}{{\mathbb R}} \newcommand{\bbP}{{\mathbb P}}
 
 \newcommand{\scr}[1]{{\mathfrak{#1}}}
 
\newcommand{\eps}{\varepsilon}

\newcommand{\bsone}{\boldsymbol 1}

\newcommand{\bsW}{{\boldsymbol W}}
 \newcommand{\bsb}{{\boldsymbol b}}

\newcommand{\bsx}{{\boldsymbol x}} \newcommand{\bst}{{\boldsymbol t}}
\newcommand{\bsy}{{\boldsymbol y}} \newcommand{\bsz}{{\boldsymbol z}}

\newcommand{\bszeta}{{\boldsymbol \zeta}}
\newcommand{\bsxi}{{\boldsymbol \xi}}

 \newcommand{\bsvarrho}{{\bm
    \varrho}} 
\newcommand{\bsnu}{{\bm \nu}} \newcommand{\bsmu}{{\bm \mu}}
\newcommand{\bseta}{{\bm \eta}} \newcommand{\bse}{{\bm e}}
 
\newcommand{\bsdelta}{{\bm \delta}}
\newcommand{\bsvarsigma}{{\bm \varsigma}}

\usepackage{authblk}

\begin{document}
\title{{Sparse approximation of triangular transports. Part I: the
    finite dimensional case}\thanks{This paper was written during the
    postdoctoral stay of JZ at MIT.  JZ acknowledges support by the
    Swiss National Science Foundation under Early Postdoc Mobility
    Fellowship 184530. YM and JZ acknowledge support from the United
    States Department of Energy, Office of Advanced Scientific
    Computing Research, AEOLUS Mathematical Multifaceted Integrated
    Capability Center.}  \thanks{The original manuscript
    \cite{2006.06994} has been split into two parts; the present paper
    is the first part and the second part is \cite{zm2}.
  }} \author[1]{Jakob Zech} \author[2]{Youssef
  Marzouk} \affil[1]{\footnotesize Heidelberg University, 69120
  Heidelberg, Germany,
  \href{mailto:jakob.zech@uni-heidelberg.de}{jakob.zech@uni-heidelberg.de}}
\affil[2]{\footnotesize Massachusetts Institute of Technology,
  Cambridge, MA 02139, USA,
  \href{mailto:ymarz@mit.edu}{ymarz@mit.edu}}
  
\maketitle

\abstract{ {For two probability measures $\measi$ and $\measii$
    with analytic densities on the $d$-dimensional cube $[-1,1]^d$, we} 
  investigate the approximation of the unique triangular monotone
  Knothe--Rosenblatt transport $T:[-1,1]^d\to [-1,1]^d$, such that the
  pushforward $T_\sharp\measi$ equals $\measii$. It is shown that for
  $d\in\N$ there exist approximations $\tilde T$ of $T$, based on
  either sparse polynomial expansions or {deep ReLU neural}
  networks, such that the distance between $\tilde T_\sharp\measi$ and
  $\measii$ decreases exponentially. More precisely, we {prove}
  error bounds of the type $\exp(-\beta N^{1/d})$ (or
  $\exp(-\beta N^{1/(d+1)})$ for neural networks), where $N$ refers to
  the dimension of the ansatz space (or the size of the network)
  containing $\tilde T$; the notion of distance comprises
  the Hellinger distance,
  {the total variation distance,
  the Wasserstein distance}
  and the Kullback--Leibler
  divergence. {Our} construction guarantees $\tilde T$ to be a
  monotone triangular bijective transport on the hypercube
  $[-1,1]^d$. Analogous results hold for the inverse transport
  $S=T^{-1}$. The proofs are constructive, and we give an explicit
  \textit{a priori} description of the ansatz space, which can be used
  for numerical implementations. %
}
\\

\noindent
{\bf Key words:} transport maps, domains of holomorphy, uncertainty quantification, sparse approximation, neural networks, sampling

\noindent
{\bf Subject classification:} 32D05, 41A10, 41A25, 41A46, 62D99, 65D15

\numberwithin{equation}{section}

\section{Introduction}
A long-standing challenge in applied mathematics and statistics
is to approximate integrals w.r.t.\ a probability measure $\measii$,
given only through its (possibly unnormalized) Lebesgue
density $f_{\measii}$, on a high-dimensional integration domain. Here
we consider probability measures on the bounded domain $[-1,1]^d$.
One of the main applications is in Bayesian
inference, %
{where parameters $\bsy\in [-1,1]^d$ are inferred} from
noisy and/or indirect data. In this case $\measii$ is interpreted as
the so-called posterior measure. It is obtained by Bayes' rule, %
{and} encompasses all information about the parameters given the
data.  A typical goal is to compute the expectation
$\int_{[-1,1]^d}g(\bsy)\dd\measii( \bsy)$ of some quantity of interest
$g:[-1,1]^d\to\R$ w.r.t.\ the
posterior. %

Various approaches to high-dimensional integration have been proposed
in the literature. One of the most common strategies consists of Monte
Carlo-type sampling, e.g., with Markov chain Monte Carlo 
(MCMC) methods \cite{10.5555/1051451}. Metropolis-Hastings
MCMC algorithms, for instance, are versatile and simple to
implement. Yet the mixing times of standard Metropolis algorithms can scale somewhat
poorly
with the dimension $d$. %
 (Here function-space MCMC algorithms \cite{cotter2013mcmc,DILI,rudolf2018generalization} 
and others \cite{Morzfeld,Tong}
represent notable exceptions, with dimension-independent mixing for
certain classes of target measures.)
In general, MCMC algorithms may suffer from slow convergence and
possibly long burn-in phases.  Furthermore, MCMC is intrinsically
serial, which can make sampling infeasible when each evaluation of
$f_\measii$ is costly. Variational inference, {e.g.,
  \cite{doi:10.1080/01621459.2017.1285773},} can improve on some of
these drawbacks. It replaces the task of sampling from $\measii$ by an
optimization problem. The idea is to minimize (for instance) the KL
divergence between $\measii$ and a second measure $\tilde\measii$ in
some given class of tractable measures, where `tractable' means that
independent and identically distributed (iid) samples from $\tilde\measii$ can easily be produced.

Transport based methods are one instance of this category: given an
easy-to-sample-from ``reference'' measure $\measi$, we look for an
approximation $\tilde T$ to the transport $T$ such that
$T_\sharp\measi=\measii$. Here $T_\sharp \measi$ denotes the
pushforward measure, i.e., $T_\sharp\measi(B)=\measi(T^{-1}(B))$ for
all measurable sets $B$.  Then
$\tilde\measii\dfn \tilde T_\sharp\measi$ is an approximation to the
``target'' $\measii=T_\sharp\measi$.  Unlike in optimal transportation
theory, e.g., \cite{MR2459454,santambrogio}, $T$ is not required to
minimize some cost.  This
allows %
imposing further structure on $T$ in order to simplify numerical
computations.  {In this paper we concentrate on the triangular
  Knothe--Rosenblatt (KR)
  rearrangement \cite{MR49525}, which has been
  found to be particularly useful in this context.  The reason for
  concentrating on the KR transport (rather than, for
  instance, optimal transport) is that it is widely used in practical
  algorithms \cite{MR2972870,
    spantini2018inference,jaini2019sum,wehenkel2019unconstrained},
  with the advantages of allowing easy inversion, efficient
  computation of the Jacobian determinant, and direct extraction of
  certain conditionals \cite{spantini2019coupling}; and from a
  mathematical standpoint, its explicit construction makes it amenable
  to a rigorous analysis.}
Given an approximation $\tilde T$ to (some) transport
$T$, for a random variable $X\sim \measi$ it holds that
$\tilde T(X)\sim \tilde\measii$. Thus iid samples from $\tilde\measii$
are obtained via $(\tilde T(X_i))_{i=1}^n$, where $(X_i)_{i=1}^n$ are
$\measi$-distributed iid.  This strategy, and various refinements and
variants, have been investigated theoretically and empirically, and
successfully applied in Bayesian inference; see, e.g.,
\cite{MR2972870,MR3821485,MR3810512,MR3800241,spantini2018inference,MR4065222}.

Normalizing flows, {now widely used} in the machine learning literature
\cite{pmlr-v37-rezende15,kobyzev2019normalizing,papamakarios2019normalizing}
for variational inference, generative modeling, and density
estimation, are another instance of this transport framework. In
particular, many so-called ``autoregressive'' flows, e.g.,
\cite{papamakarios2017masked,jaini2019sum}, employ specific {neural
  network} parameterizations of triangular
maps.
{A
  complete mathematical convergence analysis of these
  methods is not yet available in the literature.
}

Sampling methods can be contrasted with deterministic approaches,
where a quadrature rule is constructed that converges at a guaranteed
(deterministic) rate for all integrands in some function class.
Unlike sampling methods, deterministic quadratures can achieve
higher-order convergence. They even overcome the curse of
dimensionality, presuming certain smoothness properties of the
integrand. We refer to sparse-grid quadratures
\cite{CCS13,MR3056084,MR1669959,GOE16,MR3706335,ZS17} and quasi-Monte
Carlo quadrature \cite{MR3907407,MR3640631,MR3939383} as
examples. It is difficult to construct deterministic quadrature rules
for an arbitrary measure $\measii$, however, so typically they are only
available in specific cases such as for the Lebesgue or Gaussian
measure. Interpreting $\int_{[-1,1]^d}g(t)\dd\measii( t)$ as the
integral $\int_{[-1,1]^d}g(t)f_\measii(t)\dd t$ w.r.t.\ the Lebesgue
measure, such methods are still applicable. In Bayesian inference,
however, $\measii$ can be strongly concentrated, corresponding to
either small noise in the observations or a large data set. Then this
viewpoint may become problematic. For example, the error of Monte
Carlo quadrature depends on the variance of the integrand. The
variance of $gf_\measii$ (w.r.t.\ the Lebesgue measure) can be much
larger than the variance of $g$ (w.r.t.\ $\measii$) when $\measii$
is strongly concentrated, i.e., when $f_\measii$ is very ``peaky.'' This
problem was addressed in \cite{MR3580124} by combining an adaptive
sparse-grid quadrature with a linear transport map.  This approach
combines the advantage of high-order convergence with quadrature
points that are mostly within the area where $\measii$ is
concentrated.  Yet if $\measii$ is multimodal (i.e., concentrated in
several separated areas) or unimodal but strongly
non-Gaussian,
then the linearity of the transport precludes such a strategy from
being successful. {A similar statement can be made about the method
  analyzed in \cite{MR4125981}, where the (Gaussian) Laplace
  approximation is used in place of a strongly concentrated posterior
  measure. For such $\measii$,} the combination of nonlinear transport
maps with deterministic quadrature rules may lead to %
{significantly} improved algorithms.

{In a related spirit}, we also mention {interacting particle systems such as}
kernel-based Stein variational gradient descent (SVGD) and its variants, which have
recently emerged as a promising research direction
\cite{NIPS2016_6338,3327546.3327591}. Put simply, these methods try to
find $n$ points $(x_i)_{i=1}^n$ minimizing an approximation of the
KL divergence between the target $\measii$ and the uniform discrete
probability measure with support $(x_i)_{i=1}^n$. A discrete
convergence analysis is not yet available, but a connection between
the mean field limit and gradient flow has been established
\cite{NIPS2017_6904,MR3919409,1912.00894}.

In this paper, we analyze the approximation of the transport $T$
satisfying $T_\sharp\measi=\measii$ under the assumption that the
reference and target densities $f_\measi$ and $f_\measii$ are
analytic.
This assumption is quite strong, but
satisfied in many applications, including the main
  application we have in mind, which is
Bayesian inference in
{partial differential equations (PDE)}.
The reference $\measi$ can be chosen at
will, e.g., as a uniform measure so that $f_\measi$ is constant (thus
analytic) on $[-1,1]^d$. And here the target {density} $f_\measii$ is a posterior
density, stemming from a PDE-driven likelihood function.
{For certain} linear and nonlinear PDEs {(for example the
Navier--Stokes equations)}, it
can be shown under suitable conditions that the corresponding posterior density is indeed
an analytic function of the parameters; we refer for instance to
\cite{MR2903278,CCS13}. %

{
As outlined above, $T$ can be employed in the construction of either
sampling-based or deterministic quadrature methods. Understanding the
approximation properties of $T$ is the first step towards a rigorous
convergence analysis of such algorithms. In practice, once a suitable
ansatz space has been identified, a (usually non-convex) optimization
problem must be solved to find a suitable $\tilde T$ within the
ansatz space. While this optimization is beyond the scope of the
current paper, we intend to empirically analyze a numerical algorithm
based on the present results in a forthcoming publication.}

{ Throughout we consider transports on $[-1,1]^d$ with $d\in\N$.
  It is straightforward to generalize all the presented results to
  arbitrary Cartesian products
  $\bigtimes_{j=1}^d [a_j,b_j]\subset\R^d$ with
  $-\infty<a_j<b_j<\infty$ for all $j$, via an affine transformation
  of all occurring functions. Most (theoretical and empirical) earlier
  works on this topic have however assumed measures supported on all
  of $\R^d$.  A similar analysis in the unbounded case, as well as
  numerical experiments and the development and improvement of
  algorithms in this case, will be the topics of future research.}

\subsection{Contributions}
For $d\in\N$ and {under the assumption that the reference and
  target densities $f_\measi$, $f_\measii:[-1,1]^d\to (0,\infty)$ are
  analytic}, we prove that there exist sparse polynomial spaces of
dimension {$N\in\N$,} in which the best approximation of the
KR transport $T$ converges to $T$ at the exponential
rate $\exp(-\beta N^{1/d})$ for some $\beta>0$ {as $N\to\infty$};
see Thm.~\ref{thm:TdN}. %
{To guarantee that the approximation
  $\tilde T:[-1,1]^d\to [-1,1]^d$ is bijective, we propose to use an
  ansatz of rational functions, which ensures this property and
  retains the same convergence rate. In this case $N$ refers to the
  dimension of the polynomial space used in the denominator and
  numerator, i.e., again to the number of degrees of freedom; see
  Thm.~\ref{THM:POLYD}.} %
The argument is based on a result quantifying the regularity of $T$ in
terms of its {complex} domain of analyticity, which is given in
Thm.~\ref{THM:COR:DINN}.

{In Sec.~\ref{SEC:MEASURES} we investigate the implications of
  approximating the transport map
  for the corresponding pushforward measures.  We show that closeness
  of the transports in $W^{1,\infty}$ implies closeness of the
  pushforward measures in the Hellinger distance, the total variation
  distance and the KL
  divergence.
  A similar statement is true for the Wasserstein distance
  if the transports are close in $L^\infty$.
  Specifically,
  Prop.~\ref{PROP:MEASCONVD} %
  states the same $\exp(-\beta N^{1/d})$ error convergence as for the
  approximation of the transport is obtained for the distance between
  the pushforward $\tilde T_\sharp\measi$ and the target $\measii$.
}
We provide lower bounds on $\beta>0$, based on properties of
$f_\measi$ and $f_\measii$. Furthermore, given $\eps>0$, we provide
\textit{a priori} {ansatz spaces} guaranteeing the best approximation
in {this} ansatz space to be $O(\eps)$-close to $T$; %
see Thm.~\ref{THM:POLYD} and Sec.~\ref{SEC:EXAMPLE}. This allows to
improve upon existing numerical algorithms: previous approaches were
either based on heuristics or on adaptive (greedy) enrichment of the
ansatz space \cite{MR2972870}, neither of which can guarantee {asymptotic
  convergence or convergence rates in general. Moreover, greedy
  methods are inherently sequential (in contrast to \textit{a priori} approaches),
  which can slow down computations.}

Using known approximation properties of polynomials by rectified
linear unit (ReLU) neural networks (NNs), we also show that ReLU NNs can
approximate $T$ at the exponential rate {$\exp(-\beta N^{1/(1+d)})$.
In this case $N$ refers to the number of trainable parameters
(``weights and biases'') in the network;} see Thm.~\ref{THM:RELU}
{for the convergence of the transport map and
  Prop.~\ref{prop:relupushforward} for the convergence of the
  pushforward measure}. We point out that normalizing flows in machine
learning also try to approximate $T$ with a neural network; see, for
example,
\cite{pmlr-v37-rezende15,pmlr-v80-huang18d,1810.01367,2002.02798}.
{Recent theoretical works on the expressivity of neural network
  {representations of} transport maps include \cite{2004.08867,2006.11469,2012.02414}, which
  provide universal approximation results;
  moreover, \cite{pmlr-v108-kong20a} provides estimates
  on the required network depth. In the present work we do not merely
  show universality, i.e., approximability of $T$ by neural networks,
  but we even prove an exponential convergence \emph{rate}. Similar
  results have not yet been established to the best of our knowledge.}

\subsection{Main ideas}
Consider the case $d=1$. Let $\measii$ and $\measi$ be two probability
measures on $[-1,1]$ with strictly positive Lebesgue densities
$f_\measi$, $f_\measii:[-1,1]\to\set{x\in\R}{x>0}$.
Let
\begin{equation}\label{eq:T1d}
  F_\measi(x)\dfn \int_{-1}^x f_\measi(t)\dd t,\qquad
  F_\measii(x)\dfn \int_{-1}^x f_\measii(t)\dd t,
\end{equation}
be the corresponding cumulative distribution functions (CDFs), which are strictly monotonically increasing
and bijective from $[-1,1]\to [0,1]$.
For any
interval $[a,b]\subseteq [-1,1]$, it holds that
\begin{equation*}
  \measi(F_\measi^{-1}\circ F_\measii([a,b])) =
  \int_{F_\measi^{-1}(F_\measii(a))}^{F_\measi^{-1}(F_\measii(b))}f_\measi(t)\dd t
  =
  F_\measii(b)-F_\measii(a)=\measii([a,b]).
\end{equation*}
Hence $T\dfn F_{\measii}^{-1}\circ F_\measi$ is the unique monotone
transport satisfying $\measi\circ T^{-1}=\measii$, i.e.,
$T_\sharp\measi=\measii$. The formula
$T=F_{\measii}^{-1}\circ F_\measi$ implies that $T$ inherits the
smoothness of $F_{\measii}^{-1}$ and $F_\measi$. Thus it is at least
as smooth as $f_\measi$ and $f_\measii$ (more precisely, $f_\measi$,
$f_\measii\in C^k$ imply $T\in C^{k+1}$). We will see in
Prop.~\ref{prop:1d} that if $f_\measi$ and $f_\measii$ are analytic,
the domain of analyticity of $T$ is (under further conditions and in a
suitable sense) proportional to the minimum of the domain of
analyticity of $f_\measi$ and $f_\measii$. {By the ``domain of
  analyticity,'' we mean the domain of holomorphic extension to the
  complex numbers.}

Knowledge of the analyticity domain of $T$ allows to prove exponential
convergence of polynomial approximations: Assume for the moment that
$T:[-1,1]\to [-1,1]$ admits an analytic extension to the complex disc
with radius $r>1$ and center $0\in\C$. Then
$T(x)=\sum_{k\in\N}\frac{d^k}{dy^k} T(y)|_{y=0} \frac{x^k}{k!}$ for
$x\in [-1,1]$, and the $k$th Taylor coefficient can be bounded with
Cauchy's integral formula by $C r^{-k}$. This implies that the $n$th
Taylor polynomial uniformly approximates $T$ on $[-1,1]$ with error
$O(r^{-n})=O(\exp(-\log(r)n))$. Thus $r$ determines the rate of
exponential convergence.

The above construction of the transport can be generalized to
the KR transport
$T:[-1,1]^d\to [-1,1]^d$ for $d\in\N$.
We will determine an analyticity domain for each component $T_k$ of 
$T=(T_k)_{k=1}^d$: not in the shape of a polydisc, but rather as a
{pill-shaped set}
containing $[-1,1]^k$. The reason is that analyticity of $f_\measi$
and $f_\measii$ does not imply the existence of a polydisc, but
does imply the existence of {such pill-shaped domains}.
Instead of Taylor
expansions, one can then prove exponential convergence of Legendre
expansions. Rather than approximating $T$ with Legendre polynomials,
we introduce a correction guaranteeing our approximation
$\tilde T:[-1,1]^d\to [-1,1]^d$ to be bijective. This results in a
rational function $\tilde T$. Using existing theory for ReLU networks,
we also deduce a ReLU approximation result.

\subsection{Outline}
{
In Sec.~\ref{sec:notation} we introduce notation.
Sec.~\ref{SEC:T} recalls the construction of the
triangular KR transport $T$. %
In Sec.~\ref{SEC:ANT} we {investigate} the domain of
analyticity of $T$. {Sec.~\ref{SEC:APPROX} %
applies the results
of Sec.~\ref{SEC:ANT} to prove 
exponential convergence rates for the approximation of the transport
through sparse polynomial expansions.}
Subsequently, Sec.~\ref{sec:relu} discusses a deep neural network
approximation result for the transport.  We then use these results in
Sec.~\ref{SEC:MEASURES} to establish convergence rates for the
associated measures {(rather than the transport maps
  themselves)}. Finally, in Sec.~\ref{SEC:EXAMPLE} we present a
standard example in uncertainty quantification, and demonstrate how
our results may be used in inference algorithms.}

{For the convenience of the reader, in
  Sec.~\ref{sec:onedanalyticty} we discuss analyticity of the
  transport map in the one dimensional case separately from the general
  case $d\in\N$ (which builds on similar ideas but is significantly
  more technical), and provide most parts of the proof in the main
  text. In the remaining sections, all proofs and technical arguments
  are deferred to the appendix.}

\subsection{Notation}\label{sec:notation}

\subsubsection{{Sequences, multi-indices, and polynomials}}
Boldface characters denote
vectors, %
e.g., $\bsx=(x_i)_{i=1}^d$, $d\in\N$. %
For $j\le k\le d$, we denote slices by
$\bsx_{[k]}\dfn (x_i)_{i=1}^k$ and
$\bsx_{[j:k]}\dfn (x_i)_{i=j}^k$.

For a multi-index $\bsnu\in\N_0^d$, %
$\supp\bsnu\dfn \set{j}{\nu_j\neq 0}$,
and $|\bsnu|\dfn \sum_{j\in\supp\bsnu}\nu_j$,
where empty sums equal $0$ by convention. Additionally, empty products
equal $1$ by convention and
$\bsx^\bsnu\dfn\prod_{j\in\supp\bsnu} x_j^{\nu_j}$.  %
We write $\bseta\le\bsnu$ if $\eta_j\le\nu_j$ for all $j$, and
$\bseta<\bsnu$ if $\bseta\le\bsnu$ and there exists $j$ such that
$\eta_j<\nu_j$. A subset
$\Lambda\subseteq\N_0^d$ %
is called \emph{downward closed} if it is finite and satisfies
$\set{\bseta\in\N_0^d}{\bseta\le\bsnu}\subseteq \Lambda$
whenever $\bsnu\in\Lambda$.

For $n\in\N$, $\bbP_n\dfn {\rm span}\set{x^i}{i\in\{0,\dots,n\}}$,
where the span is understood over the field $\R$ (rather than
$\C$). Moreover, for $\Lambda\subseteq\N_0^d$ %
\begin{equation}\label{eq:PL}
  \bbP_\Lambda\dfn {\rm span}\set{\bsx^\bsnu}{\bsnu\in\Lambda},
\end{equation}
and a function $p\in\bbP_\Lambda$ maps from
$\C^d\to\C$. %
If $\Lambda=\emptyset$,
$\bbP_\emptyset\dfn \{0\}$, i.e., $\bbP_\emptyset$ only contains the constant
$0$ function.

\subsubsection{{Real and complex numbers}}
{Throughout, $\R^d$ is equipped with the Euclidean norm and
$\R^{d\times d}$ with the spectral norm.} We write
$\R_+\dfn \set{x\in\R}{x>0}$ and %
{denote the real and imaginary part of $z\in\C$
by $\Re(z)$, $\Im(z)$, respectively.}
For any
$\delta\in\R_+$ and $S\subseteq \C$
\begin{equation}\label{eq:cBdelta}
  \cB_\delta(S)\dfn \set{z\in\C}{\exists y\in S\text{ s.t.~}|z-y|<\delta},
\end{equation}
and thus $\cB_\delta(S)=\bigcup_{x\in S} \cB_\delta(x)$.  %
For $\bsdelta=(\delta_i)_{i=1}^d\subset \R_+$,
$\cB_{\bsdelta}(S)\dfn \bigtimes_{i=1}^d
\cB_{\delta_i}(S)\subseteq\C^d$.
If
we omit the argument $S$, then $S\dfn 0$, i.e.,
$\cB_\delta\dfn \cB_\delta(0)$.
\subsubsection{Measure spaces}
Throughout $[-1,1]^d$ is equipped with the Borel $\sigma$-algebra.
With $\lambda$ denoting the Lebesgue measure on $[-1,1]$, 
{$\mu\dfn \frac{\lambda}{2}$. By abuse of notation
also $\mu\dfn \otimes_{j=1}^k\frac{\lambda}{2}$, where $k\in\N$
will always be clear from context.}

If we write
``$f:[-1,1]^d\to\R_+$ is a
probability density,'' %
we mean that $f$ is measurable,
$f(\bsx)>0$ for all $\bsx\in [-1,1]^d$ and
$\int_{[-1,1]^d}f(\bsx)\dd\mu(\bsx)=1$, i.e., $f$ is a probability
density w.r.t.\ the measure $\mu$ on $[-1,1]^d$.  %

{
\subsubsection{Derivatives and function spaces}
For $f:[-1,1]^d\to\R$ (or $\C$) we denote by $\partial_kf(\bsx)\dfn \frac{\partial}{\partial x_k}f(\bsx)$ the (weak) partial derivative.
For $\bsnu\in\N_0^d$, we write instead
  $\partial^\bsnu_\bsx f(\bsx)\dfn \frac{\partial^{|\bsnu|}}{\partial_{x_1}^{\nu_1}\cdots \partial_{x_d}^{\nu_d}} f(\bsx)$.
  For $m\in\N_0$, the space $W^{m,\infty}([-1,1]^d)$ consists of all
  $f:[-1,1]^d\to\R$ with finite
  $\norm[W^{m,\infty}({[-1,1]^d})]{f}\dfn \sum_{j=0}^m \esssup_{\bsx\in [-1,1]^d} \norm[]{d^j f(\bsx)}$. 
  Here $d^0f(\bsx)=f(\bsx)$ and for $j\ge 1$,
  $d^j f(\bsx)\in \R^{d\times\cdots\times d}\simeq\R^{d^j}$ denotes
  the weak $j$th derivative, and $\norm{d^jf(\bsx)}$ denotes the norm
  on $\R^{d\times\cdots\times d}$ induced by the Euclidean norm. %
  For
  $j=1$ we simply write $df\dfn d^1f$. By abuse of notation, e.g., for
  $T=(T_k)_{k=1}^d:\R^d\to\R^d$ we also write
  $T\in W^{m,\infty}([-1,1]^d)$ meaning that
  $T_k\in W^{m,\infty}([-1,1]^d)$ for all $k\in\{1,\dots,d\}$, and in
  this case
  $\norm[{W^{m,\infty}([-1,1]^d)}]{T}\dfn\sum_{k=1}^d\norm[{W^{m,\infty}([-1,1]^d)}]{T_k}$. Similarly,
  for a measure $\nu$ on $[-1,1]^d$ and $p\in [1,\infty)$ we denote by
  $L^p([-1,1]^d,\nu)$ the usual $L^p$ space with norm
  $\norm[{L^p([-1,1]^d,\nu)}]{f}\dfn
  (\int_{[-1,1]^d}\norm{f(\bsx)}\dd\nu(\bsx))^{1/p}$, where either
  $f:[-1,1]^d\to\R$ or $f:[-1,1]^d\to\R^d$.}

\subsubsection{Transport maps}
Let $d\in\N$. A map $T:[-1,1]^d\to [-1,1]^d$ is called
\emph{triangular} if $T=(T_j)_{j=1}^d$ and each
$T_j:[-1,1]^j\to [-1,1]$ is a function of
$\bsx_{[j]}=(x_i)_{i=1}^j$. We say that $T$ is \emph{monotone} if
$x_j\mapsto T_j(\bsx_{[j-1]},x_j)$ is monotonically increasing
for every $\bsx_{[j-1]}\in [-1,1]^{j-1}$, $j\in\{1,\dots,d\}$. Note that
$x_j\mapsto T_j(\bsx_{[j-1]},x_j):[-1,1]\to [-1,1]$ being bijective %
for every $\bsx_{[j-1]}\in [-1,1]^{j-1}$, $j\in\{1,\dots,d\}$, implies
$T:[-1,1]^d\to [-1,1]^d$ to be bijective. {Similar to our notation
  for vectors, for the vector valued function $T=(T_j)_{j=1}^d$ we
  write  $T_{[k]}\dfn (T_j)_{j=1}^k$.  Note that for a triangular
  transport, it holds that $T_{[k]}:[-1,1]^k\to [-1,1]^k$.}

For a measurable bijection $T:[-1,1]^d\to [-1,1]^d$ and a
measure $\measi$ on $[-1,1]^d$,
the \emph{pushforward} $T_\sharp\measi$ and the \emph{pullback} $T^\sharp\measi$ {measures}
are defined as
\begin{equation*}
  T_\sharp\measi(A) = \measi(T^{-1}(A))\qquad\text{and}\qquad
  T^\sharp\measi(A) = \measi(T(A))
\end{equation*}
for all measurable $A\subseteq [-1,1]^d$.

The inverse transport $T^{-1}:[-1,1]^d\to [-1,1]^d$ is denoted by
$S$.
If $T:[-1,1]^d\to [-1,1]^d$ is a triangular monotone bijection, then
the same is true for $S:[-1,1]^d\to [-1,1]^d$: it holds that
$S_1(y_1)=T_1^{-1}(y_1)$ and
\begin{equation*}
  S_j(y_1,\dots,y_j) = T_j(S_1(y_1),\dots,S_{j-1}(y_1,\dots,y_{j-1}), \, \cdot \, )^{-1}(y_j).
\end{equation*}
Also note that $T_\sharp\measi=\measii$ is equivalent to
$S^\sharp\measi=\measii$.

\section{Knothe--Rosenblatt transport}\label{SEC:T}
\label{sec:knothe}
Let $d\in\N$.  Given a \emph{reference} probability measure $\measi$
and a \emph{target} probability measure $\measii$ on $[-1,1]^d$, under
certain conditions (e.g., as detailed below) 
 the KR transport is the (unique)
triangular monotone transport $T:[-1,1]^d\to [-1,1]^d$ such that
$T_\sharp\measi=\measii$.  We now recall the explicit construction of
$T$, as for instance presented in \cite{santambrogio}. Throughout it
is assumed that $\measii\ll\mu$ and $\measi\ll\mu$ have continuous and
positive densities, i.e.,
\begin{equation}\label{eq:densities}
f_\measii\dfn \frac{\ddd\measii}{\ddd\mu}\in C^0([-1,1]^d;\R_+)\qquad\text{and}\qquad
f_\measi\dfn\frac{\ddd\measi}{\ddd\mu}\in C^0([-1,1]^d;\R_+).
\end{equation}

For a continuous probability density $f:[-1,1]^d \to\C$, we denote
{$\hat f_0\dfn 1$} and
for
$\bsx\in [-1,1]^d$
\begin{equation}\label{eq:fk}
  \begin{aligned}
    \hat f_{k}(\bsx_{[k]})&\dfn \int_{[-1,1]^{d-k}}f(\bsx_{[k]},\bst) \dd\mu(\bst) &&\qquad k\in\{1,\dots,d\},\\
    f_{k}(\bsx_{[k]})&\dfn \frac{\hat f_{k}(\bsx_{[k]})}{\hat
      f_{k-1}(\bsx_{[k-1]})} &&\qquad k\in\{1,\dots,d\}.
  \end{aligned}
\end{equation}
{Thus $\hat f_{k}( \cdot )$ is the marginal density of $\bsx_{[k]}$ and} {$f_k(\bsx_{[k-1]},\cdot)$ is the marginal
density of $x_k$ conditioned on
$\bsx_{[k-1]}$.}
The corresponding {marginal conditional} %
CDFs
\begin{equation}\label{eq:Fk}
  \begin{aligned}
    F_{\measii;k}(\bsx_{[k]})&\dfn \int_{-1}^{x_k} f_{\measii;k}(\bsx_{[k-1]},t)\dd t,\\
    F_{\measi;k}(\bsx_{[k]})&\dfn \int_{-1}^{x_k}
    f_{\measi;k}(\bsx_{[k-1]},t)\dd t,
  \end{aligned}
\end{equation}
are well-defined for $\bsx\in [-1,1]^d$ and $k\in\{1,\dots,d\}$. They
are interpreted as functions of $x_k$ with $\bsx_{[k-1]}$ fixed; in
particular $F_{\measii;k}(\bsx_{[k-1]},\cdot)^{-1}$ denotes the
inverse of $x_k\mapsto F_{\measii;k}(\bsx_{[k]})$.

For $\bsx \in [-1,1]^d$, let
\begin{subequations}\label{eq:Tk}
  \begin{equation}\label{eq:T1}
    T_1(x_1)\dfn (F_{\measii;1})^{-1}\circ F_{\measi;1}(x_1),
  \end{equation}
  and inductively for $k\in\{2,\dots,d\}$ {with
    $T_{[k-1]}\dfn (T_j)_{j=1}^{k-1}:[-1,1]^{k-1}\to [-1,1]^{k-1}$}, let
  \begin{equation}%
    T_k(\bsx_{[k-1]},\cdot)\dfn F_{\measii;k}({T_{[k-1]}(\bsx_{[k-1]})},\cdot)^{-1}\circ F_{\measi;k}(\bsx_{[k-1]},\cdot).
  \end{equation}
\end{subequations}
Then
\begin{equation}\label{eq:knothe}
  T(\bsx)\dfn
  \begin{pmatrix}
    T_1(x_1)\\
    T_2(\bsx_{[2]})\\
    \vdots\\
    T_d(\bsx_{[d]})
  \end{pmatrix}
\end{equation}
yields the triangular KR transport
$T:[-1,1]^d\to [-1,1]^d$.  In the following we denote by
$dT:[-1,1]^d\to \R^{d\times d}$ the Jacobian matrix of $T$. The following theorem holds; see, e.g.,
\cite[Prop.~2.18]{santambrogio} for a proof.
\begin{theorem}\label{thm:knothed}
  Assume \eqref{eq:densities}. The KR transport $T$
  in \eqref{eq:knothe} satisfies $T_\sharp\measi=\measii$ and
  \begin{equation*}
    \det dT(\bsx)f_\measii(T(\bsx))=f_\measi(\bsx)\qquad\forall\bsx\in [-1,1]^d.
    \end{equation*}
\end{theorem}
{The regularity assumption \eqref{eq:densities} on the densities can be
relaxed in Thm.~\ref{thm:knothed}; see, e.g., \cite{bogachevtri}.}

In general, $T$ satisfying $T_\sharp\measi=\measii$ is not unique. To
keep the presentation succinct, henceforth we will simply refer to ``the
transport $T$,'' by which we always mean the unique
triangular KR transport in \eqref{eq:knothe}.

\section{Analyticity}\label{SEC:ANT}
The explicit formulas for $T$ given in Sec.~\ref{SEC:T} %
{imply that positive} analytic densities yield an analytic
transport. Analyzing the convergence of polynomial approximations to
$T$ requires knowledge of the domain of analyticity of $T$. {This is
investigated in the present section.}

\subsection{One dimensional case}\label{sec:onedanalyticty}
Let $d=1$. By \eqref{eq:T1},
$T:[-1,1]\to[-1,1]$ can be expressed through composition of the CDF of
$\measi$ and the inverse CDF of
$\measii$. As the inverse function theorem is usually stated without
giving details on the precise domain of extension of the inverse
function, we give a proof, {based on classical arguments,} in
Appendix \ref{sec:InvFunc}. This leads to the result in Lemma
\ref{lemma:ext}. Before stating it, we provide another short lemma
that will be used {multiple} times.

\begin{lemma}\label{lemma:lip}
  Let $\delta>0$ and let $K\subseteq\C$ be convex. Assume that $f\in
  C^1(\cB_\delta(K);\C)$ such that $\sup_{x\in \cB_\delta(K)}|f(x)|\le
  L$. Then $\sup_{x\in K}|f'(x)|\le \frac{L}{\delta}$ and
  $f:K\to\C$ is Lipschitz continuous with Lipschitz constant
  $\frac{L}{\delta}$.
\end{lemma}
\begin{proof}
  For any $x\in K$ and any $\eps\in
  (0,\delta)$, by Cauchy's integral formula
  \begin{equation*}
    |f'(x)|=\left|\frac{1}{2\measii\ii}\int_{|\zeta-x|=\delta-\eps} \frac{f(\zeta)}{(\zeta -x)^2}\dd\zeta \right| \le \frac{L}{\delta-\eps}.
  \end{equation*}
  Letting $\eps\to 0$ implies the claim.
\end{proof}

\begin{lemma}\label{lemma:ext}
  Let $\delta>0$, $x_0\in\C$ and let $f\in C^1(\cB_\delta (x_0);\C)$. Suppose that
  \begin{equation*}
    0< M\le |f(x)|\le L\qquad\forall x\in \cB_\delta(x_0).
  \end{equation*}
  Let $F:\cB_\delta (x_0)\to \C$ be an antiderivative of
  $f:\cB_\delta (x_0)\to\C$.

  With
  \begin{equation}\label{eq:rrs}
    \alpha=\alpha(M,L)\dfn \frac{M^2}{2M+4L}\qquad\text{and}\qquad
    \beta=\beta(M,L)\dfn \frac{\alpha}{M}=\frac{M}{2M+4L}
  \end{equation}
  there then exists a unique function
  $G:\cB_{\alpha\delta}(F(x_0))\to \cB_{\beta\delta}(x_0)$ such that
  $F(G(y))=y$ for all $y\in \cB_{\alpha\delta}(F(x_0))$. Moreover
  $G\in C^1(\cB_{\alpha\delta}(F(x_0));\C)$ with Lipschitz constant
  $\frac{1}{M}$.
\end{lemma}
\begin{proof}
  We verify the conditions of Prop.~\ref{prop:impl} with
  $\tilde\delta\dfn \delta/(1+\frac{2L}{M})<\delta$.
  To obtain a bound on the Lipschitz constant of $F'=f$ on
  $\cB_{\tilde\delta}(x_0)$, it suffices to bound $F''=f'$ there. %
  Due to $\tilde\delta+\frac{\tilde\delta 2L}{M} =\delta$,
  for all $x\in \cB_{\tilde\delta}(x_0)$ we have by Lemma
  \ref{lemma:lip}
  \begin{equation*}
    |f'(x)|\le \frac{L}{\frac{2\tilde \delta L}{M}}=
    \frac{M}{2\tilde\delta}\le\frac{|f(x_0)|}{2\tilde\delta} =
    \frac{|F'(x_0)|}{2\tilde\delta}.
  \end{equation*}
  Since $F'(x_0)=f(x_0)\neq 0$, the conditions of
  Prop.~\ref{prop:impl} are satisfied, and $G$ is well defined and
  exists on $\cB_{\alpha\delta}(F(x_0))$, where
  $\alpha\delta=\frac{\delta M^2}{2M+4L}=\frac{\tilde\delta M}{2}\le
  \frac{\tilde\delta
  |F'(t_0)|}{2}$. Finally, due to $1=F(G(y))'=F'(G(y))G'(y)$, it holds
  $G'(y)=\frac{1}{f(G(y))}$ for all $y\in \cB_{\alpha\delta}(F(x_0))$, which
  shows that $G:\cB_{\alpha\delta}(F(x_0))\to\C$ has Lipschitz
  constant $\frac{1}{M}$. Hence
  $G:\cB_{\alpha\delta}(F(x_0))\to
  \cB_{\alpha\delta/M}(G(F(x_0)))=\cB_{\beta\delta}(x_0)$.  Uniqueness
  of $G:\cB_{\alpha\delta}(F(x_0))\to\cB_{\beta\delta}(x_0)$
  satisfying $F\circ G={\rm Id}$ on $\cB_{\alpha\delta}(F(x_0))$
  follows by Prop.~\ref{prop:impl} and the fact that
  $\beta\delta=\frac{\tilde\delta}{2}\le\tilde\delta$.
\end{proof}

For $x\in[-1,1]$ and a density $f:[-1,1]\to\R$ the CDF equals
$F(x)=\int_{-1}^xf(t)\dd\mu( t)$. By definition of $\mu$
\begin{equation}\label{eq:ext}
  F(x)=\int_{-1}^xf(t)\frac{\ddd t}{2}=\frac{x+1}{2}\int_0^1 f(-1+t(x+1))\dd
  t.
\end{equation}
In case $f$ allows an extension $f:\cB_\delta([-1,1])\to\C$, then
$F:\cB_\delta([-1,1])\to\C$ is also well-defined via
\eqref{eq:ext}. Without explicitly mentioning it, we always consider
$F$ to be naturally extended to complex values in this sense.

{The next result generalizes Lemma \ref{lemma:ext} from complex
  balls $\cB_\delta$ to the pill-shaped domains $\cB_\delta([-1,1])$
  defined in \eqref{eq:cBdelta}. The proof is given in Appendix
  \ref{app:Fext}.}

\begin{lemma}\label{LEMMA:FEXT}
  Let $\delta>0$ and let
  \begin{enumerate}[label=(\alph*)]
  \item\label{item:Gass1} $f:[-1,1]\to\R_+$ be a probability density
    such that $f\in C^1(\cB_\delta([-1,1]);\C)$,
  \item\label{item:Gass2} $M\le |f(x)|\le L$ for some $0<M\le L<\infty$ and all
    $x\in \cB_\delta([-1,1])$.
  \end{enumerate}
  Set $F(x)\dfn \int_{-1}^x f(t)\dd\mu( t)$ and let
  $\alpha=\alpha(M,L)$, $\beta=\beta(M,L)$ be as in \eqref{eq:rrs}.

  Then
  \begin{enumerate}
  \item\label{item:Fiso} $F:[-1,1]\to [0,1]$ is a
    $C^1$-diffeomorphism, and $F\in C^1(\cB_\delta([-1,1]);\C)$ with
    Lipschitz constant $L$,
  \item\label{item:Finvhol} there exists a unique
    $G:\cB_{\alpha\delta}([0,1])\to \cB_{\beta\delta}([-1,1])$ such
    that $F(G(y))=y$ for all $y\in\cB_{\alpha\delta}([0,1])$ and
    \begin{equation}\label{eq:uniqueness}
      G:\cB_{\alpha\delta}(F(x_0))\to \cB_{\beta\delta}(x_0)
    \end{equation}
    for all $x_0\in [-1,1]$. Moreover
    $G\in C^1(\cB_{\alpha\delta}([0,1]);\C)$ with Lipschitz constant
    $\frac{1}{M}$.
  \end{enumerate}
\end{lemma}

We arrive at a statement about the domain of analytic extension of the
one dimensional monotone transport
$T\dfn F_\measii^{-1}\circ F_\measi:[-1,1]\to [-1,1]$ as in
\eqref{eq:T1}.

\begin{proposition}\label{prop:1d}
  Let $\delta_\measi$, $\delta_\measii>0$, {$L_\measi<\infty$,}
  $0<M_\measii\le {L_\measii<\infty}$ and
  \begin{enumerate}[label=(\alph*)]
  \item for $*\in\{\measi,\measii\}$ let $f_*:[-1,1]\to\R_+$ be a probability density and
    $f_*\in C^1(\cB_{\delta_*}([-1,1]);\C)$,
  \item for $x\in \cB_{\delta_\measi}([-1,1])$, $t\in \cB_{\delta_\measii}([-1,1])$
    \begin{equation*}
      |f_\measi(x)|\le L_\measi,\qquad 0< M_\measii\le |f_\measii(t)|\le L_\measii.
    \end{equation*}
  \end{enumerate}

  Then with $r\dfn \min\{\delta_\measi, \frac{\delta_\measii
    M_\measii^2}{L_\measi(2M_\measii+4L_\measii)}\}$ and $q\dfn
  \frac{r L_\measi}{M_\measii}$ it holds $T\in C^1(\cB_r([-1,1]);
  \cB_{q}([-1,1]))$.
\end{proposition}
\begin{proof}
  First, according to Lemma \ref{LEMMA:FEXT} \ref{item:Fiso},
  $F_\measi:[-1,1]\to [0,1]$ admits an extension
  \begin{subequations}\label{eq:FmeasiFmeasii-1}
  \begin{equation}
    F_\measi\in C^1(\cB_{\delta_\measi}([-1,1]); \cB_{L_\measi \delta_\measi}([0,1])),
  \end{equation}
  where we used that $F_\measi$ is Lipschitz continuous with Lipschitz
  constant $L_\measi$. Furthermore, Lemma \ref{LEMMA:FEXT}
  \ref{item:Finvhol} implies with
  $\eps\dfn \frac{\delta_\measii M_\measii^2}{2M_\measii+4L_\measii}$ that
  $F_\measii^{-1}:[0,1]\to [-1,1]$ admits an extension
  \begin{equation}
    F_\measii^{-1}\in C^1(\cB_{\eps}([0,1]);\cB_{\frac{\eps}{M_\measii}}([-1,1])),
  \end{equation}
  \end{subequations}
  where we used that $F_\measii^{-1}$ is Lipschitz continuous with
  Lipschitz constant $\frac{1}{M_\measii}$.

  {Assume first $r=\delta_\measi$, which implies}
  $L_\measi\delta_\measi\le\eps$. Then
  $F_\measii^{-1}\circ F_\measi\in
  C^1(\cB_{\delta_\measi}([-1,1]);\C)$ is well-defined {by
  \eqref{eq:FmeasiFmeasii-1}}. %
  In the second case where
  $r=\frac{\delta_\measii M_\measii^2}{L_\measi(2M_\measii+4L_\measii)}$, we have
  $\eps=L_\measi r$ and $r\le\delta_\measi$. Hence
  $F_\measi:\cB_r([-1,1])\to \cB_{L_\measi r}([-1,1])=\cB_{\eps}([-1,1])$ is
  well-defined. Thus $F_\measii^{-1}\circ F_\measi\in C^1(\cB_{r}([-1,1]);\C)$
  is well-defined. %
  In both cases, since $T=F_\measii^{-1}\circ F_\measi$ is Lipschitz
  continuous with Lipschitz constant $\frac{L_\measi}{M_\measii}$
  {(cp.~Lemma \ref{LEMMA:FEXT})}, $T$ maps to
  $\cB_{\frac{rL_\measi}{M_\measii}}([-1,1])$.
\end{proof}

{%
  The radius $r$ in Prop.~\ref{prop:1d} describing the analyticity
  domain of the transport behaves like
  $O(\min\{\delta_\measi,\delta_\measii\})$ as
  $\min\{\delta_\measi,\delta_\measii\}\to\infty$ (considering the
  $M_*$, $L_*$ constants fixed). %
  In this sense,
  the analyticity domain of $T$ is proportional to the minimum of the
  analyticity domains of the reference and target densities.}

\subsection{General case}
We now come to the main result of Sec.~\ref{SEC:ANT}, which is a
multidimensional version of Prop.~\ref{prop:1d}.
\proofref{The technical proof
along with several lemmata is deferred to Appendix \ref{app:thm:dinN}.}
More precisely, we give a
statement about the analyticity domain of
$(\partial_{k}T_k)_{k=1}^d$. The reason is that, from both a
theoretical and a practical viewpoint, it is convenient first to approximate
$\partial_{k}T_k:[-1,1]^k\to [0,1]$ and then to obtain an approximation
to $T_k$ by integrating over $x_k$. {We explain this in more detail
in Sec.~\ref{SEC:APPROX}, see \eqref{eq:tildeTkgeneric}.}

The following technical assumption gathers our requirements on the
reference $\measi$ and the target $\measii$.

\begin{assumption}\label{ass:finite}
  Let $0 < M \le L <\infty$, $C_6>0$, %
  $d\in\N$ and
  $\bsdelta\in (0,\infty)^d$. %
  For $*\in\{\measi,\measii\}$:
  \begin{enumerate}[label=(\alph*)]
  \item\label{item:cordinN:1}
    $f_*:[-1,1]^d\to\R_+$ is a probability density %
    and
    $f_{*}\in C^1(\cB_{\bsdelta}([-1,1]);\C)$,
  \item\label{item:cordinN:2} $M\le |f_{*}(\bsx)|\le L$ for
    $\bsx\in \cB_{\bsdelta}([-1,1])$,
  \item\label{item:cordinN:3}
    $\sup_{\bsy\in 
      \cB_{\bsdelta}}|f_{*}(\bsx+\bsy)-f_{*}(\bsx)| \le C_6$
    for $\bsx\in[-1,1]^d$,
  \item\label{item:cordinN:4} %
    $      \sup_{\bsy\in
        \cB_{\bsdelta_{[k]}}\times \{0\}^{d-k}}|f_{*}(\bsx+\bsy)-f_{*}(\bsx)|\le C_6 \delta_{k+1}
$
    for $\bsx\in[-1,1]^d$ and $k\in\{1,\dots,d-1\}$.
  \end{enumerate}
\end{assumption}

Assumptions \ref{item:cordinN:1} and \ref{item:cordinN:2} state that
$f_*$ is a {positive} analytic probability density on $[-1,1]^d$
that allows a complex differentiable extension to the set
$\cB_\bsdelta([-1,1])\subseteq\C^d$, cp.~\eqref{eq:cBdelta}.  Equation
\eqref{eq:Tk} shows that $T_{k+1}$ is obtained by a composition of
$F_{\measii;k+1}(T_{1},\dots,T_{k},\cdot)^{-1}$ (the inverse in the last
variable) and $F_{\measi;k+1}$. The smallness conditions
\ref{item:cordinN:3} and \ref{item:cordinN:4} can be interpreted as
follows: they will guarantee $F_{\measi;k+1}(\bsy)$ (for certain
complex $\bsy$) to belong to the domain where the complex extension of
$F_{\measii;k+1}(T_1,\dots,T_{k},\cdot)^{-1}$ is well-defined.

\begin{theorem}\label{THM:COR:DINN}
  Let $0<M\le L<\infty$, $d\in\N$ and $\bsdelta\in
  (0,\infty)^d$. There exist $C_6$, $C_7$, and $C_8>0$ depending
  on $M$ and $L$ (but not on $d$ or $\bsdelta$) such that if
  Assumption \ref{ass:finite} holds with $C_6$,
  then:

  Let $T=(T_k)_{k=1}^d$ be as in \eqref{eq:knothe}
  and $R_k\dfn \partial_{k}T_k$.
  With $\bszeta=(\zeta_k)_{k=1}^d$ where %
  \begin{equation}\label{eq:zetak2}
    \zeta_{k}\dfn C_7 \delta_{k} %
  \end{equation}
  it holds for all $k\in\{1,\dots,d\}$ that
  \begin{enumerate}
  \item\label{item:cordinN:a}
    $R_k\in C^1(\cB_{\bszeta_{[k]}}([-1,1]);\cB_{C_8}(1))$ and
    $\Re(R_k(\bsx))\ge \frac{1}{C_8}$ for all
    $\bsx\in \cB_{\bszeta_{[k]}}([-1,1])$,
  \item\label{item:cordinN:b} if $k\ge 2$,
    $R_k:\cB_{\bszeta_{[k-1]}}([-1,1])\times [-1,1]\to
    \cB_{\frac{C_8}{\max\{1,\delta_k\}}}(1)$.
  \end{enumerate}
\end{theorem}

{ Put simply, the first item of the theorem can be interpreted as
  follows: The function $\partial_{k}T_k$ allows in the $j$th variable
  an analytic extension to the set $\cB_{\zeta_j}([-1,1])$, where
  $\zeta_j$ is proportional %
  to $\delta_j$. The constant
  $\delta_j$ describes the domain of analytic extension of the
  densities $f_\measi$, $f_\measii$ in the $j$th variable. Thus the
  analyticity domain of each $\partial_{k} T_k$ is proportional to the
  (intersection of the) domains of analyticity of the
  densities. Additionally, the real part of $\partial_{k} T_k$ remains
  strictly positive on this extension to the complex domain. Note that
  $\partial_{k} T_k(\bsx)$ is necessarily positive for \emph{real}
  $\bsx\in [-1,1]^k$, since the transport is monotone.

  The second item of the theorem states that the $k$th variable $x_k$
  plays a special role for $T_k$: if we merely extend
  $\partial_{k}T_k$ in the first $k-1$ variables to the complex
  domain and let the $k$th argument $x_k$ belong to the real interval
  $[-1,1]$, then the values of this extension behave like
  $1+O(\frac{1}{\delta_k})$ and thus the extension becomes very close to the
  constant $1$ function as $\delta_k\to\infty$. In other words, if the
  densities $f_\measi$, $f_\measii$ allow a (uniformly bounded from
  above and below) analytic extension to a very large subset of the
  complex domain in the $k$th variable, then the $k$th component of
  the transport $T_k(\bsx_{[k]})$ will be close to
  $-1+\int_{-1}^{x_k}1\dd t=x_k$, i.e., to the identity in the $k$th
  variable.

  We also emphasize that we state the analyticity results here for
  $\partial_{k}T_k$ (in the form they will be needed below), but this immediately implies that $T_k$ allows an
  analytic extension to the same domain.  }

\begin{remark}
  Crucially, for any $k<d$ the left-hand side of the inequality in
  Assumption \ref{ass:finite} \ref{item:cordinN:4}
  depends on $(\delta_j)_{j=1}^k$, while the right-hand side %
  depends only on $\delta_{k+1}$ but not on $(\delta_j)_{j=1}^k$. This
  will allow us to suitably choose $\bsdelta$ when verifying this
  assumption (see the proof of Lemma
  \ref{LEMMA:ANALYTICITYASSUMPTION}).
\end{remark}

\begin{remark}
  The proof of Thm.~\ref{THM:COR:DINN} shows that there exists
  $C\in (0,1)$ independent of $M$ and $L$ such that we can choose
  $C_6=C\frac{\min\{M,1\}^5}{\max\{L,1\}^4}$,
  $C_7=C\frac{\min\{M,1\}^3}{\max\{L,1\}^3}$ and
  $C_8=C^{-1}(\frac{\max\{L,1\}^4}{\min\{M,1\}^4})$; see
  \eqref{eq:C6}, \eqref{eq:C7} and \eqref{eq:Rkext1}, \eqref{eq:C5}.
\end{remark}

To give an example for $\measi$, $\measii$ fitting our setting, we
show that Assumption \ref{ass:finite} holds (for \emph{some} sequence
$\bsdelta$) whenever the densities $f_\measi$, $f_\measii$ are
analytic.

\begin{lemma}\label{LEMMA:ANALYTICITYASSUMPTION}
  For $*\in\{\measi,\measii\}$ let $f_*:[-1,1]^d\to\R_+$ be a probability
  density, and assume that $f_*$ is analytic on an open
  set {in $\R^d$} containing $[-1,1]^d$.
  Then there exist $0<M\le L <\infty$ %
  and
  $\bsdelta\in (0,\infty)^d$ such that Assumption \ref{ass:finite}
  holds with $C_6(M,L)$ as in Thm.~\ref{THM:COR:DINN}.
\end{lemma}

\section{Polynomial based approximation}\label{SEC:APPROX}
Analytic functions on $[-1,1]^d\to\R$ allow for exponential
convergence when approximated by multivariate polynomial expansions.
We recall this in Sec.~\ref{sec:legendre} for truncated Legendre
expansions. These results are then applied to the KR
transport in Sec.~\ref{sec:pold}.

{
\subsection{Exponential convergence of Legendre
  expansions}\label{sec:legendre}
For $n\in\N_0$, let $L_n\in\bbP_n$ be the $n$th Legendre polynomial
normalized in $L^2([-1,1],\mu)$. For $\bsnu\in\N_0^d$ set
$L_\bsnu(\bsy)\dfn \prod_{j=1}^dL_{\nu_j}(y_j)$.
Then $(L_\bsnu)_{\bsnu\in\N_0^d}$ is an orthonormal basis of
$L^2([-1,1]^d,\mu)$. Thus, every $f\in L^2([-1,1]^d,\mu)$ admits the
multivariate Legendre expansion
$f(\bsy) = \sum_{\bsnu\in\N_0^d}l_\bsnu L_\bsnu(\bsy)$ with
coefficients
$l_\bsnu = \int_{[-1,1]^d}f(\bsy)L_\bsnu(\bsy)\dd\mu(\bsy)$. For
finite $\Lambda\subseteq \N_0^d$,
$\sum_{\bsnu\in\Lambda}l_\bsnu L_\bsnu(\bsy)$ is the orthogonal
projection of $f$ in the Hilbert space $L^2([-1,1]^d,\mu)$
onto its subspace
\begin{equation*}
  \bbP_\Lambda\dfn{\rm span}\set{\bsx^\bsnu}{\bsnu\in\Lambda}.
\end{equation*}

As is well-known, functions that are holomorphic on sets of the type
$\cB_\bsdelta([-1,1])$ have exponentially decaying Legendre
coefficients. We recall this in the next lemma, which is adapted to
the regularity we showed for the transport in Thm.~\ref{THM:COR:DINN}.

\begin{lemma}\label{LEMMA:LEGEST}
  Let $k\in\N$, $\bsdelta\in (0,\infty)^k$ and
  $f\in C^1(\cB_\bsdelta([-1,1]);\C)$. Set
  $w_\bsnu\dfn \prod_{j=1}^k(1+2\nu_j)^{3/2}$, $\bsvarrho=(1+\delta_j)_{j=1}^k$
  and
  $l_\bsnu\dfn \int_{[-1,1]^d} f(\bsy) L_\bsnu(\bsy)\dd\mu(\bsy)$
  for $\bsnu\in\N_0^k$.
  Then
  \begin{enumerate}
  \item\label{item:lemma:legest:1} for all $\bsnu\in\N_0^k$
    \begin{equation}\label{eq:lknubound}
    |l_\bsnu|
    \le \bsvarrho^{-\bsnu} w_\bsnu \norm[L^\infty(\cB_\bsdelta({[-1,1]}))]{f}
    \prod_{j\in\supp\bsnu}\frac{2\varrho_j}{\varrho_j-1},
  \end{equation}
\item\label{item:lemma:legest:2} for all $\bsnu\in\N_0^{k-1}\times\{0\}$
    \begin{equation}\label{eq:lknubound2}
    |l_\bsnu|
    \le \bsvarrho^{-\bsnu} w_\bsnu \norm[{L^\infty(\cB_{\bsdelta_{[k-1]}}([-1,1])\times [-1,1])}]{f}
    \prod_{j\in\supp\bsnu}\frac{2\varrho_j}{\varrho_j-1}.
  \end{equation}  
  \end{enumerate}
\end{lemma}

The previous lemma in combination with Thm.~\ref{THM:COR:DINN} yields
a bound on the Legendre coefficients
of the partial derivatives $\partial_k T_k-1$ of the $k$th component
of the transport. Specifically, for $\bsnu\in\N_0^k$
Thm.~\ref{THM:COR:DINN} \ref{item:cordinN:a} together with Lemma
\ref{LEMMA:LEGEST} \ref{item:lemma:legest:1} implies with
$t_1\dfn
\bsvarrho^{-\bsnu}\norm[L^\infty(\cB_\bsdelta({[-1,1]}))]{\partial_kT_k-1}$
and
$t_2\dfn w_\bsnu\prod_{j\in\supp\bsnu}\frac{\varrho_j}{\varrho_j-1}$
the bound $t_1t_2$ for the corresponding Legendre coefficient.  For a
multi-index $\bsnu\in\N_0^{k-1}\times\{0\}$, applying instead
Thm.~\ref{THM:COR:DINN} \ref{item:cordinN:b} together with Lemma
\ref{LEMMA:LEGEST} \ref{item:lemma:legest:2} yields the bound
$\tilde t_1 t_2$ where
$\tilde t_1\dfn
\bsvarrho^{-\bsnu}\norm[{L^\infty(\cB_{\bsdelta_{[k-1]}}({[-1,1]})\times[-1,1])}]{\partial_kT_k-1}$.
By Thm.~\ref{THM:COR:DINN}, the last norm is bounded by
$\frac{C_3}{\delta_k}$.
Hence, %
compared to the first estimate $t_1$, we gain the factor
$\frac{1}{\delta_k}$ by using the second estimate $\tilde t_1$
instead.\footnote{It is desirable to have sharp upper bounds
  on the Legendre coefficients in order to construct the most
  efficient ansatz spaces. However, solely using the first bound $t_1$
  in the analysis would not alter the type of exponential convergence
  shown in Prop.~\ref{PROP:KAPPROX}. Using the improved bound
  $\tilde t_1$ becomes crucial for the analysis of the
  high-dimensional case $d\gg 1$, which we discuss in \cite{zm2}.}

Taking the minimum of these estimates leads us to introduce for
$k\in\N$, $\bsnu\in\N_0^k$ and $\bsvarrho\in (1,\infty)^k$ the
quantity
\begin{equation}\label{eq:eta}
  \gamma(\bsvarrho,\bsnu)\dfn \varrho_k^{-\max\{1,\nu_k\}}\prod_{j=1}^{k-1}\varrho_j^{-\nu_j},
\end{equation}
and the set
\begin{equation}\label{eq:Leps}
  \Lambda_{k,\eps}\dfn \set{\bsnu\in\N_0^k}{\gamma(\bsvarrho,\bsnu)\ge\eps},
\end{equation}
corresponding to the largest values of $\gamma(\bsvarrho,\bsnu)$.

{ The structure of $\Lambda_{k,\eps}$ %
  is as follows: The larger $\varrho_j$, the smaller
  $\varrho_j^{-1}$. Thus the larger $\varrho_j$, the fewer
  multi-indices $\bsnu$ with $j\in\supp\bsnu$ belong to
  $\Lambda_{k,\eps}$. In this sense $\varrho_j$ measures the
  importance of the $j$th variable in the Legendre expansion of
  $\partial_kT_k$. The $k$th variable plays a special role, however:
  it is always among the most important variables, since for all
  $\bsnu\in\N_0^k$ it holds that
  $\gamma(\bsvarrho, \bsnu) \ge \gamma(\bsvarrho, \bse_k)$, where
  $\bse_k=(0,\dots,0,1)\in\N_0^k$. In other words, whenever $\eps>0$
  is so small that $\Lambda_{k,\eps}\neq\emptyset$, at least one
  $\bsnu$ with $\nu_k\neq 0$ belongs to $\Lambda_{k,\eps}$.  }

Having determined a set of multi-indices corresponding to the largest
upper bounds obtained for the Legendre coefficients, we arrive at the
next proposition. %
The assumptions on the function $f$ correspond to the regularity of
$\partial_{k}T_k$
shown in Thm.~\ref{THM:COR:DINN}. The proposition
states that such $f$ can be approximated with the error decreasing
as $O(-\beta N^{1/k})$ in terms of the dimension $N$ of the
polynomial space.

\begin{proposition}\label{PROP:KAPPROX}
  Let $k\in\N$, $\bsdelta\in (0,\infty)^k$ and $r>0$, such that
  $f\in C^1(\cB_{\bsdelta}([-1,1]);\cB_r)$ and
  $f:\cB_{\bsdelta_{[k-1]}}([-1,1])\times[-1,1]\to
  \cB_{\frac{r}{1+\delta_k}}$. With $\varrho_j\dfn 1+\delta_j$ set
  \begin{equation}\label{eq:beta0}
    \beta\dfn \left(k!\prod_{j=1}^k\log(\varrho_j)\right)^{\frac 1 k}.
  \end{equation}
  For $\bsnu\in\N_0^k$ set
  $l_\bsnu\dfn \int_{[-1,1]^k}f(\bsy)L_\bsnu(\bsy)\dd\mu(\bsy)$.
  
  Then for every $\tilde\beta<\beta$, %
  there exists $C=C(k,m,\tilde\beta,\bsdelta,r,\norm[L^\infty({\cB_\bsdelta([-1,1])})]{f})$ s.t.\ for
  every $\eps\in (0,\varrho_k^{-1})$ {the following} holds with $\Lambda_{k,\eps}$ as
  in \eqref{eq:Leps}:
  \begin{equation}\label{eq:kapprox}
    \normc[W^{m,\infty}({[-1,1]^k})]{f(\cdot)-\sum_{\bsnu\in\Lambda_{k,\eps}}l_\bsnu L_\bsnu(\cdot)}
      \le C \eps^{\frac{\tilde\beta}{\beta}}
      \le C \exp\left(-\tilde\beta |\Lambda_{k,\eps}|^{\frac 1 k}\right).
  \end{equation}
\end{proposition}

{In Prop.~\ref{PROP:KAPPROX}, $\tilde \beta$ can
  be chosen arbitrarily close to $\beta$.  However as $\tilde\beta$
  approaches $\beta$, the constant $C$ in \eqref{eq:polerr} will tend
  to $\infty$, cp.~\eqref{eq:Wminftyeps}.}

}%
\subsection{Polynomial and rational approximation}\label{sec:pold}

Combining Prop.~\ref{PROP:KAPPROX} with Thm.~\ref{THM:COR:DINN} we
obtain the following approximation result for the transport. It states
that $T:[-1,1]^d\to [-1,1]^d$ can be approximated by multivariate
polynomials, converging in $W^{m,\infty}([-1,1]^d)$ with the error
decreasing as $\exp(-\beta N_\eps^{1/d})$. Here $N_\eps$ is the
dimension of the (ansatz) space in which $T$ is approximated.

\begin{theorem}\label{thm:TdN}
  Let $m\in\N_0$. Let $f_\measi$, $f_\measii$ satisfy Assumption
  \ref{ass:finite} for some constants $0<M\le L<\infty$,
  $\bsdelta\in (0,\infty)^d$ and with $C_6=C_6(M,L)$ as in
  Thm.~\ref{THM:COR:DINN}. Let $C_7=C_7(M,L)$ be as in
  Thm.~\ref{THM:COR:DINN}. For $j\in\{1,\dots,d\}$
  set
  \begin{subequations}\label{eq:beta}
  \begin{equation}\label{eq:xidef}
    {\varrho_j}\dfn %
    1+C_7\delta_j.
  \end{equation}
    For $k\in\{1,\dots,d\}$ let
    $\Lambda_{k,\eps}$ be as in \eqref{eq:Leps} and define
  \begin{equation}\label{eq:betad}
    \beta\dfn\left((d-1)!\prod_{j=1}^d\log(\varrho_j) \right)^{\frac{1}{d}}.
  \end{equation}
  \end{subequations}  
  {For every $\tilde\beta<\beta$}
  there exists
  $C=C(\bsvarrho,m,d,\tilde\beta,f_\measi,f_\measii)>0$
  {such that for every $\eps\in (0,1)$ with}
  \begin{equation*}
    \tilde T_{k}\dfn\sum_{\bsnu\in\Lambda_{k,\eps}} l_{k,\bsnu} L_\bsnu
    \in\bbP_{\Lambda_{k,\eps}}\qquad\text{where}\qquad
    l_{k,\bsnu}\dfn \int_{[-1,1]^k} T_k(\bsy)L_\bsnu(\bsy)\dd\mu(\bsy),
  \end{equation*}
  {$\tilde T\dfn (\tilde T_k)_{k=1}^d$}
  and $N_\eps\dfn \sum_{k=1}^d|\Lambda_{k,\eps}|$, it holds 
  \begin{equation}\label{eq:polerr}
    \norm[{W^{m,\infty}([-1,1]^k)}]{{T- \tilde T}}\le C \exp(-{\tilde\beta} N_\eps^{1/d}).
  \end{equation}
\end{theorem}

\begin{remark}
  We set $\varrho_j=1+C_7\delta_j>1$ in Thm.~\ref{thm:TdN}, %
  where $\delta_j$ as in Assumption \ref{ass:finite} encodes the size
  of the analyticity domain of the densities $f_\measi$ and
  $f_\measii$ (in the $j$th variable). The constant $\beta$ in
  \eqref{eq:betad} is an increasing function of each $\varrho_j$.
  Loosely speaking, Thm.~\ref{thm:TdN} states that the larger the
  analyticity domain of the densities, the {faster the convergence}
  when approximating the corresponding transport $T$ with polynomials.
\end{remark}

We skip the proof of the above theorem and instead proceed with a
variation of this result. It states a convergence rate for an
approximation $\tilde T_{k}$ to $T_k$, which enjoys the property that
$\tilde T_{k}({\bsx_{[k-1]}},\cdot):[-1,1]\to [-1,1]$ is
monotonically increasing and bijective for every
${\bsx_{[k-1]}}\in [-1,1]^{k-1}$. Thus, contrary to $\tilde T$ in
Thm.~\ref{thm:TdN}, the $\tilde T$ in the next proposition is a bijection
from $[-1,1]^d\to [-1,1]^d$ by construction.

This is achieved as follows: Let $g:\R\to \set{x\in\R}{x\ge 0}$ be
analytic, such that $g(0)=1$ and
$h\dfn g^{-1}:(0,\infty)\to\R$ is also
analytic. We first approximate $h(\partial_{k}T_k)$ by some function
$p_k$ and then obtain
$-1+\int_{-1}^{x_k}g(p_k(\bsx_{[k-1]},t))\dd t$
as an approximation $\tilde T_k$ to $T_k$.  {This approach,
  similar to what is proposed in \cite{MR1616049}, and in the present context
  in \cite{MR3821485},} guarantees
$\partial_{k}\tilde T_k=g(p_k(\bsx_{[k]}))\ge 0$ and
$\tilde T_k(\bsx_{[k-1]},-1){=-1}$ so that $\tilde T_k{(\bsx)}$
is monotonically increasing in
$x_k$. %
In order to force $\tilde T_k(\bsx_{[k-1]},\cdot):[-1,1]\to [-1,1]$ to
be bijective we introduce a normalization which {leads} to
\begin{equation}\label{eq:tildeTkgeneric}
  \tilde T_k(\bsx)=-1+2\frac{\int_{-1}^{x_k}g(p_k(\bsx_{[k-1]},t))\dd t}{\int_{-1}^{1}g(p_k(\bsx_{[k-1]},t))\dd t}.
\end{equation}
The meaning of
$g(0)=1$ is that the trivial approximation $p_k\equiv 0$ then yields
$\tilde T_k(\bsx)=x_k$.

To avoid further technicalities, henceforth we choose
$g(x)=(x+1)^2$ (and thus $h(x)=\sqrt{x}-1$), but emphasize that our
analysis works just as well with any other positive analytic function
such that $g(0)=1$, e.g., $g(x)=\exp(x)$ and $h(x)=\log(x)$. The choice
$g(x)=(x+1)^2$ has the advantage that $g(p_k)$ is polynomial if $p_k$
is polynomial. This allows exact evaluation of the integrals in
\eqref{eq:tildeTkgeneric} without resorting to numerical quadrature,
and results in a \emph{rational} approximation $\tilde T_k$:

\proofref{The proof of Thm.~\ref{THM:POLYD} is given in Appendix
  \ref{app:polyd}.}

\begin{theorem}\label{THM:POLYD}
  Let $m\in\N_0$.  Let $f_\measi$, $f_\measii$ satisfy Assumption
  \ref{ass:finite} for some constants $0<M\le L<\infty$,
  $\bsdelta\in (0,\infty)^d$ and with $C_6=C_6(M,L)$ as in
  Thm.~\ref{THM:COR:DINN}.
  Let %
  $\varrho_j$, $\beta$ and $\Lambda_{k,\eps}$ be as in
  Thm.~\ref{thm:TdN}.

  {For every $\tilde\beta<\beta$} there exists
  $C=C(\bsxi,m,d,{\tilde\beta},f_\measi,f_\measii)>0$ and for every
  $\eps\in (0,1)$ there exist polynomials
  $p_{k,\eps}\in\bbP_{\Lambda_{k,\eps}}$, $k\in\{1,\dots,d\}$, such that with
  \begin{equation}\label{eq:ck}
    \begin{aligned}
    \tilde T_{k,\eps}(\bsx)&\dfn -1+\frac{2}{c_{k,\eps}(\bsx_{[k-1]})}\int_{-1}^{x_k}(1+p_{k,\eps}(\bsx_{[k-1]},t))^2 \dd t,\\
    c_{k,\eps}(\bsx_{[k-1]})&\dfn\int_{-1}^{1}(1+p_{k,\eps}(\bsx_{[k-1]},t))^2 \dd t,
    \end{aligned}
  \end{equation}
  the map $\tilde T_\eps \dfn (\tilde T_{k,\eps})_{k=1}^d:[-1,1]^d\to [-1,1]^d$ %
  is a monotone triangular bijection, and with
  \begin{equation}\label{eq:Neps}
    N_\eps\dfn \sum_{k=1}^d |\Lambda_{k,\eps}|
  \end{equation}
    it holds
  \begin{equation}\label{eq:errorrational}
    \norm[{W^{m,\infty}([-1,1]^k)}]{{T-\tilde T_\eps}}\le C \exp(-{\tilde\beta} N_\eps^{1/d}).
  \end{equation}
\end{theorem}

{We emphasize that our reason for using rational functions rather
  than polynomials in Thm.~\ref{THM:POLYD} is merely to guarantee that
  the resulting approximation $\tilde T:[-1,1]^d\to [-1,1]^d$ is a
  bijective and monotone map. We do \emph{not} employ specific
  properties of rational functions (as done for Pad\'e approximations)
  in order to improve the convergence order.}

\begin{remark}
  If $\Lambda_{k,\eps}=\emptyset$ then by convention
  $\bbP_{\Lambda_{k,\eps}}=\{0\}$; thus $p_{k,\eps}= 0$ and
  $\tilde T_{k,\eps}(\bsx)=x_k$.
\end{remark}

\begin{remark}\label{rmk:S}
  Let $S=T^{-1}$ so that $T_\sharp\measi=\measii$ is equivalent to
  $S^\sharp\measi=\measii$. It is often %
  simpler to first
  approximate $S$, and then compute $T$ by inverting $S$,
  {see \cite{MR3821485}}. Since the
  assumptions of Thm.~\ref{THM:POLYD} (and Thm.~\ref{THM:COR:DINN}) on
  the measures $\measi$ and $\measii$ are identical,
  Thm.~\ref{THM:POLYD} also yields an approximation result for the
  \emph{inverse transport map} $S$: {for all $\eps>0$} and with
  $\Lambda_{k,\eps}$ as in Thm.~\ref{THM:POLYD} there exist
  multivariate polynomials $p_k\in\bbP_{\Lambda_{k,\eps}}$ such that
  with
  \begin{equation}\label{eq:tSk}
    \begin{aligned}
    \tilde S_k(\bsx)&\dfn -1+\frac{2}{c_k(\bsx_{[k-1]})}\int_{-1}^{x_k}(1+p_k(\bsx_{[k-1]},t))^2 \dd t,\\
    c_k(\bsx_{[k-1]})&\dfn\int_{-1}^{1}(1+p_k(\bsx_{[k-1]},t))^2 \dd t,
    \end{aligned}
  \end{equation}
  it holds
  \begin{equation*}
    \norm[{W^{m,\infty}([-1,1]^d)}]{{S-\tilde S}}\le C \exp(-{\tilde\beta} N_\eps^{1/d}).
  \end{equation*}
\end{remark}

\section{{Deep neural network approximation}}\label{sec:relu}
Based on the seminal paper \cite{yarotsky}, it has recently been
observed that ReLU neural networks (NNs) are capable of approximating
analytic functions at an exponential convergence rate
\cite{MR3856963,OSZ19}, and slight improvements can be shown for
certain smoother activation functions, e.g., \cite{CiCP-27-379}. We
also refer to \cite{MR1230251} for much earlier results of this type
for different activation functions. {As a consequence}, our
analysis in Sec.~\ref{SEC:ANT} yields approximation results of the
transport by deep neural networks (DNNs). Below we present the
statement, which is based on \cite[Thm.~3.7]{OSZ19}.  \proofref{For
  the proof see Appendix \ref{app:relu}.}

To formulate the result, we recall the definition of a feedforward
ReLU NN. The (nonlinear) ReLU activation function is defined as
$\varphi(x)\dfn \max\{0,x\}$. We call a function $f:\R^d\to\R^d$ a
ReLU NN, if it can be written as
\begin{equation}\label{eq:nn}
  f(\bsx) = \bsW_{L} \varphi\Big(\bsW_{L-1}\varphi\Big(\cdots \varphi\Big(\bsW_0\bsx+\bsb_0 \Big)\Big)+\bsb_{L-1} \Big) + \bsb_{L},
\end{equation}
for certain \emph{weight matrices} $\bsW_j\in\R^{n_{j+1}\times n_j}$
and \emph{bias vectors} $\bsb_j\in\R^{n_{j+1}}$ where $n_0=n_{L+1}=d$.
For simplicity, we do not distinguish between the network (described
by $(\bsW_j,\bsb_j)_{j=0}^L$), and the function it expresses
(different networks can have the same output).  We then write
${\rm depth}(f)\dfn L$, ${\rm width}(f)\dfn \max_j n_j$ and
${\rm size}(f)\dfn \sum_{j=0}^{L+1}(|\bsW_j|_0+|\bsb_j|_0)$, where
{$|\bsW_j|_0=|\set{(k,l)}{(\bsW_j)_{kl}\neq 0}|$ and
  $|\bsb_j|_0=|\set{k}{(\bsb_j)_{k}\neq 0}|$.}  In other words, the
depth corresponds to the number of applications of the activation
function (the number of hidden layers) and the size equals the number
of nonzero weights and biases, {i.e., the number of trainable
  parameters in the network.}

\begin{theorem}\label{THM:RELU}
  Let $f_\measi$, $f_\measii$ be two positive and analytic probability
  densities on $[-1,1]^d$. Then there exists $\beta>0$ and for every
  $N\in\N$ there exists a ReLU NN
  $\Phi_N=(\Phi_{N,j})_{j=1}^d:\R^d\to \R^d$, such that
  $\Phi_N:[-1,1]^d\to [-1,1]^d$ is bijective, triangular and monotone,
  \begin{equation}\label{eq:thm:relu}
    \norm[{W^{1,\infty}([-1,1]^d)}]{T-\Phi_N}\le C \exp(-\beta N^{\frac{1}{d+1}}),
  \end{equation}
  ${\rm size}(\Phi_N) \le C N$ %
  and ${\rm depth}(\Phi_N)\le C\log(N)N^{1/2}$. %
  {Here $C$ is a constant} depending on $d$, $f_\measi$ and
  $f_\measii$ but independent of $N$.
\end{theorem}

{
\begin{remark}
  Compared to Thm.~\ref{thm:TdN} and Thm.~\ref{THM:POLYD}, for ReLU
  networks we obtain the slightly worse convergence rate
  $\exp(-\beta N^{1/(d+1)})$ instead of $\exp(-\beta N^{1/d})$. This
  stems from the fact that, for ReLU networks, the best known
  approximation results of analytic functions in $d$ dimensions
  converge at rate
  $\exp(-\beta N^{1/(d+1)})$; see \cite[Thm.~3.5]{OSZ19}.
\end{remark}
}

The proof of Thm.~\ref{THM:RELU} proceeds as follows: First, we apply
results from \cite{OSZ19} to obtain a neural network approximation
$\tilde \Phi_k$ to $T_k$. The constructed
$(\tilde \Phi_k)_{k=1}^d:[-1,1]^d\to\R^d$ is a triangular map that is
close to $T$ in the norm of $W^{1,\infty}([-1,1]^d)$. However, it is
not necessarily a monotone bijective {self-mapping of}
$[-1,1]^d$. {To correct the construction, we  use the following
  lemma:

\begin{lemma}\label{LEMMA:RELU}
  Let $f:[-1,1]^{k-1}\to\R$ be a ReLU NN. Then there exists a ReLU NN
  $g_f:[-1,1]^{k}\to\R$ such that $|g_f(\bsy,t)|\le |f(\bsy)|$ for all
  $(\bsy,t)\in[-1,1]^{k-1}\times [-1,1]$,
  \begin{equation*}
    g_f(\bsy,t)=\begin{cases}
      f(\bsy) &t=1\\
      0 &t=-1,
    \end{cases}
  \end{equation*}
  and ${\rm depth}(g_f)\le 1+{\rm depth}(f)$ and
  ${\rm size}(g_f)\le C (1+{\rm size}(f))$ with $C$ independent of $f$
  and $g_f$. Moreover, in the sense of weak derivatives
  $|\nabla_\bsy g_f(\bsy,t)|\le |\nabla f(\bsy)|$ and
  $|\frac{d}{dt} g_f(\bsy,t)|\le \esssup_{\bsy\in[-1,1]^{k-1}}|f(\bsy)|$,
  i.e., these inequalities hold a.e.\ in $[-1,1]^{k-1}\times [-1,1]$.
\end{lemma}

With $\tilde \Phi_k:[-1,1]^k\to\R$ approximating the $k$th component
$T_k:[-1,1]^k\to [-1,1]$, it is then easy to check that with
$f_1(\bsx_{[k-1]})\dfn 1-\tilde\Phi_k(\bsx_{[k-1]})$ and $f_{-1}(\bsx_{[k-1]})\dfn -1-\tilde\Phi_k(\bsx_{[k-1]},1)$ for $\bsx\in[-1,1]$, the NN
\begin{equation*}
  \Phi_k(\bsx)\dfn \tilde \Phi_k(\bsx)+g_{f_1}(\bsx_{[k-1]},x_k)+g_{f_{-1}}(\bsx_{[k-1]},-x_k)
\end{equation*}
satisfies $\Phi_k(\bsx_{[k-1]},1)=1$ and $\Phi_k(\bsx_{[k-1]},-1)=-1$.
Since the introduced correction terms $g_{f_1}$ and $g_{f_{-1}}$ have
size and depth bounds of the same order as $\tilde \Phi_k$, they will
not worsen the resulting convergence rates. The details are provided
in Appendix \ref{app:relu}.}

In the previous theorem we consider a ``sparsely-connected'' network
$\Phi$, meaning that certain weights and biases {are, by choice of
  the network architecture, set \textit{a priori} to zero}. This reduces the
overall size of $\Phi$.  We note that this also yields a
convergence rate for \emph{fully connected} networks: Consider all
networks of width $O(N)$ and depth $O(\log(N) N^{1/2})$. %
The size of a network within this architecture is bounded by
$O(N^2\log(N) N^{1/2})=O(N^{5/2}\log(N))$, since the number of
elements of the weight matrix $\bsW_j$ between two consecutive layers
is $n_jn_{j+1}\le N^2$. The network $\Phi$ from Thm.~\ref{THM:RELU}
belongs to this class, and thus the best approximation among networks
with this architecture achieves {at least} the exponential convergence
$\exp(-\beta N^{1/(d+1)})$. In terms of the number of trainable
parameters $M=O(N^{5/2}\log(N))$, this convergence rate is, up to
logarithmic terms, %
$\exp(-\beta M^{2/(5d+5)})$.

{
\begin{remark}
  The constant $C$ in \eqref{eq:errorrational} and the (possibly
  different) constant $C$ in \eqref{eq:thm:relu} typically depend
  exponentially on the dimension $d$: This dependence is true for
  polynomial approximation results of analytic functions in $d$
  dimensions, which is why it will hold for $C$ in
  \eqref{eq:errorrational} in general.  Since the proof in
  \cite{OSZ19}, upon which our analysis is based, uses polynomial
  approximations, the same can be expected for the constant in
  \eqref{eq:thm:relu}. In \cite{zm2} we will discuss the
  approximation of $T$ by rational functions in the high-dimensional
  case. There we give sufficient assumptions on the reference and
  target to guarantee algebraic convergence of the error, with all
  constants being controlled independent of the dimension.
\end{remark}}

{\begin{remark}
Normalizing flows approximate a transport map $T$ using a variety of neural network constructions, typically by composing a series of nonlinear bijective transformations; each individual transformation employs neural networks in its parameterization, embedded into a specific functional form (possibly augmented with constraints) that ensures bijectivity \cite{kobyzev2019normalizing,papamakarios2019normalizing}.
``Residual'' normalizing flows \cite{pmlr-v37-rezende15,berg2018sylvester} compose maps that are not in general triangular, but ``autoregressive'' flows \cite{pmlr-v80-huang18d,jaini2019sum,wehenkel2019unconstrained} use monotone triangular maps as their essential building block. 
Many practical implementations of autoregressive flows, however, limit
the class of triangular maps that can be expressed. Thus, they cannot
seek to directly approximate the KR transport in a
single step; rather, they compose multiple such triangular maps,
interleaved with permutations \cite{2006.11469}. In principle, though,
a direct approximation of the KR map is sufficient,
and our results {could be a starting point for}
constructive and quantitative guidance on
the parameterization and expressivity of autoregressive flows in this setting. 
Our result is also close in style to \cite{2004.08867}, which shows low order convergence rates for neural network approximations of transport maps for certain classes of target densities, by writing the map as a gradient of a potential function given by a neural network. This construction, which employs semi-discrete optimal transport, is not in general triangular and does not necessarily coincide with common normalizing flow architectures.
\end{remark}}

\section{{Convergence of pushforward measures}}\label{SEC:MEASURES}
{Let again $T_\sharp\measi=\measii$.  In 
  Sec.~\ref{SEC:APPROX} we have shown that the approximation
  $\tilde T$ to $T$ obtained in Thm.~\ref{THM:POLYD} and
  Thm.~\ref{THM:RELU} converges to $T$ in the
  $W^{m,\infty}([-1,1]^d)$ norm for suitable $m\in\N$. In the present
  section, we show that these results imply corresponding error bounds
  for the approximate pushforward measure, i.e., bounds for
  \begin{equation}\label{eq:distpushfwd}
    {\rm dist}(\measii,\tilde T_\sharp\measi) =
    {\rm dist}(T_\sharp\measi,\tilde T_\sharp\measi)
  \end{equation}
  with ``${\rm dist}$'' referring to the Hellinger distance, the total
  variation distance, the Wasserstein distance or the KL divergence.
  Specifically, we will see that smallness of
  $\norm[W^{1,\infty}]{T-\tilde T}$ (or
  $\norm[L^{\infty}]{T-\tilde T}$ in case of the Wasserstein distance)
  implies smallness of \eqref{eq:distpushfwd}.  }

{As mentioned before, when casting the approximation of the transport
as an optimization problem, it is often more convenient to first
approximate the inverse transport $S=T^{-1}$ by some $\tilde S$, and then to invert $\tilde S$ to
obtain an approximation $\tilde T = \tilde S^{-1}$ of $T$; see \cite{MR3821485} and also, e.g.,
the method in \cite{MR4065222}. In this case we usually have an upper
bound on $\norm[]{S-\tilde S}$ rather than $\norm{T-\tilde T}$ in a
suitable norm; cp.~Rmk.~\ref{rmk:S}.
However, a bound of the type
$\norm[W^{m,\infty}]{T-\tilde T}<\varepsilon$ implies
$\norm[W^{m,\infty}]{S-\tilde S}=O(\varepsilon)$ for $m\in\{0,1\}$ as
the next lemma shows. Since closeness in $L^\infty$ or $W^{1,\infty}$
is all we require for the results of this section, the following
analysis covers either situation.

\begin{lemma}\label{LEMMA:WINFTY}
  Let $T:[-1,1]^d\to [-1,1]^d$ and $\tilde T:[-1,1]^d\to [-1,1]^d$ be
  bijective. Denote $S=T^{-1}$ and $\tilde S=\tilde T^{-1}$.
  Suppose that $S$ has Lipschitz constant $L_S$.
  Then
  \begin{enumerate}
  \item it holds
    \begin{equation*}
      \norm[{L^\infty([-1,1]^d)}]{S-\tilde S}\le L_S \norm[{L^\infty([-1,1]^d)}]{T-\tilde T},
    \end{equation*}
  \item if $S$, $T$, $\tilde S$, $\tilde T\in W^{1,\infty}([-1,1]^d)$
    and $dT:[-1,1]^d\to\R^{d\times d}$ has Lipschitz constant $L_{dT}$,
    then
    \begin{equation*}
      \norm[{L^\infty([-1,1]^d)}]{dS-d\tilde S}\le (1+L_SL_{dT})
      \norm[{L^\infty([-1,1]^d)}]{dS}
      \norm[{L^\infty([-1,1]^d)}]{d\tilde S}
      \norm[{W^{1,\infty}([-1,1]^d)}]{T-\tilde T}.
    \end{equation*}
  \end{enumerate}
\end{lemma}
}%

\subsection{Distances}\label{sec:distances}
{For two probability measures $\measi\ll\mu$ and $\measii\ll\mu$ on
  $[-1,1]^d$ equipped with the Borel sigma-algebra, we consider the
  following distances:}
\begin{subequations}\label{eq:distances}
  \begin{itemize}
  \item Hellinger distance
    \begin{equation}
      {\rm H}(\measi,\measii) \dfn \left(\frac{1}{2}\int_{[-1,1]^d} \left(\sqrt{\frac{\ddd\measi}{\ddd\mu}(\bsx)}-\sqrt{\frac{\ddd\measii}{\ddd\mu}(\bsx)}\right)^2 \dd\mu(\bsx)\right)^{1/2},
    \end{equation}
  \item total variation distance
    \begin{equation}
      {\rm TV}(\measi,\measii) \dfn \sup_{A\in \cA}|\measi(A)-\measii(A)|=\frac{1}{2}\int_{[-1,1]^d} \Big|\frac{\ddd\measi}{\ddd\mu}(\bsx)-\frac{\ddd\measii}{\ddd\mu}(\bsx)\Big|\dd\mu(\bsx),
    \end{equation}
  \item Kullback--Leibler (KL) divergence
    \begin{equation}
      {\rm KL}(\measi \| \measii) \dfn \begin{cases}\int_{[-1,1]^d} \log\left(\frac{\ddd\measi}{\ddd\measii}(\bsx)\right) \dd\measi(\bsx) &\text{if }\measi\ll\measii\\
        \infty&\text{otherwise},
      \end{cases}
    \end{equation}
  \item {for $p\in [1,\infty)$, the $p$-Wasserstein
      distance %
    }
    \begin{equation}\label{eq:W1}
      {W_p}(\measi,\measii)\dfn
      \inf_{\nu\in\Gamma(\measi,\measii)}\left(\int_{[-1,1]^d\times[-1,1]^d} \norm{\bsx-\bsy}^p\dd\nu(\bsx,\bsy)\right)^{1/p},
    \end{equation}
    where $\Gamma(\measi,\measii)$ denotes the set of all
    measures on $[-1,1]^d\times[-1,1]^d$ with marginals $\measi$,
    $\measii$.
  \end{itemize}
\end{subequations}

Contrary to the Hellinger, total variation, and Wasserstein distances,
the KL divergence is not symmetric; however
${\rm KL}(\measi \|\measii)>0$ iff $\measi\neq\measii$.

\begin{remark}\label{rmk:hellinger}
  As is well known, the Hellinger distance %
  {provides an upper bound of the difference of integrals w.r.t.\
    two} different measures. Assume that
  $g\in L^2([-1,1]^d,\measi) \cap L^2([-1,1]^d,\measii)$. Then
  \begin{align*}
    &\left|\int_{[-1,1]^d} g(\bsx)\dd\measi( \bsx)
      -\int_{[-1,1]^d} g(\bsx)\dd\measii( \bsx)\right|
      =\left|\int_{[-1,1]^d} g(\bsx)\left(\frac{\ddd\measi}{\ddd\mu}(\bsx)-\frac{\ddd\measii}{\ddd\mu}(\bsx)\right)\dd\mu(\bsx)\right|\nonumber\\
    &\qquad\qquad\qquad\qquad\le \normc[L^2({[-1,1]^d},\mu)]{\sqrt{\frac{\ddd\measi}{\ddd\mu}}-\sqrt{\frac{\ddd\measii}{\ddd\mu}}}\normc[L^2({[-1,1]^d},\mu)]{g\left(\sqrt{\frac{\ddd\measi}{\ddd\mu}}+\sqrt{\frac{\ddd\measii}{\ddd\mu}}\right)}\nonumber\\
    &\qquad\qquad\qquad\qquad\le \sqrt{2}{\rm H}(\measi,\measii) \left (\norm[{L^2([-1,1]^d,\measi)}]{g}+\norm[{L^2([-1,1]^d,\measii)}]{g} \right ).
  \end{align*}
\end{remark}

\subsection{Error bounds}
{Throughout this subsection $p\in [1,\infty)$ is arbitrary but
  fixed. Under suitable assumptions it was shown in \cite[Theorem
  2]{MR4120535}, that
  \begin{equation}\label{eq:sagiv}
    W_p(T_\sharp\measi,\tilde T_\sharp\measii)\le \norm[L^\infty({[-1,1]^d})]{T-\tilde T}.
  \end{equation}
  Thus, Thm.~\ref{THM:POLYD} and Thm.~\ref{THM:RELU} readily yield
  bounds on $W_p(\tilde T_\sharp \measi,\measii)$ for the approximate
  polynomial, rational, and NN transport maps from
  Sec.~\ref{sec:pold} and Sec.~\ref{sec:relu}.}

For the other three distances/divergences in \eqref{eq:distances}, to
obtain a bound on ${\rm dist}(\tilde T_\sharp\measi,T_\sharp \measi)$, we
will upper bound the difference between the densities of those
measures. {Since the density of $\tilde T_\sharp\measi$ is given by
$f_\measi(\tilde S(\bsx))\det d\tilde S(\bsx)$, where
$\tilde S=\tilde T^{-1}$, we need to upper bound
$|f_\measi(S(\bsx))\det dS(\bsx)
-f_\measi(\tilde S(\bsx))\det d\tilde S(\bsx)|$, where $S=T^{-1}$.
This will be done in the proof of the following theorem.
To state the result, for a triangular map $S\in
C^1([-1,1]^d;[-1,1]^d)$ we first define
\begin{equation*}%
  S_{\rm min}\dfn \min_{j=1,\dots,d}\min_{\bsx\in [-1,1]^j}\partial_{j}S_j(\bsx).
\end{equation*}

\begin{theorem}\label{THM:HTVKL}
  Let $T$, $\tilde T:[-1,1]^d\to [-1,1]^d$ be bijective, monotone, and
  triangular. Define $S\dfn T^{-1}$ and $\tilde S\dfn \tilde T^{-1}$
  and assume that $T$, $S\in W^{2,\infty}([-1,1]^d)$ and $\tilde T$,
  $\tilde S\in W^{1,\infty}([-1,1]^d)$. %
  Moreover let $\measi$ be a
  probability measure on $[-1,1]^d$ such that
  $f_\measi\dfn \frac{\ddd\measi}{\ddd\mu}:[-1,1]^d\to\R_+$ is
  strictly positive and Lipschitz continuous.  Suppose that $\tau_0>0$
  is such that $\norm[{L^\infty([-1,1]^d)}]{d\tilde S}<\frac{1}{\tau_0}$
  and $\tilde S_{\rm min}\ge \tau_0$.

  Then there exists $C$ depending on $\tau_0$ but not on $\tilde T$
  such that for
  ${\rm dist}\in{\{{\rm H},{\rm TV}\}}$
  \begin{equation}\label{eq:measdiffd}
    {\rm dist}(T_\sharp\measi,\tilde T_\sharp\measi)\le C \norm[{W^{1,\infty}([-1,1]^d)}]{T-\tilde T}
  \end{equation}
  and
  \begin{equation}\label{eq:measdiffdKL}
    {\rm KL}(T_\sharp\measi,\tilde T_\sharp\measi)\le C \left(
      \norm[{W^{1,\infty}([-1,1]^d)}]{T-\tilde T}+\norm[{W^{1,\infty}([-1,1]^d)}]{T-\tilde T}^2\right).
  \end{equation}  
\end{theorem}

\smallskip

Together with Thm.~\ref{THM:POLYD}, we can now show
exponential convergence of the pushforward measure {in the case of
  analytic densities}.
For $\eps>0$ denote by
  $\tilde T_\eps = (\tilde T_{\eps,k})_{k=1}^d$ the approximation to
  $T$ from Thm.~\ref{THM:POLYD}. Moreover, recall that $N_\eps$ in
  \eqref{eq:Neps} denotes the number of degrees of freedom of this
  approximation (the number of coefficients of this rational
  function). As shown in Lemma \ref{LEMMA:ANALYTICITYASSUMPTION}, the
  exponential convergence shown in the next proposition holds in
  particular for positive and analytic reference and target densities
  $f_\measi$, $f_\measii$.}

\begin{proposition}\label{PROP:MEASCONVD}
  {Consider the setting of Thm.~\ref{THM:POLYD}; in particular let
    $f_\measi$, $f_\measii$ satisfy Assumption \ref{ass:finite}, and
    let $\beta>0$ be as in \eqref{eq:betad}.

    Then for every $\tilde \beta<\beta$ there exists $C>0$ such that
    for every $\eps\in (0,1)$ and
    ${\rm dist}\in{\{{\rm H},{\rm TV},{\rm KL},W_p\}}$ with
    $\tilde T_\eps$ as in \eqref{eq:ck} and $N_\eps$ as in
    \eqref{eq:Neps} it holds
  \begin{equation*}
    {\rm dist}((\tilde T_\eps)_\sharp \measi,\measii) \le C \exp(-{\tilde\beta} N_\eps^{1/d}).
  \end{equation*}} %
\end{proposition}

{Similarly, we get a bound for the pushforward under the
  NN transport from Thm.~\ref{THM:RELU}.

\begin{proposition}\label{prop:relupushforward}
  Let $f_\measi$ and $f_\measii$ be two positive and analytic
  probability densities on $[-1,1]^d$.  Then for every
  ${\rm dist}\in\{{\rm H},{\rm TV},{\rm KL},W_p\}$, there exist
  constants $\beta>0$ and $C>0$, and for every $N\in\N$ there exists a
  ReLU neural network $\Phi_N:[-1,1]^d\to [-1,1]^d$ such that
  \begin{equation*}
    {\rm dist}((\Phi_N)_\sharp \measi,\measii) \le \exp(-\beta N^{1/(d+1)}),
  \end{equation*}
  and ${\rm size}(\Phi_N)\le CN$ and
  ${\rm depth}(\Phi_N)\le C\log(N)N^{-1/2}$.
\end{proposition}

The proof is completely analogous to Prop.~\ref{PROP:MEASCONVD} (but
using Thm.~\ref{THM:RELU} instead of Thm.~\ref{THM:POLYD} to
approximate the transport $T$ with the NN $\Phi_N$) which is why we do
not give it in the appendix.}

\section{%
  {Application to inverse problems in UQ}
}\label{SEC:EXAMPLE}
To give an application and explain in more detail the practical value
of our results, we briefly discuss a standard inverse problem in
uncertainty quantification.

\subsection{Setting}\label{sec:setting}
Let $n\in\N$ and let $\D\subseteq \R^n$ be a bounded Lipschitz domain.
For a diffusion coefficient $a\in L^\infty(\D;\R)$ such that
$\essinf_{x\in\D}a(x)>0$, and a forcing term $f\in H^{-1}(\D)$,
the PDE
\begin{equation}\label{eq:pde}
  -\divg(a\nabla \scr{u})=f,\qquad \scr{u}|_{\partial\D}=0,
\end{equation}
has a unique weak solution in $H_0^1(\D)$. We denote it by
$\scr{u}(a)$, and call $\scr{u}:a\mapsto \scr{u}(a)$ the \emph{forward
  operator}.

Let $A:H_0^1(\D)\to\R^m$ be a linear \emph{observation operator} for
some
$m\in\N$. 
The inverse problem consists of
recovering the diffusion coefficient $a\in L^\infty(\D)$,
given %
the noisy observation
\begin{equation}\label{eq:meas}
  {\bsvarsigma = A(\scr{u}(a))+\bseta\in\R^m}
\end{equation}
with the additive observation noise {$\bseta\sim\cN(0,\Sigma)$,
for a symmetric positive definite covariance matrix $\Sigma\in\R^{m\times m}$.}

\subsection{Prior and posterior}\label{sec:priorposterior}
In uncertainty quantification and statistical inverse problems, the
diffusion coefficient $a\in L^\infty(\D)$ is modelled as a random
variable ({independent of the observation noise $\bseta$}) distributed according to some known prior distribution;
see, e.g., \cite{MR2102218}.  Bayes' theorem provides a formula for
the distribution of the diffusion coefficient conditioned on the
observations. This conditional is called the \emph{posterior} and
interpreted as the solution to the inverse problem.

To construct a prior, let $(\psi_j)_{j=1}^d\subset L^\infty(\D)$ and
set
\begin{equation}\label{eq:ay}
  a(\bsy)=a(\bsy,x)\dfn 1+\sum_{j=1}^dy_j\psi_j(x)\qquad x\in\D,
\end{equation}
where $y_j\in [-1,1]$. We consider the uniform measure on $[-1,1]^d$
as the prior: Every realization $(y_j)_{j=1}^d\in [-1,1]^d$
corresponds to a diffusion coefficient $a(\bsy)\in L^\infty(\D)$, and
equivalently the pushforward of the uniform measure on $[-1,1]^d$
under $a:[-1,1]^d\to L^\infty(\D)$ can be interpreted as a prior on
$L^\infty(\D)$.
Throughout we assume
$\essinf_{x\in\D}a(\bsy,x)>0$ for all $\bsy\in [-1,1]^d$ and write
$u(\bsy)\dfn\scr{u}(a(\bsy))$ for the solution of \eqref{eq:pde}.

Given $m$ measurements
$(\varsigma_i)_{i=1}^m$ as in \eqref{eq:meas}, the %
posterior measure $\measii$ on $[-1,1]^d$ is the distribution of
$\bsy|\bsvarsigma$. Since %
{$\bseta\sim\cN(0,\Sigma)$,} %
the likelihood (the density of
$\bsvarsigma|\bsy$) equals
\begin{equation*}
  {\ell(\bsvarsigma|\bsy)}
  = \frac{1}{{\sqrt{(2\pi)^m\det(\Sigma)}}} \exp\left(
    {\frac{-(A(u(\bsy))-\bsvarsigma)^\top\Sigma^{-1}(A(u(\bsy))-\bsvarsigma)}{2}}
  \right)\qquad\bsvarsigma=(\varsigma_j)_{j=1}^m\in \R^m.
\end{equation*}
By Bayes' theorem the posterior %
density $f_\measii$, corresponding to the distribution of
$\bsy|\bsvarsigma$, is proportional to the density of $\bsy$ times the
density of $\bsvarsigma|\bsy$.
Since the (uniform) prior has constant density $1$,
\begin{equation}\label{eq:postdens}
  f_\measii(\bsy) = \frac{1}{Z}{\ell(\bsvarsigma|\bsy)}
  \qquad\text{where}\qquad
  Z\dfn \int_{[-1,1]^d} {\ell(\bsvarsigma|\bsy)}\dd\mu(\bsy).
\end{equation}
The normalizing constant $Z$ is in practice unknown. For more
details see, e.g., \cite{MR3839555}.

In order to compute expectations w.r.t.\ the posterior $\measii$, we
want to determine a transport map $T:[-1,1]^d \to [-1,1]^d$ such that
$T_\sharp\mu=\measii$: then if $X_i\in[-1,1]^d$, $i=1,\dots,N$, are
iid uniformly distributed on $[-1,1]^d$, $T(X_i)$, $i=1,\dots,N$, are
iid with distribution $\pi$. This allows us to approximate integrals
$\int_{[-1,1]^d}g(\bsy)\dd\measii(\bsy)$ via Monte Carlo sampling as
$\frac{1}{N}\sum_{i=1}^N g(T(X_i))$.

\subsection{Determining $\Lambda_{k,\eps}$}
Choose as the reference measure the uniform measure $\measi=\mu$ on
$[-1,1]^d$ and let the target measure $\measii$ be the posterior with
density $f_\measii$ as in \eqref{eq:postdens}. To apply
Thm.~\ref{THM:POLYD}, we first need to determine
$\bsdelta\in (0,\infty)^d$, such that
$f_\measii\in C^1(\cB_{\bsdelta}([-1,1]);\C)$. Since $\exp:\C\to\C$ is
an entire function, by \eqref{eq:postdens}, in case
$u\in C^1(\cB_{\bsdelta}([-1,1]);\C)$
we have
$f_\measii\in C^1(\cB_{\bsdelta}([-1,1]);\C)$. One can show that the
forward operator $\scr{u}$ is complex differentiable from
$\set{L^\infty(\D;\C)}{\essinf_{x\in\D}\Re(a(x))>0}$ to
$H_0^1(\D;\C)$; see \cite[Example 1.2.39]{JZdiss}. Hence
$u(\bsy)=\scr{u}(\sum_{j=1}^dy_j\psi_j)$ indeed is complex differentiable,
e.g., for all $\bsy\in\C^d$ such that
\begin{equation*}
  1-\sum_{j=1}^d|\Re(y_j)|\norm[L^\infty(\D)]{\psi_j}>0.
\end{equation*}
Complex differentiability implies analyticity, and therefore $u(\bsy)$ is
analytic on $\cB_\bsdelta([-1,1];\C)$ with $\delta_j$ proportional to
$\norm[L^\infty(\D)]{\psi_j}^{-1}$:

\begin{lemma}\label{LEMMA:EXAMPLE}
  There exists $\tau=\tau(\scr{u},{\Sigma},d)>0$ and an increasing
  sequence $(\kappa_j)_{j=1}^d\subset (0,1)$ such that with
  $\delta_j\dfn \kappa_j+\frac{\tau}{\norm[L^\infty(\D)]{\psi_j}}$,
  $f_\measii$ in \eqref{eq:postdens} satisfies Assumption
  \ref{ass:finite}.
\end{lemma}

Let ${\varrho_j}=1+C_7\delta_j$ be as in Thm.~\ref{THM:POLYD} (i.e., as in
\eqref{eq:xidef}), where $C_7$ is as in Thm.~\ref{THM:COR:DINN}. With
$\kappa_j\in (0,1]$, $\tau>0$ as in Lemma \ref{LEMMA:EXAMPLE}
\begin{equation*}
  {\varrho_j} = 1+C_7\delta_j
  = 1 + C_7\left(\kappa_j+\frac{\tau}{\norm[L^\infty(\D)]{\psi_j}}\right).
\end{equation*}
In particular ${\varrho_j}\ge 1+C_7\tau\norm[L^\infty(\D)]{\psi_j}^{-1}$. In
practice we do not know $\tau$ and $C_7$ (although pessimistic
estimates could be obtained from the proofs), and we simply set
${\varrho_j}\dfn 1+\norm[L^\infty(\D)]{\psi_j}^{-1}$. Thm.~\ref{thm:TdN}
(and Thm.~\ref{THM:POLYD}) then suggest the choice
(cp.~\eqref{eq:eta}, \eqref{eq:Leps})
\begin{equation}\label{eq:Lepsconc}
  \Lambda_{k,\eps}\dfn \setc{\bsnu\in\N_0^k}{(1+\norm[L^\infty(\D)]{\psi_k}^{-1})^{-\max\{1,\nu_k\}}\prod_{j=1}^{k-1}(1+\norm[L^\infty(\D)]{\psi_j}^{-1})^{-\nu_j}\ge \eps}
\end{equation}
to construct a sparse polynomial ansatz space $\bbP_{\Lambda_{k,\eps}}$
in which to approximate $T_k$ (or
$\sqrt{\partial_{k}T_k}-1$), %
$k\in\{1,\dots,d\}$. Here $\eps>0$ is a
thresholding parameter, and as $\eps\to 0$ the ansatz spaces become
arbitrarily large.
We interpret \eqref{eq:Lepsconc} as follows: the smaller
$\norm[L^\infty(\D)]{\psi_j}$, the less important variable $j$ is in
the approximation of $T_k$ if $j<k$. The $k$th variable plays a
special role for $T_k$ however, and is always among the most
important in the approximation of $T_k$.

{
\subsection{Performing inference}

Given data $\bsvarsigma\in\R^m$ as in \eqref{eq:meas}, we can now
describe a high-level algorithm to perform inference on the model
problem in Sec.~\ref{sec:setting}--\ref{sec:priorposterior}:
\begin{enumerate}
\item {\bf Determine ansatz space:} Fix $\eps>0$ and determine $\Lambda_{k,\eps}$ in
  \eqref{eq:Lepsconc} for $k=1,\dots,d$.
\item {\bf Find transport map:} Use as a target $\pi$ the posterior
  with density in \eqref{eq:postdens}, and solve the optimization
  problem
  \begin{equation}\label{eq:optimization}
    \argmin_{\tilde T\text{ as in \eqref{eq:ck} with $p_k\in\bbP_{\Lambda_{k,\eps}}$}}{\rm dist}(\tilde T_\sharp\measi,\measii).
  \end{equation}
\item {\bf Estimate parameter:} Estimate the unknown parameter
  $\bsy\in [-1,1]^d$ via its conditional mean (CM), i.e., compute the
  expectation under the posterior $\pi\simeq\tilde T_\sharp\measi$
  \begin{equation}\label{eq:conditionalmean}
    \int_{[-1,1]^d} \bsy \dd\measii(\bsy) \simeq \int_{[-1,1]^d} \tilde T(\bsx)\dd\measi(\bsx)\dfnn \tilde \bsy.
  \end{equation}
  An estimate of the unknown diffusion coefficient
  $a\in L^\infty(\D)$ in \eqref{eq:ay} is obtained via
  $1+\sum_{j=1}^d \tilde y_j\psi_j$.
\end{enumerate}
We next provide more details for each of those steps.

\subsubsection{Determining the ansatz space}
An efficient algorithm (of linear complexity) to determine multi-index
sets $\Lambda_{k,\eps}$ of the type \eqref{eq:Lepsconc} is given in
\cite{MR2566594} or \cite[Sec.~3.1.3]{JZdiss}. We emphasize that, in
general, it is not an easy task to come up with suitable ansatz
spaces. In the current setting, our explicit knowledge of the prior
measure and its possibly anisotropic structure, and the analyticity of
the forward operator, allow us---using the analysis of the
previous sections---to determine \textit{a priori} ansatz spaces yielding proven
exponential convergence for the best approximating transport map
within this space. By ``anisotropic,'' we mean that certain variables
$y_j$ may contribute less to the posterior than others due to
$\norm{\psi_j}$ being small in \eqref{eq:ay}; in this case our
construction \eqref{eq:Lepsconc} results in fewer degrees of freedom
spent on such variables, thus increasing the efficiency of the
algorithm. This is to be contrasted with the use of generic ansatz
spaces, which do not leverage such knowledge, as for example proposed
in \cite{MR3821485}. The explicit construction of these spaces is one
of the main contributions of this work.

\subsubsection{Finding the transport map}
Again let $\measii$ with density $f_\measii$ be the posterior in
\eqref{eq:postdens}. The ``${\rm dist}$'' function in
\eqref{eq:optimization} is often chosen to be the KL divergence. In
this case the optimization problem \eqref{eq:optimization} can
equivalently be written as
\begin{equation}\label{eq:opt2}
  \argmin_{\tilde T\text{ as in \eqref{eq:ck} with $p_k\in\bbP_{\Lambda_{k,\eps}}$}}
  \int_{[-1,1]^d}\left(\log(f_\measi(\tilde T(\bsx))) -\log(\ell(\bsvarsigma|\tilde T(\bsy))) \right) \dd\measi(\bsx).
\end{equation}
To minimize this term, the normalizing constant $Z$ in
\eqref{eq:postdens}, which is in general unavailable, need not be
known; see, e.g., \cite[Sec.~3.2]{MR3821485} for more
details. Moreover, since the reference $\measi$ is a tractable
measure, the integral in \eqref{eq:opt2} can be approximated. The
simplest way is by Monte Carlo sampling (from $\measi$),
but also higher-order methods like quasi-Monte Carlo
\cite{DLS16_1336} or sparse-grid quadrature, e.g.,
\cite{MR1669959,ZS17}, could be used, though they have not yet been rigorously
investigated in this context. Solving the optimization problem
\eqref{eq:opt2} is in general hard, as the objective is non-convex and
defined on a high-dimensional space. Practical implementations have employed quasi-Newton or Newton-CG methods \cite{MR2972870,brennan2020greedy,transportmapslibrary}, with continuation heuristics to address non-convexity. Coming up with good optimization algorithms for this objective is out of the scope of the present paper
but would be an interesting topic for future research.

Instead of minimizing over rational transports from \eqref{eq:ck} as
in \eqref{eq:optimization} and \eqref{eq:opt2}, alternatively we could
minimize over the NN transports from Thm.~\ref{THM:RELU}. Since the
proof of Thm.~\ref{THM:RELU} is constructive, we could in principle
give an explicit architecture achieving the rate in
\eqref{eq:thm:relu}. We refrain from doing so, since an implementation
of this (sparse) architecture may be cumbersome in practice, and not
necessarily yield significantly better global minima than a fully
connected network of similar size (which also allows for exponential
convergence; see the discussion after Thm.~\ref{THM:RELU}).

Other objectives and approaches are possible as well. For instance, if
the reference $\measi$ is chosen as the uniform measure, then
$\tilde T_\sharp\rho$ has density $\det d\tilde S$ with
$\tilde S=\tilde T^{-1}$. In this case we may minimize the Hellinger
distance
\begin{equation*}
  \argmin_{\tilde T}{\rm H}(\tilde T_\sharp\measi,\measii)
  =\argmin_{\tilde S}\frac{1}{\sqrt{2}}\norm[{L^2([-1,1],\mu)}]{\sqrt{f_\measii}-\sqrt{\det d\tilde S}}.
\end{equation*}
To make this optimization problem independent of the normalizing
constant---and, moreover, convex---one can proceed as follows:
Substitute $\tilde g \dfn \sqrt{\det d \tilde S}$, and find
\begin{equation*}
  \argmin_{\tilde g}\norm[{L^2([-1,1],\mu)}]{\sqrt{f_\measii}-\tilde g}.
\end{equation*}
Observe that
by using, for example, a polynomial ansatz
$\tilde g(\bsy)=\sum_{\bsnu\in\Lambda}c_\bsnu \bsy^\bsnu$, this yields
a (convex) linear least-squares problem for the expansion
coefficients $(c_\bsnu)_{\bsnu\in\Lambda}$. Moreover, after
normalizing
\begin{equation*}
  g(\bsy)\dfn \frac{\tilde g(\bsy)}{\norm[{L^2([-1,1]^d,\mu)}]{\tilde g}},
\end{equation*}
which guarantees $g^2$ to be a probability density, this method does
not depend on the unknown normalizing constant of the target
$f_\measii$.  Using the explicit formulas for the (inverse)
KR transport in Sec.~\ref{SEC:T} we get
\begin{equation*}
  \tilde S_j(\bsx_{[j]}) = \frac{\int_{-1}^{x_j}\int_{[-1,1]^{j-d}} g^2(\bsx_{[j-1]},t,\bsy_{[j+1:d]})\dd\mu(\bsy_{[j+1:d]})\dd\mu(t)}{\int_{[-1,1]^{d-j}}g^2(\bsx_{[j-1]},\bsy_{[j:d]}\dd\mu(\bsy_{[j:d]})}.
\end{equation*}
If $\tilde g$ is a multivariate polynomial, these integrals can be
evaluated analytically (without the use of numerical quadrature),
which makes the method computationally feasible.  A similar algorithm
based on tensor-train approximations of the density has recently been
proposed in \cite{MR4065222}.

\subsubsection{Estimating the parameter}
{In the Bayesian setting, a natural point estimate of the parameter $\bsy$ is the CM estimator in \eqref{eq:conditionalmean}.
An alternative is the MAP estimator} (a point maximizing the posterior density), but the CM has the advantage of satisfying several useful stability properties. For example, in
contrast to the MAP, under suitable assumptions the CM depends
continuously on the data $\bsvarsigma$, e.g.,
\cite{stuartacta,MR3640629,MR4075344}. Moreover, minimizing objectives
like the KL divergence or the Hellinger distance will guarantee
convergence of the CM as the number of degrees of freedom in the
approximation of $\tilde T$ tends to infinity.  It will not guarantee
convergence of a MAP point, however, since $\tilde T_\sharp\measi$ need
not converge pointwise to $\measii$ if $\tilde T$ minimizes the
KL divergence as in \eqref{eq:opt2}. To see convergence of the CM,
recall that the Hellinger distance is bounded according to
${\rm H}(\tilde T_\sharp\measi,\measii)^2\le \frac{1}{2}{\rm KL}(\tilde T_\sharp\measi,\measii)$; 
see \cite{GS2002}. Moreover, by
Rmk.~\ref{rmk:hellinger}
\begin{equation*}
  \normc{\int_{[-1,1]^d}\tilde T(\bsx)\dd\measi(\bsx)-\int_{[-1,1]^d}\bsy \dd\measii(\bsy)}\le C \,  {\rm H}(\tilde T_\sharp\measi,\measii).
\end{equation*} %
By Prop.~\ref{PROP:MEASCONVD}, for $\tilde T$ solving the optimization
problem \eqref{eq:opt2} (i.e., minimizing
${\rm KL}(\tilde T_\sharp\measi,\measii)$), we thus have
\begin{equation*}
  \normc{\int_{[-1,1]^d}\tilde T(\bsx)\dd\measi(\bsx)-\int_{[-1,1]^d}\bsy \dd\measii(\bsy)}\le C \exp\left(-\frac{\tilde \beta}{2} N^{\frac{1}{d}} \right)
\end{equation*} %
with $\tilde \beta<\beta$ as in \eqref{eq:beta}, i.e., our
approximation to the actual CM
$\int_{[-1,1]^d}\bsy \dd\measii(\bsy)\in\R^d$
 converges exponentially
in terms of the number of degrees of freedom $N$ used in \eqref{eq:ck}
to approximate $T$. If instead of \eqref{eq:ck} we use a NN, by
Prop.~\ref{prop:relupushforward} we obtain a similar bound in terms of
the trainable parameters $N$ of the network, but for a different
constant $\tilde \beta$, and with $d$ replaced by $d+1$.}

{Given an approximation $\tilde T$ to the transport $T$, any other posterior expectation $$\int_{[-1,1]^d}g(\bsy)\dd\measii(\bsy) = \int_{[-1,1]^d} g(T(\bsx)) \dd \measi(\bsx)$$ can be also approximated by substituting $\tilde T$ for $T$ above; the CM simply corresponds to setting $g$ to be the identity function. Choosing a polynomial $g$, for example, enables computation of the covariance or higher moments of the posterior, as a way of characterizing uncertainty in the parameter.}

\begin{remark}
  The forward operator being defined by the diffusion equation
  \eqref{eq:pde} is not {essential to} the discussion {in
    Sec.~\ref{SEC:EXAMPLE}}. Other models such as the 
  Navier-Stokes equations allow similar arguments
  \cite{CSZ18}. More generally, the proof of Lemma
  \ref{LEMMA:EXAMPLE} merely requires the existence of a complex
  differentiable {function $\scr{u}$ (between two complex Banach
    spaces)} such that $u(\bsy)=\scr{u}(1+\sum_{j=1}^dy_j\psi_j)$, and
  $\scr{u}$ does not even need to stem from a PDE.
\end{remark}

  \section{Conclusions}
  In this paper we proved several results for the approximation of the
  KR transport $T:[-1,1]^d\to [-1,1]^d$. The central
  requirement was that the reference and target densities are both
  analytic.  Based on this, we first conducted a careful analysis of
  the regularity of $T$ by investigating its domain of analytic
  extension. This implied exponential convergence for sparse
  polynomial and ReLU neural network approximations. We gave an ansatz
  for the computation of the approximate transport, which necessitates
  that it is a bijective map on the hypercube $[-1,1]^d$. This led to
  a statement for rational approximations of $T$. {Moreover, we
    discussed how our results can be used in the
    development of inference algorithms.}

  {Most of these results are generalized and extended in
    \cite{zm2}, where we establish dimension-robust convergence in the
    high- or infinite-dimensional case $d\gg 1$ or $d=\infty$.}
  For future research, we intend to use our proposed ansatz, including
  the \textit{a priori}-determined sparse polynomial spaces, to
  construct {and analyze in more detail concrete inference
    algorithms as outlined in Sec.~\ref{SEC:EXAMPLE}}. The present
  regularity and approximation results for $T$ provide crucial tools
  to rigorously prove convergence and convergence rates for such
  methods. Additionally, we will investigate similar results on
  unbounded domains, e.g., $\R^d$.

\appendix
\section{Inverse function theorem}\label{sec:InvFunc}
In the following, if $O\subseteq\C^n$ is a set, by $f\in C^1(O;\C)$ we
mean that $f\in C^0(O;\C)$ and for every open $S\subseteq O$ it holds
$f\in C^1(S;\C)$.

\begin{lemma}\label{lemma:impl}
  Let $n\in\N$, $\delta>0$, $\kappa\in (0,1)$ and let
  $O\subseteq\C^n$. Assume that $x_0\in O$, $t_0\in \C$ and
  $f\in C^1(O\times \cB_\delta(t_0);\C)$ satisfy
  \begin{enumerate}[label=(\alph*)]
  \item\label{item:impl1} $f(x_0,t_0)=0$ and $f_t(x_0,t_0)\neq 0$,
  \item\label{item:impl2}
    $|1-\frac{f_t(x,t)}{f_t(x_0,t_0)}|\le \kappa$ for all
    $(x,t)\in O\times \cB_\delta(t_0)$,
  \item\label{item:impl3}
    $|\frac{f(x,t_0)}{f_t(x_0,t_0)}|\le\delta(1-\kappa)$ for all
    $x\in O$.
  \end{enumerate}

  Then there exists a unique function $t:O\to \cB_\delta(t_0)$ such
  that $f(x,t(x))=0$ for all $x\in O$. Moreover $t\in C^1(O;\C)$.
\end{lemma}
\begin{proof}
  Define $S(x,t)\dfn t-\frac{f(x,t)}{f_t(x_0,t_0)}$, i.e.,
  $S:O\times \cB_\delta(t_0)\to\C$. Then
  $S_t(x,t) = 1 - \frac{f_t(x,t)}{f_t(x_0,t_0)}$, and by
  \ref{item:impl2} we have for all $(x,t)\in O\times \cB_\delta(t_0)$
  \begin{equation}\label{eq:contr}
    |S_t(x,t)|\le \kappa<1.
  \end{equation}
  Moreover, \ref{item:impl3} and \eqref{eq:contr} imply for all
  $(x,t)\in O\times \cB_\delta(t_0)$
  \begin{equation}\label{eq:self}
    |S(x,t)-t_0|\le |S(x,t_0)-t_0|+|S(x,t)-S(x,t_0)|
    \le \delta(1-\kappa)+\delta \sup_{(x,t)\in O\times \cB_\delta(t_0)}|S_t(x,t)|
    \le \delta.
  \end{equation}

  Now define the Banach space $X\dfn C^0(O;\C)$ with norm
  $\norm[X]{f}\dfn \sup_{t\in O}|f(t)|$, and consider the closed
  subset $A\dfn \set{f\in X}{\norm[X]{f-t_0}\le\delta}\subset X$
  (here, by abuse of notation, $t_0:O\to\C$ is interpreted as the
  constant function in $X$). By \eqref{eq:self},
  $t(\cdot)\mapsto S(\cdot,t(\cdot))$ maps $A$ to itself, and by
  \eqref{eq:contr} the map is a contraction there, so that it has a
  unique fixed point by the Banach fixed point theorem.  We have shown
  the existence of $t\in C^0(O;\cB_\delta(t_0))$ satisfying
  $S(x,t(x))\equiv t(x)$, which is equivalent to $f(x,t(x))\equiv 0$.
  It remains to show that $t\in C^1(O;\cB_{\delta}(t_0))$. Letting
  $t_0:O\to\C$ again be the constant function and
  $t_k(x)\dfn S(x,t_{k-1}(x))$ for every $k\ge 2$, it holds $t_k\to t$
  in $X$, i.e., $(t_k)_{k\in\N}$ converges uniformly. Since
  $t_0\in C^1(O;\C)$, we inductively obtain $t_k\in C^1(O;\C)$ for all
  $k\in\N$.
  Since $X$ is a \emph{complex} Banach space, as a uniform limit of
  differentiable (analytic) functions it holds
  $\lim_{k\to\infty}t_k=t\in C^1(O;\C)$, see for instance
  \cite[Sec.~3.1]{herve89}.
  
  Finally, to see that for each $x\in O$ there exists only one
  $s\in \cB_\delta(t_0)$ such that $f(x,s)=0$ (namely $s=t(x)$), one
  can argue similar as above and apply the Banach fixed point theorem
  to the map $s\mapsto S(x,s)$ for $x\in O$ fixed and
  $s\in \cB_\delta(t_0)$.
\end{proof}

From the previous lemma, we deduce two types of inverse function
theorems in Prop.~\ref{prop:impl} and Prop.~\ref{prop:localhol}.

\begin{proposition}\label{prop:impl}
  Let $t_0\in\C$ and $\delta>0$ be such that
  \begin{enumerate}[label=(\alph*)]
  \item\label{item:prop:impl1} $F\in C^1(\cB_\delta(t_0);\C)$ and
    $F'(t_0)\neq 0$,
  \item\label{item:prop:impl2} $F':\cB_\delta(t_0)\to\C$ is Lipschitz
    continuous with Lipschitz constant
    $L\dfn \frac{|F'(t_0)|}{2\delta}$.
  \end{enumerate}
  
  Then with $r\dfn \delta\frac{|F'(t_0)|}{2}$ there exists a unique
  $G:\cB_r(F(t_0))\to \cB_\delta(t_0)$ such that $F(G(x))=x$ for all
  $x\in \cB_r(F(t_0))$. Moreover $G\in C^1(\cB_r({F(t_0)});\C)$.
\end{proposition}
\begin{proof}
  Let $O\dfn \cB_{r}(F(t_0))$. Define $x_0\dfn F(t_0)$ as well as
  $f(x,t)\dfn F(t)-x$. Then $f(x_0,t_0)=0$,
  $f_t(x_0,t_0)=F'(t_0)\neq 0$, showing Lemma \ref{lemma:impl}
  \ref{item:impl1}. Furthermore, for $t\in \cB_{\delta}(t_0)$ %
  and $x\in O=\cB_r(x_0)$, due to the Lipschitz continuity of $F'$ and
  $L=\frac{|F'(t_0)|}{2\delta}$
  \begin{equation*}
    \left|1-\frac{f_t(x,t)}{f_t(x_0,t_0)}\right|=\left|1-\frac{F'(t)}{F'(t_0)}\right|=
    \left|\frac{F'(t_0)-F'(t)}{F'(t_0)}\right|\le
    \frac{L|t-t_0|}{|F'(t_0)|}\le \frac{L\delta}{|F'(t_0)|}=\frac{1}{2},
  \end{equation*}
  which shows Lemma \ref{lemma:impl} \ref{item:impl2} with
  $\kappa=\frac{1}{2}$. Finally, Lemma \ref{lemma:impl}
  \ref{item:impl3} (with $\kappa=\frac{1}{2}$) follows from the fact
  that for $x\in O=\cB_r(F(t_0))$
  \begin{equation*}
    \left|\frac{f(x,t_0)}{f_t(x_0,t_0)}\right|=\frac{|F(t_0)-x|}{|F'(t_0)|}
    \le \frac{r}{|F'(t_0)|}
    =\frac{\delta}{2}.
  \end{equation*}
  Hence the statement follows by Lemma \ref{lemma:impl}.
\end{proof}

The next lemma shows that $G$ in Prop.~\ref{prop:impl} depends
continuously on $F$.

\begin{lemma}\label{lemma:Gcontinuous}
  Let $t_0\in\C$, $\delta>0$ be such that both
  $F\in C^1(\cB_\delta(t_0))$ and $\tilde F\in C^1(\cB_\delta(t_0))$
  satisfy Prop.~\ref{prop:impl} \ref{item:prop:impl1},
  \ref{item:prop:impl2}.  Denote the functions from
  Prop.~\ref{prop:impl} by $G$, $\tilde G$ respectively.  With
  $r=\delta\frac{|F'(t_0)|}{2}$,
  $\tilde r=\delta\frac{|\tilde F'(t_0)|}{2}$
  \begin{equation*}
    \sup_{
      x\in \cB_r(F(t_0))\cap \cB_{\tilde r}(\tilde F(t_0))}|G(x)-\tilde G(x)|\le %
    \frac{2\norm[L^\infty(\cB_\delta(t_0))]{F-\tilde F}}{\max\{|F'(t_0)|,|\tilde F'(t_0)|\}}.
  \end{equation*}
\end{lemma}
\begin{proof}
  Let $s$, $t\in \cB_\delta(t_0)$. Then
  \begin{equation*}
    |F(s)-F(t)| = |s-t|\left|\int_{0}^1F'((1-\zeta)s+\zeta t)\dd \zeta\right|.
  \end{equation*}
  For any $\zeta\in [0,1]$ it holds
  $(1-\zeta)s+\zeta t\in\cB_\delta(t_0)$ and therefore
  $|(1-\zeta)s+\zeta t-t_0|\le\delta$.  Thus, using that
  $\frac{|F'(t_0)|}{2\delta}$ is a Lipschitz constant of $F'$,
  \begin{align*}
    \left|\int_{0}^1F'((1-\zeta)s+\zeta t)\dd \zeta\right|
    &= \left|\int_{0}^1F'(t_0)+(F'((1-\zeta)s+\zeta t)-F'(t_0))\dd
      \zeta\right|\nonumber\\
    &\ge |F'(t_0)|-\frac{|F'(t_0)|}{2\delta}\delta
      \ge \frac{|F'(t_0)|}{2}.
  \end{align*}
  We get for $x\in\cB_r(F(t_0))\cap \cB_{\tilde r}(\tilde F(t_0))$
  (applying this inequality with $s=G(x)$ and $t=\tilde G(x)$)
  \begin{align*}
    |G(x)-\tilde G(x)|&\le \frac{2}{|F'(t_0)|}
                        |F(G(x))- F(\tilde G(x))|\nonumber\\
                      &=\frac{2}{|F'(t_0)|} |\tilde F(\tilde G(x))- F(\tilde G(x))|\nonumber\\
                      &\le \frac{2\norm[L^\infty(\cB_\delta(t_0))]{F-\tilde F}}{|F'(t_0)|}.\qedhere
  \end{align*}
\end{proof}

\begin{proposition}\label{prop:localhol}
  Let %
  $k\in\N$, let $S\subseteq \C^{k+1}$ be open and $F\in C^1(S;\C)$.
  Assume that $x_0\in \C^k$ and $t_0\in\C$ are such that
  $F_t(x_0,t_0)\neq 0$.

  Then there exists $\delta>0$, a neighbourhood $O\subseteq \C^{k+1}$
  of $(x_0,F(x_0,t_0))$ and a unique function
  $G:O \to \cB_\delta(t_0)$ such that $F(x,G(x,y))=y$ for all
  $(x,y)\in O$. Moreover $G\in C^1(O;\C)$.
  
\end{proposition}
\begin{proof}
  Let $\kappa\dfn \frac{1}{2}$. Choose an open set
  $O\subseteq\C^{k+1}$ and $\delta>0$ such that with
  $y_0\dfn F(x_0,t_0)\in\C$ it holds $(x_0,y_0)\in O$
  and %
  $|1-\frac{F_t(x,t)}{F_t(x_0,t_0)}|\le \kappa$ for all
  $((x,y),t)\in O\times \cB_\delta(t_0)$, and
  $|\frac{F(x,t_0)-y}{F_t(x_0,t_0)}|\le\delta(1-\kappa)$ for all
  $(x,y)\in O$. This is possible because $F$, $F_t$ are locally
  continuous around $(x_0,t_0)$. Set $f((x,y),t)\dfn
  F(x,t)-y$. Applying Lemma \ref{lemma:impl} to
  $f:O\times \cB_\delta(t_0)\to\C$ gives the result.
\end{proof}

\section{Proofs of Sec.~\ref{SEC:ANT}}
\subsection{Lemma \ref{LEMMA:FEXT}}\label{app:Fext}
\begin{proof}[Proof of Lemma \ref{LEMMA:FEXT}]
  We start with \ref{item:Fiso}.  By definition
  $F\in C^1([-1,1];[0,1])$ is strictly monotonically increasing with
  derivative $F'=f$. This implies that $F$ is bijective and by the
  inverse function theorem, its inverse belongs to $C^1([0,1])$.

  Holomorphy of $f:\cB_\delta([-1,1])\to\C$ and the simply
  connectedness of $\cB_\delta([-1,1])$ imply the existence and
  well-definedness of a unique holomorphic antiderivative
  $F:\cB_\delta([-1,1])\to\C$ of $f$ satisfying $F(-1)=0$. Since
  $F'=f$, $F:\cB_\delta([-1,1])\to \C$ has Lipschitz constant
  $\sup_{x\in \cB_{\delta}([-1,1])}|f(x)|\le L$.

  We show \ref{item:Finvhol}. For every $x_0\in [-1,1]$ holds
  $F'(x_0)=f(x_0)\neq 0$. By Lemma \ref{lemma:ext} there exists a
  unique $G_{x_0}:\cB_{\alpha\delta}(F(x_0))\to \cB_{\beta\delta}(x_0)$
  satisfying $F(G_{x_0}(y))=y$ for all $y\in
  \cB_{\alpha\delta}(F(x_0))$. It holds $\beta\le 1$ by definition and
  thus $\cB_{\beta\delta}(x_0)\subseteq \cB_{\delta}(x_0)$. Since
  $F:[-1,1]\to [0,1]$ is bijective,
  {we can define
  $G:\cB_{\alpha\delta}([0,1])\to \cB_{\beta\delta}([-1,1])$  
  locally via
  $G(y)\dfn G_{x_0}(y)$ whenever $y\in \cB_{\alpha\delta}(F(x_0))$.
  It remains to show well-definedness of $G$, i.e., that these local
  functions coincide wherever their domain overlaps. We then have
  $F(G(y))=y$ for all $y\in \cB_{\alpha\delta}([0,1])$ by definition.}

Let $x_0\neq x_1$ be arbitrary in $[-1,1]$ and denote the
corresponding local inverse functions of $F$ by
$G_{x_i}:\cB_{\alpha\delta}(F(x_i))\to \cB_{\beta\delta}(x_i)$,
$i\in\{0,1\}$. The uniqueness of $G_{x_0}$ and $G_{x_1}$ (as stated in
Lemma \ref{lemma:ext}) and the continuity of $F^{-1}:[0,1]\to [-1,1]$
imply that $G_{x_j}\equiv F^{-1}$ on
$[0,1]\cap \cB_{\alpha\delta}(F(x_j))\neq\emptyset$, $j\in\{0,1\}$.
Now assume that there exists
$y\in \cB_{\alpha\delta}(F(x_0))\cap \cB_{\alpha\delta}(F(x_1))\cap
[0,1]$. Then $G_{x_0}(y)=F^{-1}(y)=G_{x_1}(y)$ and again by the local
uniqueness of $G_{x_0}$, $G_{x_1}$ as the inverse of $F$ those two
functions coincide on a complex open subset of
$\cB_{\alpha\delta}(F(x_0))\cap \cB_{\alpha\delta}(F(x_1))$. Since
they are holomorphic, by the identity theorem they coincide on all of
$\cB_{\alpha\delta}(F(x_0))\cap \cB_{\alpha\delta}(F(x_1))$. Thus $G$
is well-defined on $\cB_{\alpha\delta}([-1,1])$.
\end{proof}

We will also use the following consequence of Lemma \ref{LEMMA:FEXT}.
\begin{lemma}\label{LEMMA:FEXT2}
  Let $f$, $F$, {$M$, $L$} and $\alpha$ be as in Lemma
  \ref{LEMMA:FEXT}. If $\tilde f\in C^1(\cB_\delta([-1,1]);\C)$
  satisfies $M\le |\tilde f(x)|\le L$ for all
  $x\in \cB_\delta([-1,1])$ and
  \begin{equation}\label{eq:f-ft}
    \sup_{x\in [-1,1]}|f(x)-\tilde f(x)|\le \frac{\alpha\delta}{2},
  \end{equation}
  then for $\tilde F(x)\dfn \int_{-1}^x\tilde f(t)\dd\mu( t)$ it holds
  $\cB_{\frac{\alpha\delta}{2}}([0,1])\subseteq \set{\tilde F(x)}{x\in
    \cB_\delta([-1,1])}$. Furthermore there exists a unique
  $\tilde G:\cB_{\frac{\alpha\delta}{2}}([0,1])\to
  \cB_{\beta\delta}([-1,1])$ such that $\tilde F(\tilde G(y))=y$ for
  all $y\in \cB_{\frac{\alpha\delta}{2}}([0,1])$ and
  \begin{equation}\label{eq:uniquenesstilde}
    \tilde G:\cB_{\frac{\alpha\delta}{2}}(F(x))\to \cB_{\beta\delta}(x)\qquad
    \forall  x\in [-1,1].
  \end{equation}
  Moreover $\tilde G\in C^1(\cB_{\frac{\alpha\delta}{2}}([0,1]);\C)$
  with Lipschitz constant $\frac{1}{M}$.
\end{lemma}
\begin{proof}
  Fix $x_0\in \cB_\delta([-1,1])$. It holds
  $\tilde F'(x_0)=\tilde f(x_0)$. By Lemma \ref{lemma:ext}, %
  there exists a unique
  $\tilde G_{x_0}:\cB_{\alpha\delta}(\tilde F(x_0))\to
  \cB_{\beta\delta}(x_0)\subseteq \cB_\delta(x_0)$ satisfying
  $\tilde F(\tilde G_{x_0}(y))=y$ for all
  $y\in \cB_{\alpha\delta}(\tilde F(x_0))$.
  In particular
  $\cB_{\alpha\delta}(\tilde F(x_0))\subseteq \set{\tilde F(x)}{x\in
    \cB_\delta(x_0)}$.
  By \eqref{eq:f-ft} and because $\mu$ is a
  probability measure on $[-1,1]$, for all $x\in [-1,1]$
  \begin{equation*}
    |F(x)-\tilde F(x)|\le \int_{-1}^x|f(t)-\tilde f(t)|\dd\mu( t) %
    \le\frac{\alpha\delta}{2}.
  \end{equation*}
  Therefore
  \begin{equation*}%
    \cB_{\frac{\alpha\delta}{2}}([0,1])=
    \bigcup_{x\in [-1,1]}\cB_{\frac{\alpha\delta}{2}}(F(x))\subseteq
    \bigcup_{x\in [-1,1]} \cB_{\alpha\delta}(\tilde F(x))
    \subseteq \set{\tilde F(x)}{x\in \cB_\delta([-1,1])}.
  \end{equation*}

  Restricting $\tilde G_{x_0}$ to
  $\cB_{\frac{\alpha\delta}{2}}(F(x_0))\subset
  \cB_{\alpha\delta}(\tilde F(x_0))$ locally defines
  $\tilde
  G:\cB_{\frac{\alpha\delta}{2}}([0,1])\to\cB_{\beta\delta}([-1,1])$,
  since $F:[-1,1]\to [0,1]$ is bijective by Lemma \ref{lemma:ext}.
  {We
  next show well-definedness of $\tilde G$, i.e., these local
  functions coincide whenever their domain of definition overlaps.}

  Let $y_1$, $y_2\in [0,1]$ be
  arbitrary with
  \begin{equation}\label{eq:ydiff}
    |y_1-y_2|<\frac{\alpha\delta M}{M+L}=\frac{\beta\delta M^2}{M+L},
  \end{equation}
  where the equality holds by definition of $\beta=\frac{\alpha}{M}$
  in \eqref{eq:rrs}.
  There exist unique $x_i\in [-1,1]$ with
  $y_i=F(x_i)$, $i\in\{1,2\}$. Let
  $\tilde G_{x_i}:\cB_{\alpha\delta}(\tilde F(x_i))\to
  \cB_{\alpha\beta}(x_i)$ be the unique local inverse of $\tilde
  F$. We need to show that $\tilde G_{x_1}\equiv \tilde G_{x_2}$ on
  \begin{equation}\label{eq:Fx1domG2}
    \cB_{\frac{\alpha\delta}{2}}(y_1)\cap\cB_{\frac{\alpha\delta}{2}}(y_2)\subset
    \cB_{\alpha\delta}(\tilde F(x_1))\cap\cB_{\alpha\delta}(\tilde F(x_2)).
  \end{equation}

  First, by Lemma \ref{lemma:ext}, $F^{-1}$ has Lipschitz constant
  $\frac{1}{M}$, and we recall that since
  $|\tilde F'|=|\tilde f|\le L$ and $|F'|=|f|\le L$, both $F$ and
  $\tilde F$ have Lipschitz constant $L$. Thus by \eqref{eq:ydiff}
  \begin{equation*}
    |\tilde F(x_1)-\tilde F(x_2)|\le L|x_1-x_2|
    =L|F^{-1}(y_1)-F^{-1}(y_2)|  \le \frac{L}{M}|y_1-y_2|{<
    \frac{L}{M+L}\alpha\delta\le}
    \alpha\delta.
  \end{equation*}
  This implies $\tilde F(x_1)\in \cB_{\alpha\delta}(\tilde F(x_2))$,
  i.e., $\tilde F(x_1)$ is in the domain of $\tilde G_{x_2}$. Again by
  Lemma \ref{lemma:ext}, $\tilde G_{x_2}$ has Lipschitz constant
  $\frac{1}{M}$ and $\tilde F$ has Lipschitz constant $L$. Thus
  \begin{equation*}
    |\tilde G_{x_2}(\tilde F(x_1))-x_2|
    =|\tilde G_{x_2}(\tilde F(x_1))-\tilde G_{x_2}(\tilde F(x_2))|
    \le\frac{1}{M}|\tilde F(x_1)-\tilde F(x_2)|
    \le \frac{L}{M}|x_1-x_2|,
  \end{equation*}
  and we obtain
  \begin{equation*}
    |\tilde G_{x_2}(\tilde F(x_1))-x_1|\le \left(1+\frac{L}{M}\right)|x_1-x_2|.
  \end{equation*}
  Using again Lipschitz continuity of $F^{-1}:[0,1]\to [-1,1]$ with
  Lipschitz constant $\frac{1}{M}$, by \eqref{eq:ydiff}
  \begin{equation*}
    |\tilde G_{x_2}(\tilde F(x_1))-x_1|\le
    \left(1+\frac{L}{M}\right)|F^{-1}(y_1)-F^{-1}(y_2)|
    \le \frac{M+L}{M^2}|y_1-y_2|
    < \beta\delta.
  \end{equation*}
  Hence $\tilde G_{x_2}(\tilde F(x_1))\in\cB_{\beta\delta}(x_1)$.
  Uniqueness of
  $\tilde G_{x_1}:\cB_{\alpha\delta}(\tilde F(x_1))\to
  \cB_{\beta\delta}(x_1)$ (with the property
  $\tilde F(\tilde G_{x_1}(y))=y$ for all
  $y\in \cB_{\alpha\delta}(\tilde F(x_1))$) implies
  $\tilde G_{x_2}(\tilde F(x_1))=\tilde G_{x_1}(\tilde F(x_1))$.  By
  continuity of $\tilde G_{x_1}$ and $\tilde G_{x_2}$ and the uniqueness
  property, $\tilde G_{x_1}$ and $\tilde G_{x_2}$ coincide on a complex
  neighbourhood of $\tilde F(x_1)$.  By the identity theorem of
  complex analysis, $\tilde G_{x_1}$ and $\tilde G_{x_2}$ coincide on the
  whole intersection
  $\cB_{\alpha\delta}(\tilde F(x_1))\cap \cB_{\alpha\delta}(\tilde
  F(x_2))\supseteq \cB_{\frac{\alpha\delta}{2}}(y_1)\cap
  \cB_{\frac{\alpha\delta}{2}}(y_2)$.
\end{proof}

\subsection{Lemma \ref{LEMMA:ANALYTICITYASSUMPTION}}\label{app:lemma:analyticityassumption}
\begin{proof}[Proof of Lemma \ref{LEMMA:ANALYTICITYASSUMPTION}]
  Fix $*\in\{\measi,\measii\}$. %
  By analytic continuation, there is an open \emph{complex} set
  $O\subseteq\C^d$ containing $[-1,1]^d$ on which both functions are
  holomorphic. Moreover $0<\inf_{\bsx\in[-1,1]^d} |f_*(\bsx)|$ and
  $\sup_{\bsx\in [-1,1]^d}|f_*(\bsx)|<\infty$ for
  $*\in\{\measi,\measii\}$ by compactness of $[-1,1]^d$. Hence we can
  find $\tilde\bsdelta\in (0,\infty)^d$ such that Assumption
  \ref{ass:finite} \ref{item:cordinN:1} and \ref{item:cordinN:2} are
  satisfied for some $0<M\le L<\infty$. Fix $C_6=C_6(M,L)>0$ as in
  Thm.~\ref{THM:COR:DINN}. Again by compactness (and the fact that
  $f_*$ is continuous), decreasing the components of $\tilde\bsdelta$
  if necessary, also Assumption \ref{ass:finite} \ref{item:cordinN:3}
  holds. Before verifying \ref{item:cordinN:4}, we point out that
  Assumption \ref{ass:finite}
  \ref{item:cordinN:1}-\ref{item:cordinN:3} is valid for any
  $\bsdelta\in (0,\infty)^d$ with $\delta_j\le \tilde \delta_j$,
  $j\in\{1,\dots,d\}$.

  Let $\delta_{\rm min}\dfn \min_{j=1}^d\tilde \delta_j$.
  Continuity of $f_*:\cB_{\tilde\bsdelta}([-1,1]^d)\to\C$ and
  compactness of $[-1,1]^d$ imply with the notation
  $\eps \bsone_k=(\eps)_{j=1}^k\in\R^k$ for any $k\in\{1,\dots,d\}$
  \begin{equation}\label{eq:deltafanalytic}
    \lim_{\eps\to 0}\sup_{\bsx\in [-1,1]^d}\sup_{\bsy\in\cB_{\eps\bsone_k}\times\{0\}^{d-k}}
    |f_*(\bsx+\bsy)-f_*(\bsx)|=0.
  \end{equation}
  Let $\delta_{d}\in(0,\tilde\delta_d)$ be so small that
  $\sup_{\bsx\in [-1,1]^d}\sup_{\bsy\in\cB_{\delta_d\bsone_d}}
  |f_*(\bsx+\bsy)-f_*(\bsx)|\le C_6$.
  By \eqref{eq:deltafanalytic} we can inductively (starting with
  $k=d-1$ and ending with $k=1$) choose
  $\delta_k\in (0,\min\{\delta_{\rm min},\delta_{k+1}\})$ so small
  that
  $\sup_{\bsx\in [-1,1]^d}\sup_{\bsy\in\cB_{\delta_k\bsone_k}\times\{0\}^{d-k}}\inf_{\bsx\in [-1,1]^d}
  |f_*(\bsx+\bsy)-f_*(\bsx)|\le C_6\delta_{k+1}$.
\end{proof}

\subsection{Thm.~\ref{THM:COR:DINN}}\label{app:thm:dinN}
To prove Thm.~\ref{THM:COR:DINN}, we start with some preliminary
results investigating the functions $f_k$ in \eqref{eq:fk}.  First, we
will analyze the domain of analytic extension of
$T:[-1,1]^d\to [-1,1]^d$ for the general $d$-dimensional Knothe
transport in \eqref{eq:knothe}. The following variation of Assumption
\ref{ass:finite} will be our working assumption on the densities. In
particular, item \ref{item:fpi:eps} stipulates that the analytic
extensions of the densities do not deviate too much. This will
guarantee that the inverse CDFs can be suitably analytically extended.

\begin{assumption}\label{ass:fpi}
  For $\bsdelta \in (0,\infty)^d$, some constants
  $0<M\le L<\infty$ and
  \begin{equation}\label{eq:eps}
    0\le \eps_d\le\frac{M^3}{64 L^2},\qquad
    0\le \eps_k\le \min\{1,\delta_{k+1}\}\frac{M^5}{512 L^4+64 L^2 M^2}
    \qquad \forall k\in\{1,\dots,d-1\},
  \end{equation}
  it holds
  \begin{enumerate}[label=(\alph*)]
  \item\label{item:fpi:1}
    $f:[-1,1]^d\to\R_+$ is a probability density and
    $f\in C^1(\cB_{\bsdelta}([-1,1]);\C)$,
  \item\label{item:fpi:3} $M\le |f(\bsx)|\le L$ for all
    $\bsx\in \cB_{\bsdelta}([-1,1])$,
  \item\label{item:fpi:eps}
    $\sup_{\bsx\in [-1,1]^d}\sup_{\bsy\in \cB_{\bsdelta_{[k]}}
      \times \{0\}^{d-k}}|f(\bsx+\bsy)-f(\bsx)| \le \eps_k$
    for all $k\in\{1,\dots,d\}$.
  \end{enumerate}
\end{assumption}

\begin{remark}\label{rmk:epslepsd}
  We could have equivalently written
  $\min_{l\in\{k,\dots,d\}}\eps_l$ on the right-hand side of
  the inequality in \ref{item:fpi:eps}.
  In particular
  $|f(\bsx+\bsy)-f(\bsx)|\le \eps_d=\frac{M^3}{64L^2}$ for all
  $\bsx\in [-1,1]^d$, $\bsy\in \cB_\bsdelta\subseteq\C^d$.
\end{remark}

Item \ref{item:fk:hol} of the following lemma states that
$x_k\mapsto f_k(\bsx_{[k]})$ is a probability density on
$[-1,1]$, %
{and $\bsx_{[k]}\mapsto f_k(\bsx_{[k]})$ has the same domain of
  analyticity as $\bsx_{[k]}\mapsto f(\bsx)$.}  Items
\ref{item:fk:eps} and \ref{item:fk:1-fk} are statements about how much
$f_k$ varies in its variables: \ref{item:fk:eps} is mainly a technical
requirement used in later proofs and \ref{item:fk:1-fk} will be
relevant for large values of $\delta_k>0$. It states that the maximum
deviation of %
the probability density $x_k\mapsto f_k(\bsx_{[k]})$ from the constant
$1$ function is indirect proportional to $\delta_k$, i.e., is small
for large $\delta_k$.

\begin{lemma}\label{lemma:fk}
  Let $f:[-1,1]^d\to\R_+$ satisfy Assumption \ref{ass:fpi}, and let
  $f_k:[-1,1]^k\to\R_+$ be as in \eqref{eq:fk}. Then for every
  $k\in\{1,\dots,d\}$
  \begin{enumerate}
  \item\label{item:fk:hol} %
    $f_k\in C^1(\cB_{\bsdelta_{[k]}}([-1,1]);\C)$,
    $f_k:[-1,1]^k\to\R_+$ and $\int_{-1}^{1} f_k(\bsx,t)\dd\mu( t)=1$
    if $\bsx\in [-1,1]^{k-1}$,
  \item\label{item:fk:tMtL} with $\tilde M\dfn \frac{M}{2L}$ and
    $\tilde L\dfn \frac{2L}{M}$ it holds
    \begin{equation}\label{eq:tMletL}
      \tilde M\le |f_k(\bsx)|\le
      \tilde L\qquad\forall \bsx \in \cB_{\bsdelta_{[k]}}([-1,1]),
    \end{equation}
    and
    \begin{equation}\label{eq:Refkge}
      \Re (f_k(\bsx))\ge \frac{M}{4L},\quad
      |\Im (f_k(\bsx))|\le \frac{M}{8L},
      \qquad
      \forall\bsx \in \cB_{\bsdelta_{[k]}}([-1,1]),
    \end{equation}
  \item\label{item:fk:eps} if $k\ge 2$
    \begin{align}\label{eq:fksupinf}
      &\sup_{\bsy \in \cB_{\bsdelta_{[k-1]}}([-1,1])}\inf_{\bsx\in
      [-1,1]^{k-1}}\norm[{L^\infty([-1,1])}]{f_k(\bsy,\cdot)-f_k(\bsx,\cdot)}\nonumber\\
      &\qquad\qquad\le
      \eps_{k-1} \frac{8 L}{M}
      \le \min\{1,\delta_k\} C_1(M,L)
    \end{align}
    where
    $C_1(M,L)\dfn  \frac{\tilde M^2}{4\tilde M+8\tilde L}$,
  \item\label{item:fk:1-fk} if $k\ge 2$,
    for any $\gamma\in (0,\frac{\delta_k}{2}]$ %
    \begin{equation*}
      \sup_{\bsy \in  \cB_{\bsdelta_{[k-1]}}([-1,1])}
      \sup_{t\in \cB_\gamma([-1,1])}
      |f_k(\bsy,t)-1|
      \le
      \frac{2+\gamma}{\delta_k} C_2(M,L)
    \end{equation*}
    where
    $C_2(M,L)\dfn {\frac{4L(L+M)}{M^2}}$.
  \end{enumerate}
\end{lemma}
\begin{proof}
  {\bf Step 1.} We establish some preliminary inequalities and show
  \ref{item:fk:hol}. Analyticity of $f:\cB_{\bsdelta}([-1,1]) \to\C$
  implies that $\hat f_k:\cB_{\bsdelta_{[k]}}([-1,1])\to\C$ and
  $\hat f_{k-1}:\cB_{\bsdelta_{[k-1]}}([-1,1])\to\C$ in \eqref{eq:fk}
  are holomorphic for all $k\in\{1,\dots,d\}$ (if $k=1$,
  $\hat f_{k-1}\equiv 1$ by convention, so that $\hat f_{k-1}$ is an
  entire function). By Assumption \ref{ass:fpi} \ref{item:fpi:eps} and
  Rmk.~\ref{rmk:epslepsd}, for every $k\in\{1,\dots,d\}$ and all
  $\bsx\in [-1,1]^k$, $\bsy\in \cB_{\bsdelta_{[k]}}$
  (cp.~\eqref{eq:fk})
  \begin{align}\label{eq:fj-fj}
    |\hat f_k(\bsx+\bsy)-
    \hat f_k(\bsx)|&\le \int_{[-1,1]^{d-k}} |f(\bsx+\bsy,\bst)-
                     f(\bsx,\bst)|
                     \dd\mu(\bst)\nonumber \\
                   &\le \min\left\{\eps_{k}, \frac{M^3}{64 L^2}\right\},
  \end{align}
  and similarly for every $k\in\{2,\dots,d\}$
  \begin{align}\label{eq:fj-fj2}
    & |\hat f_k(\bsx_{[k-1]}+\bsy_{[k-1]},x_k)-
      \hat f_k(\bsx_{[k-1]},x_k)|\nonumber \\
    & \qquad\le \int_{[-1,1]^{d-k}} |f(\bsx_{[k-1]}+\bsy_{[k-1]},x_k,\bst)-
      f(\bsx_{[k-1]},x_k,\bst)|
      \dd\mu(\bst)\nonumber \\
    &\qquad\le \min\left\{\eps_{k-1},\frac{M^3}{64L^2}\right\}.
  \end{align}

  Since $f:[-1,1]^d\to\R_+$ is a probability density and
  $|f(\bsx)|\ge M$ for all $\bsx\in [-1,1]^d$, it holds
  $\hat f_k(\bsx)\ge M$ for all $\bsx\in [-1,1]^k$ and all
  $k\in\{1,\dots,d\}$. With \eqref{eq:fj-fj} and because
  $\frac{M^3}{64 L^2}\le \frac{M}{2}$ we conclude that
  \begin{equation}\label{eq:fjpos}
    |\hat f_k(\bsx+\bsy)|\ge |\hat f_k(\bsx)|-\min\left\{\eps_k,\frac{M^3}{64L^2}\right\}\ge \frac{M}{2}\qquad
    \forall \bsx\in [-1,1]^k,~\forall \bsy\in \cB_{\bsdelta_{[k]}}.
  \end{equation}
  Moreover, we note that $|f(\bsx)|\le L$ for all
  $\bsx\in \cB_\bsdelta([-1,1])$ implies
  (cp.~\eqref{eq:fk}) %
  \begin{equation}\label{eq:fjbound}
    |\hat f_k(\bsx)|
    \le L\qquad \forall \bsx\in \cB_{\bsdelta_{[k]}}([-1,1]).
  \end{equation}

  By definition $f_k(\bsx)=\frac{\hat f_k(\bsx)}{\hat
  f_{k-1}(\bsx_{[k-1]})}$. The modulus of the denominator is uniformly
  positive on $\cB_{\bsdelta_{[k]}}([-1,1])$ according to
  \eqref{eq:fjpos}, and hence
  $f_k\in C^1(\cB_{\bsdelta_{[k]}}([-1,1]);\C)$. Now \ref{item:fk:hol}
  is a consequence of Assumption \ref{ass:fpi} \ref{item:fpi:1} and
  the definition of $\hat f_k$, $\hat f_{k-1}$ in \eqref{eq:fk}.
  
  {\bf Step 2.} We show \ref{item:fk:tMtL} and let at first $k\ge 2$.
  By \eqref{eq:fjpos} and \eqref{eq:fjbound}
  \begin{equation}\label{eq:lowupbound}
    \frac{M}{2L} \le |f_k(\bsx)|=\left|\frac{\hat f_k(\bsx)}{\hat f_{k-1}(\bsx_{[k-1]})}\right|\le \frac{2L}{M}\qquad
    \forall \bsx \in \cB_{\bsdelta_{[k]}}([-1,1]),
  \end{equation}
  which shows \eqref{eq:tMletL} for $k\ge 2$.  For $k=1$, we have
  $f_1(x_1)=\hat f_1(x_1)$ (since $\hat f_0\equiv 1$). From
  {Assumption
  \ref{ass:fpi} \ref{item:fpi:1} and %
  \ref{item:fpi:3}
  it follows $M\le 1\le L$ because}
$\mu$ is a probability measure. Hence the definition of $\hat f_1$
and
  {Assumption \ref{ass:fpi} \ref{item:fpi:3}}
  imply $\frac{M}{2L}\le M\le |f_1(x_1)|\le L \le \frac{2L}{M}$ for all
  $x_1\in \cB_{\delta_1}([-1,1])$.

  To show \eqref{eq:Refkge} note that $\frac{M}{2L}\le f_k(\bsx)\in\R$
  whenever $\bsx \in [-1,1]^k$ by %
  \eqref{eq:lowupbound} and because $f_k:[-1,1]^k\to\R_+$. If
  $k\ge 2$, for $\bsx \in [-1,1]^k$ and
  $\bsy\in \cB_{\bsdelta_{[k]}}$, by \eqref{eq:fj-fj},
  \eqref{eq:fjpos} and \eqref{eq:fjbound}
  \begin{align*}
    &|f_k(\bsx+\bsy)-f_k(\bsx)|=\left|\frac{\hat f_k(\bsx)}{\hat f_{k-1}(\bsx_{[k-1]})}
      -\frac{\hat f_k(\bsx+\bsy)}{\hat f_{k-1}(\bsx_{[k-1]}+\bsy_{[k-1]})}\right|\nonumber\\
    &\qquad\qquad\le \left|\frac{\hat f_k(\bsx)(\hat f_{k-1}(\bsx_{[k-1]}+\bsy_{[k-1]})-\hat f_{k-1}(\bsx_{[k-1]}))}{\hat f_{k-1}(\bsx_{[k-1]})\hat f_{k-1}(\bsx_{[k-1]}+\bsy_{[k-1]})}\right|+\left|\frac{\hat f_{k-1}(\bsx_{[k-1]})(\hat f_{k}(\bsx)-\hat f_{k}(\bsx+\bsy))}{\hat f_{k-1}(\bsx_{[k-1]})\hat f_{k-1}(\bsx_{[k-1]}+\bsy_{[k-1]})}\right|\nonumber\\
    &\qquad\qquad\le
      \left(\frac{L}{(M/2)^2} \frac{M^3}{64L^2}\right)
      +\left(\frac{L}{(M/2)^2} \frac{M^3}{64L^2}\right)
      =\frac{2L (M^3/64L^2)}{(M/2)^2} = \frac{M}{8L}.
  \end{align*}
  Hence
  $\Re(f_k(\bsx+\bsy))\ge \Re(f_k(\bsx))-\frac{M}{8L}\ge\frac{M}{4L}$
  and
  $|\Im(f_k(\bsx+\bsy))|=|\Im(f_k(\bsx+\bsy)-f_k(\bsx))|\le
  \frac{M}{8L}$.
  
  For $k=1$ we use again that $f_1(x)=\hat f_1(x)$ due to
  $\hat f_0\equiv 1$, and thus by \eqref{eq:fj-fj}
  $|f_1(x_1+y_1)-f_1(x_1)|\le \frac{M^3}{64L^2}\le \frac{M}{8L}$ for all
  $x_1\in[-1,1]$ and $y_1\in\cB_{\delta_1}$. We conclude similar as in the
  case $k\ge 2$ that \eqref{eq:Refkge} holds.

  {\bf Step 3.} We show \ref{item:fk:eps}.  %
  Let $k\in\{2,\dots,d\}$,
  $\bsx\in [-1,1]^{k-1}$, $\bsy\in \cB_{\bsdelta_{[k-1]}}$
  and $t\in [-1,1]$.
  By \eqref{eq:fj-fj}-\eqref{eq:fjbound}
  \begin{align}\label{eq:fkxys-fkxs}
    &|f_k(\bsx+\bsy,t)-f_k(\bsx,t)|=\left|\frac{\hat f_k(\bsx+\bsy,t)}{\hat f_{k-1}(\bsx+\bsy)}
      -\frac{\hat f_k(\bsx,t)}{\hat f_{k-1}(\bsx)}\right|\nonumber\\
    &\qquad\qquad\le \left|\frac{\hat f_k(\bsx+\bsy,t)(\hat f_{k-1}(\bsx)-\hat f_{k-1}(\bsx+\bsy))}{\hat f_{k-1}(\bsx+\bsy)\hat f_{k-1}(\bsx)}\right|+\left|\frac{\hat f_{k-1}(\bsx+\bsy)(\hat f_{k}(\bsx+\bsy,t)-\hat f_{k}(\bsx,t))}{\hat f_{k-1}(\bsx+\bsy)\hat f_{k-1}(\bsx)}\right|\nonumber\\
    &\qquad\qquad\le \frac{(\eps_{k-1}+\eps_{k-1}) L}{(M/2)^2}
      = \eps_{k-1} \frac{2L}{(M/2)^2}.
  \end{align}
  The condition on $\eps_{k-1}$ in \eqref{eq:eps} is chosen exactly
  such that the last term is bounded by
  $\min\{1,\delta_k\} \frac{\tilde M^2}{4\tilde M+8\tilde L}$.

  {\bf Step 4.} We show \ref{item:fk:1-fk}.
  Fix $k\in\{2,\dots,d\}$ and $\gamma\in (0,\frac{\delta_k}{2}]$.

  By \eqref{eq:fjbound} and Lemma \ref{lemma:lip} (with
  $K=\cB_\gamma([-1,1])$), for any
  $\bsy\in \cB_{\bsdelta_{[k-1]}}([-1,1])$
  \begin{equation}\label{eq:hfkLip}
    t\mapsto \hat f_k(\bsy,t):\cB_{\gamma}([-1,1])\to\C
    \qquad
    \text{has Lipschitz constant}
    \qquad
    \frac{L}{\delta_k-\gamma}\le
    \frac{2L}{\delta_k}.
  \end{equation}
  By \eqref{eq:tMletL} and Lemma \ref{lemma:lip}, for any $\bsx\in [-1,1]^{k-1}$
  \begin{equation}\label{eq:lip0}
    t\mapsto f_k(\bsx,t):\cB_{\gamma}([-1,1])\to\C
    \qquad
    \text{has Lipschitz constant}\qquad
    \frac{\tilde L}{\delta_k-\gamma} \le \frac{2\tilde L}{\delta_k}.    
  \end{equation}  

  Now, $\int_{-1}^1f_k(\bsx,t)\dd\mu( t)=1$ (see \ref{item:fk:hol})
  and the mean value theorem imply that for every
  $\bsx_{[k-1]} \in [-1,1]^{k-1}$ there exists $x_k\in [-1,1]$
  (depending on $\bsx_{[k-1]}$) such that
  $f_k(\bsx_{[k-1]},x_k)=1$. Any two points in $[-1,1]$ having distance
  at most $2$ and \eqref{eq:lip0} yield
  \begin{equation}\label{eq:fk-11bound}
    |f_k(\bsx)-1|
    \le \frac{4 \tilde L}{\delta_k}\qquad \forall \bsx\in [-1,1]^k.
  \end{equation}
  
  Next let $\bsx\in [-1,1]^{k-1}$ and $\bsy\in \cB_{\bsdelta_{[k-1]}}$
  be arbitrary and fix $t\in \cB_{\gamma}([-1,1])$. Then, using
  \eqref{eq:fjpos}, \eqref{eq:fjbound}, \eqref{eq:hfkLip}, the fact
  that
  ${\hat f_{k-1}}(\bsx_{[k-1]})=\int_{-1}^1\hat
  f_k(\bsx_{[k-1]},s)\dd\mu( s)$ (see \eqref{eq:fk}) and the fact that
  $|s-t|\le 2+\gamma$ for any $s\in [-1,1]$,
  \begin{align}\label{eq:fkxpyt-fkxt}
    &|f_k(\bsx+\bsy,t)-f_k(\bsx,t)|= \left|\frac{\hat f_k(\bsx+\bsy,t)}{\hat f_{k-1}(\bsx+\bsy)}-\frac{\hat f_k(\bsx,t)}{\hat f_{k-1}(\bsx)}\right|\nonumber\\
    &\qquad\qquad \le \frac{\int_{-1}^1|\hat f_k(\bsx+\bsy,t)\hat f_{k}(\bsx,s)-\hat f_k(\bsx,t)\hat f_k(\bsx+\bsy,s) |\dd\mu( s)}{|\hat f_{k-1}(\bsx+\bsy)\hat f_{k-1}(\bsx)|}\nonumber\\
    &\qquad\qquad \le \frac{\int_{-1}^1|\hat f_k(\bsx+\bsy,t)(\hat f_{k}(\bsx,s) - \hat f_k(\bsx,t))|+|\hat f_k(\bsx,t)(\hat f_k(\bsx+\bsy,t) - \hat f_k(\bsx+\bsy,s))|\dd\mu( s)}{|\hat f_{k-1}(\bsx+\bsy)\hat f_{k-1}(\bsx)|}\nonumber\\
    &\qquad\qquad\le \frac{\int_{-1}^1L (2+\gamma)\frac{2L}{\delta_k}+L (2+\gamma)\frac{2L}{\delta_k}\dd\mu( s)}{M^2}\le
      \frac{(2+\gamma)}{\delta_k}\frac{4L^2}{M^2},
  \end{align}
  where we used that $\mu$ is a probability measure. Now additionally
  fix $s\in [-1,1]$ such that $|t-s|<\gamma$. Then by
  \eqref{eq:fkxpyt-fkxt}, \eqref{eq:lip0} and \eqref{eq:fk-11bound}
  \begin{align*}
    |f_k(\bsx+\bsy,t)-1| & \le |f_k(\bsx+\bsy,t)-f_k(\bsx,t)|
                         +|f_k(\bsx,t)-f_k(\bsx,s)|
                         +|f_k(\bsx,s)-1|\nonumber\\
                       & \le \frac{2+\gamma}{\delta_k}\frac{4L^2}{M^2}
                         +\gamma \frac{2\tilde L}{\delta_k} +\frac{4\tilde L}{\delta_k}\nonumber\\
                       & = \frac{2+\gamma}{\delta_k}\left(\frac{4L^2}{M^2}+ 2\tilde L\right)\nonumber\\
                       & = \frac{2+\gamma}{\delta_k}\left(
                         \frac{4L(L+M)}{M^2}\right).
  \end{align*}
  This gives the bound in \ref{item:fk:1-fk}. 
\end{proof}

The next lemma facilitates to determine the analyticity domains of
$\bsx\mapsto F_{\measi;k}(\bsx):[-1,1]^k\to\R$ and
$(\bsx,y)\mapsto (F_{\measii;k}(\bsx,\cdot))^{-1}(y):[-1,1]^{k-1}\times
[0,1]\to [-1,1]$ in \eqref{eq:Fk}, \eqref{eq:Tk}.

\begin{lemma}\label{lemma:FGprop}
  Let $2\le k\in\N$, $O\subseteq\C^{k-1}$ be open and convex,
  $[-1,1]^{k-1}\subseteq O$, $\delta>0$,
  $0<\tilde M\le \tilde L<\infty$,
  \begin{equation}\label{eq:epsilon}
    \epsilon\le \min\{1,\delta\}\frac{\tilde M^2}{4\tilde M+8\tilde L}
  \end{equation}
  and assume that
  \begin{enumerate}[label=(\alph*)]
  \item\label{item:ML} $f\in C^1(O\times \cB_\delta([-1,1]);\C)$,
    $f:[-1,1]^k\to\R_+$ and $\int_{-1}^{1} f(\bsx,t)\dd\mu( t) =1$
    for $\bsx\in [-1,1]^{k-1}$,
  \item\label{item:intf} $\tilde M \le |f(\bsx,t)| \le \tilde L$ for
    all $(\bsx,t)\in O\times \cB_\delta([-1,1])$,
  \item\label{item:supinf}
    $\sup_{\bsy \in O}\inf_{\bsx\in
      [-1,1]^{k-1}}\norm[{L^\infty([-1,1])}]{f(\bsy,\cdot)-f(\bsx,\cdot)}
    <\epsilon$.
  \end{enumerate}
  For $\bsx=(x_i)_{i=1}^k \in O\times \cB_\delta([-1,1])$ let
  $F(\bsx)\dfn \int_{-1}^{x_k} f(\bsx_{[k-1]},t)\dd\mu( t)$.

  Then with $\alpha=\alpha(\tilde M,\tilde L)$,
  $\beta=\beta(\tilde M,\tilde L)$ as in \eqref{eq:rrs}
  \begin{enumerate}
  \item\label{item:FGprop:F}
    for every $\xi\in (0,\delta]$ holds
    $F\in C^1(O\times \cB_\xi([-1,1]); \cB_{\epsilon+\tilde
      L\xi}([0,1]))$,
  \item\label{item:FGprop:G} there exists
    $G\in C^1(O\times \cB_{\frac{\alpha\delta}{2}}([0,1]);
    \cB_{\beta\delta}([-1,1]))$ such that $F(\bsx,G(\bsx,y))=y$ for all
    $(\bsx,y)\in O\times \cB_{\frac{\alpha\delta}{2}}([0,1])$,
  \item\label{item:FGprop:min1}
    $G:O\times \cB_{\frac{\alpha\min\{1,\delta\}}{2}}([0,1])\to
    \cB_{\beta\min\{1,\delta\}}([-1,1])$.
  \end{enumerate}
\end{lemma}
\begin{proof}
  We start with \ref{item:FGprop:F}. Analyticity of
  $F:O\times \cB_\delta([-1,1])\to \C$ is an immediate consequence of
  the analyticity of $f$. By \ref{item:ML}, we have
  $F:[-1,1]^{k}\to [0,1]$.

  Let $\bsy\in O$ and $\xi\in (0,\delta]$.  By \ref{item:supinf} we can find
  $\bsx \in [-1,1]^{k-1}$ such that
  $\sup_{\zeta\in [-1,1]}|f(\bsy,\zeta)-f(\bsx,\zeta)|\le\epsilon$. Fix
  $t\in \cB_\xi([-1,1])$. There exists $s\in [-1,1]$ such that
  $|s-t|<\xi$. Then
  \begin{align*}
    |F(\bsy,t)-F(\bsx,s)|&\le |F(\bsy,t)-F(\bsy,s)|+|F(\bsy,s)-F(\bsx,s)|\\
                         &= \left|\frac{1}{2}\int_{[s,t]} f(\bsy,\zeta)\dd \zeta\right| + \left|\int_{-1}^s f(\bsy,\zeta)-f(\bsx,\zeta)\dd\mu( \zeta) \right|\le \tilde L\xi+\epsilon,
  \end{align*}
  where we used that $|f(\bsy,\zeta)|\le \tilde L$ for all
  $\zeta\in \cB_\delta([-1,1])$. Here
  $\frac{1}{2}\int_{[s,t]}\cdot\dd \zeta$ (for complex $t$) is
  interpreted as a path integral over the path $\gamma(p)=s+p(t-s)$,
  $p\in[0,1]$, and the factor $\frac{1}{2}$ stems from the fact that
  $F(\bsx,\zeta)=\int_{-1}^{\zeta}f(\bsx,z)\dd\mu( z)=
  \frac{1}{2}\int_{-1}^{\zeta}f(\bsx,z)\dd z=
  \frac{1}{2}\int_{[-1,\zeta]} f(\bsx,z)\dd z$ for
  $\zeta\in[-1,1]$. This shows \ref{item:FGprop:F}.

  To show \ref{item:FGprop:G}, first let $\bsx\in
  [-1,1]^{k-1}$. According to Lemma \ref{LEMMA:FEXT}
  \ref{item:Finvhol} there exists
  $G_{\bsx}:\cB_{\alpha\delta}([0,1])\to \cB_{\beta\delta}([-1,1])$
  satisfying $F(\bsx,G_{\bsx}(z))=z$ for all
  $z\in \cB_{\alpha\delta}([0,1])$.
  Now let $\bsy\in O\backslash [-1,1]^{k-1}$. By assumption
  \ref{item:supinf}, we can find $\bsx\in [-1,1]^{k-1}$ such that
  \begin{equation*}
    \sup_{t\in [-1,1]}|f(\bsy,t)-f(\bsx,t)|<
    \epsilon \le \frac{\delta\alpha}{2},
  \end{equation*}
  by \eqref{eq:epsilon} (cp.~\eqref{eq:rrs}). Therefore, Lemma
  \ref{LEMMA:FEXT2} (with ``$f$''$=f(\bsx,\cdot)$ and
  ``$\tilde f$''$=f(\bsy,\cdot)$) yields the existence of
  $G_{\bsy}:\cB_{\frac{\alpha\delta}{2}}([0,1])\to
  \cB_{\beta\delta}([-1,1])$ satisfying $F(\bsy,G_{\bsy}(z))=z$ for
  all $z\in \cB_{\frac{\alpha\delta}{2}}([0,1])$.
  
  Set $G(\bsx,z)\dfn G_{\bsx}(z)$ for all
  $(\bsx,z)\in O\times \cB_{\frac{\alpha\delta}{2}}([0,1])$. We have found
  \begin{equation*}
    G:O\times \cB_{\frac{\alpha\delta}{2}}([0,1])\to\cB_{\beta\delta}([0,1])
    \qquad
    \text{satisfying}\qquad
    F(\bsx,G(\bsx,z))=z.
  \end{equation*}
  By Lemma \ref{lemma:Gcontinuous}, the local inverse functions in
  Prop.~\ref{prop:impl} depend continuously on ``$F$'' (in the there
  stated sense). Since the proofs of Lemma \ref{LEMMA:FEXT} and Lemma
  \ref{LEMMA:FEXT2} stitch together the local inverse functions from
  Prop.~\ref{prop:impl}, one can show that $G(\bsx,z)$ (obtained via
  Lemma \ref{LEMMA:FEXT} and Lemma \ref{LEMMA:FEXT2}) is in fact a
  continuous function of $\bsx$, since $F(\bsx,z)$ and
  $\partial_zF(\bsx,z)$ depend continuously on $\bsx$.  Moreover,
  Lemma \ref{LEMMA:FEXT} and Lemma \ref{LEMMA:FEXT2} state that
  $z\mapsto G_\bsx(z)=G(\bsx,z)$ has Lipschitz constant
  $\frac{1}{\tilde M}$ independent of $\bsx\in O$.  In all this
  implies that $G(\bsx,z)$ is a (jointly) continuous function of
  $(\bsx,z)\in O\times \cB_{\frac{\alpha\delta}{2}}([0,1])$.
  
  By Prop.\ \ref{prop:localhol}, if $G$ is continuous and satisfies
  $F(\bsx,G(\bsx,y))=y$, then $G$ is locally unique and analytic in
  both arguments (here we use that
  $\partial_z F(\bsx,z)=f(\bsx,z)\neq 0$). This shows that
  $G:O\times \cB_{\frac{\alpha\delta}{2}}([0,1])$ is analytic.

  Finally we show \ref{item:FGprop:min1}. If $\delta\le 1$, then
  \ref{item:FGprop:min1} follows from \ref{item:FGprop:G}. If
  $\delta>1$, $f$ also satisfies assumptions
  \ref{item:ML}-\ref{item:supinf} with $\delta$ replaced by
  $\tilde\delta\dfn 1$ (because
  $\cB_{\delta}([-1,1])\supset\cB_{\tilde\delta}$ and
  $\min\{1,\delta\}=1=\min\{1,\tilde\delta\}$ in
  \eqref{eq:epsilon}). Hence \ref{item:FGprop:min1} follows again from
  \ref{item:FGprop:G}.
\end{proof}

\begin{theorem}\label{THM:DINN}
  Let
  \begin{enumerate}[label=(\alph*)]
  \item $f_\measii$ satisfy Assumption \ref{ass:fpi} with
    $\bsdelta_\measii=(\delta_{\measii;j})_{j=1}^d\subset (0,\infty)$,
    $0<M_\measii\le L_\measii<\infty$ and
    $(\eps_{\measii;k})_{k=1}^d\subset [0,\infty)$, and define
    $\tilde M_\measii\dfn \frac{M_\measii}{2L_\measii}\le 1$,
    $\tilde L_\measii\dfn \frac{2L_\measii}{M_\measii}\ge 1$,
  \item $f_\measi$ satisfy Assumption \ref{ass:fpi} with
    $\bsdelta_\measi=(\delta_{\measi;j})_{j=1}^d\subset (0,\infty)$,
    $0\le M_\measi\le L_\measi<\infty$ and
    $(\eps_{\measi;k})_{k=1}^d\subset [0,\infty)$ such that
    (additional to \eqref{eq:eps}) with
    $\alpha=\alpha(\tilde M_\measii,\tilde L_\measii)$ as in
    \eqref{eq:rrs}
    \begin{equation}\label{eq:epseta}
      0\le \eps_{\measi;k} \le
      \frac{\alpha M_\measi}{32L_\measi} \min\left\{1,
        \delta_{\measii;k+1}
      \right\}
      \qquad\forall k\in\{1,\dots,d-1\}
    \end{equation}
    and set $\tilde M_\measi\dfn \frac{M_\measi}{2 L_\measi}$ and
    $\tilde L_\measi\dfn \frac{2L_\measi}{M_\measi}$.
  \end{enumerate}

  Define $\bszeta=(\zeta_k)_{k=1}^d$ via
  \begin{equation}\label{eq:zetak}
    \zeta_{k}\dfn \min\left\{\delta_{\measi;k} %
      ,\delta_{\measii;k} \min\left\{1,\frac{\alpha(\tilde M_\measii,\tilde L_\measii)}{4 \tilde L_{\measi}}\right\} \right\} \qquad \forall k\in\{1,\dots,d\}.
  \end{equation}
  Then, for every $k\in\{1,\dots,d\}$
  \begin{enumerate}
  \item\label{item:dinN:Tk} $T_k:[-1,1]^k\to [-1,1]$ in \eqref{eq:Tk}
    allows an extension
    \begin{equation}\label{eq:Tkhol}
      T_k\in C^1(\cB_{\bszeta_{[k]}}([-1,1]);
      \cB_{\delta_{\measii;k}}([-1,1])),
    \end{equation}
  \item\label{item:dinN:Rk}
    $R_k(\bsx)\dfn \partial_{k} T_k(\bsx)$ satisfies with
    $C_3\dfn \frac{4L_\measii L_\measi}{ M_\measii M_\measi}$ and
    $C_4\dfn \frac{3M_\measii^3 M_\measi}{256 L_\measii^3 L_\measi}$
    \begin{subequations}\label{eq:Rkext}
      \begin{equation}\label{eq:Rkext1}
        R_k\in C^1(\cB_{\bszeta_{[k]}}([-1,1]);
        \cB_{C_3}(1))\qquad\text{s.t.}\qquad
        \inf_{\bsx\in \cB_{\bszeta_{[k]}}([-1,1])}\Re(R_k(\bsx))
        \ge C_4
      \end{equation}
      and there exists
      $C_5=C_5(M_\measi,M_\measii,L_\measi,L_\measii)$ such that %
      if $k\ge 2$
      \begin{equation}\label{eq:Rkext2}
        R_k:\cB_{\bszeta_{[k-1]}}([-1,1])\times [-1,1]\to
        \cB_{\frac{C_5}{\min\{\delta_{\measii;k},\delta_{\measi;k}\}}}(1).
      \end{equation}
    \end{subequations}
  \end{enumerate}
\end{theorem}

\begin{proof}
  {\bf Step 1.} We establish notation and preliminary results.
  Throughout this proof (as in \eqref{eq:rrs})
    \begin{equation}\label{eq:thmalphabeta}
      \alpha=\alpha(\tilde M_\measii,\tilde L_\measii)= \frac{\tilde M_\measii^2}{2\tilde M_\measii+4\tilde L_\measii},\qquad
      \beta=\beta(\tilde M_\measii,\tilde L_\measii)= \frac{\tilde M_\measii}{2\tilde M_\measii+4\tilde L_\measii}.
    \end{equation}
  
  In the following for $k\in\{1,\dots,d\}$ and $*\in\{\measi,\measii\}$
  \begin{equation*}
    \cB_{\bsdelta_{*;{[k]}}}([-1,1])\dfn \bigtimes_{j=1}^k
    \cB_{\delta_{*;j}}([-1,1])\subseteq\C^k.
  \end{equation*}
  For $k\in\{1,\dots,d\}$ let
  \begin{equation*}
    f_{\measii;k}: \cB_{\bsdelta_{\measii;{[k]}}}([-1,1])\to\C,\qquad
    f_{\measi;k}:\cB_{\bsdelta_{\measi;{[k]}}}([-1,1])\to\C
  \end{equation*}
  be as in \eqref{eq:fk}.  Lemma \ref{lemma:fk}
  \ref{item:fk:hol}-\ref{item:fk:eps} states that for
  $k\in\{2,\dots,d\}$, the functions
  $f_{\measii;k}\in C^1(\cB_{\bsdelta_{\measii;[k-1]}}\times
  \cB_{\delta_{\measii;k}};\C)$ and
  $f_{\measi;k}\in C^1(\cB_{\bsdelta_{\measi;[k-1]}}\times
  \cB_{\delta_{\measi;k}};\C)$ satisfy the assumptions of Lemma
  \ref{lemma:FGprop}, with the constants
  $0<\tilde M_\measi \le \tilde L_\measi<\infty$,
  $0<\tilde M_\measii \le \tilde L_\measii<\infty$, and
  \begin{equation}
    \begin{aligned}\label{eq:epsiloneta}
    \epsilon_{\measi;k}&\dfn \eps_{\measi;k-1} \frac{8 L_\measi}{M_\measi}
    \le \min\{1,\delta_{\measi;k}\}\frac{\tilde M_\measi^2}{4\tilde M_\measi + 8\tilde L_\measi},\\
    \epsilon_{\measii;k}&\dfn \eps_{\measii;k-1} \frac{8 L_\measii}{M_\measii}
    \le \min\{1,\delta_{\measii;k}\}\frac{\tilde M_\measii^2}{4\tilde M_\measii + 8\tilde L_\measii}=\min\{1,\delta_{\measii;k}\}\frac{\alpha}{2}
    \end{aligned}
  \end{equation}
  in \eqref{eq:epsilon} and assumption \ref{item:supinf} of Lemma
  \ref{lemma:FGprop} (cp.~\eqref{eq:eps} and \eqref{eq:fksupinf}) with
  $O=\cB_{\bsdelta_{\measi;[k-1]}}([-1,1])$,
  $O=\cB_{\bsdelta_{\measii;[k-1]}}([-1,1])$ respectively.

  By Lemma \ref{lemma:FGprop} \ref{item:FGprop:F}, for
  $k\in\{2,\dots,d\}$ the functions $F_{\measi;k}$,
  $F_{\measii;k}$ as in \eqref{eq:Fk} are well-defined,
  and in particular $F_{\measi;k}$ is holomorphic from
  \begin{equation}\label{eq:Fetakhol}
    F_{\measi;k}:\cB_{\bsdelta_{\measi;{[k-1]}}}([-1,1])
    \times \cB_{\xi}([-1,1])\to \cB_{\epsilon_{\measi;k}+\tilde L_\measi\xi}([0,1])\qquad
    \forall \xi\in (0,\delta_{\measi;k}].
  \end{equation}

  For $k\in \{2,\dots,d\}$ let $G_{\measii;k}$ be as in Lemma
  \ref{lemma:FGprop} \ref{item:FGprop:G} (w.r.t.~$F_{\measii;k}$).  By
  Lemma \ref{lemma:FGprop} \ref{item:FGprop:G} for every
  $k\in\{2,\dots,d\}$, this map is holomorphic between
  \begin{equation}\label{eq:Gpikhol0}
    G_{\measii;k}:\cB_{\bsdelta_{\measii;[k-1]}}([-1,1])\times \cB_{\frac{\alpha\delta_{\measii;k}}{2}}([0,1])\to \cB_{\beta\delta_{\measii;k}}([-1,1]),
  \end{equation}
  and by Lemma \ref{lemma:FGprop} \ref{item:FGprop:min1}
  \begin{equation}\label{eq:Gpikhol}
    G_{\measii;k}:\cB_{\bsdelta_{\measii;[k-1]}}([-1,1])\times \cB_{\frac{\alpha\min\{1,\delta_{\measii;k}\}}{2}}([0,1])\to \cB_{\beta\min\{1,\delta_{\measii;k}\}}([-1,1]).
  \end{equation}
  {\bf Step 2.} We show \ref{item:dinN:Tk}.  By definition of $T_k$ in
  \eqref{eq:Tk}, for $k\in\{2,\dots,d\}$ and $\bsx\in [-1,1]^k$
  \begin{equation}\label{eq:Tk2}
    T_k(\bsx) = G_{\measii;k}\big(T_{[k-1]}(\bsx_{[k-1]}),F_{\measi;k}(\bsx)\big),
  \end{equation}
  {where we recall the notation $T_{[k-1]}=(T_j)_{j=1}^{k-1}$.}
  We show
  by induction \eqref{eq:Tkhol}, i.e.,
  \begin{equation*}
    T_k\in C^1(\cB_{\bszeta_{[k]}}([-1,1]); \cB_{\delta_{\measii;k}}([-1,1]))
    \qquad\forall k\in\{1,\dots,d\}.
  \end{equation*}

  Let $k=1$. By Lemma \ref{lemma:fk} \ref{item:fk:hol}
  \begin{equation*}
    f_{\measi;1}:\cB_{\delta_{\measi;1}}\to \C,\qquad
    f_{\measii;1}:\cB_{\delta_{\measii;1}}\to \C
  \end{equation*}
  are holomorphic functions satisfying \eqref{eq:tMletL} with the
  corresponding constants $\tilde M_\measi\le\tilde L_\measi$,
  $\tilde M_\measii\le\tilde L_\measii$. Thus by Prop.~\ref{prop:1d} with
  $r=\min\{\delta_{\measi;1},\frac{\delta_{\measii;1}\tilde M_\measii^2}{\tilde
    L_\measi(2\tilde M_\measii+4\tilde L_\measii)}\}$
  \begin{equation*}
    T_1:\cB_{r}([-1,1])\to \cB_{r\frac{\tilde L_{\measi}}{\tilde M_\measii}}([-1,1])
  \end{equation*}
  is holomorphic. Since $r\ge \zeta_1$ and
  $\frac{r\tilde L_\measi}{\tilde M_\measii}\le \delta_{\measii;1}$,
  this shows \eqref{eq:Tkhol} for $k=1$.

  For the induction step, we first note that there hold the following
  inequalities for $k\ge 2$:
  \begin{enumerate}[label=(i\arabic*)]
  \item\label{item:proof1}
    $\zeta_k\le\min\{\delta_{\measii;k},\delta_{\measi;k}\}$: This is
    immediate from the definition of $\zeta_k$ in \eqref{eq:zetak}.
  \item\label{item:proof2}
    $\epsilon_{\measi;k} + \tilde L_\measi \zeta_k\le
    \frac{\alpha\delta_{\measii;k}}{2}$: %
    \eqref{eq:epsiloneta} and \eqref{eq:epseta} give
    \begin{equation}\label{eq:epsmeasik}
      \epsilon_{\measi;k}\le \eps_{\measi;k-1} \frac{8L_\measi}{M_\measi}
      \le \frac{\alpha M_\measi}{32 L_\measi}\frac{8L_\measi}{M_\measi}
      \min\{1,\delta_{\measii;k}\}
      = \frac{\alpha\min\{1,\delta_{\measii;k}\}}{4},
    \end{equation}
    and %
    by \eqref{eq:zetak} it holds
    $\tilde L_\measi \zeta_k\le \frac{\alpha\delta_{\measii;k}}{4}$.
  \end{enumerate}

  Let $k\in\{2,\dots,d\}$. %
  Assume that
  \begin{equation}\label{eq:indhop}
    T_j\in C^1(\cB_{\bszeta_{[j]}}([-1,1]);
  \cB_{\delta_{\measii;j}}([-1,1]))\qquad\forall
  j\in\{1,\dots,k-1\},
  \end{equation}
  which is the induction hypothesis.
  We show that it also holds for $k$. By
  \eqref{eq:Fetakhol}, \ref{item:proof1} and \ref{item:proof2}
  \begin{equation*}
    F_{\measi;k}(\bsx)\in \cB_{\epsilon_{\measi;k} + \tilde L_\measi \zeta_k}([-1,1])
    \subseteq \cB_{\frac{\alpha\delta_{\measii;k}}{2}}([-1,1])
    \qquad \forall
    \bsx\in \cB_{\bszeta_{[k]}}([-1,1]).
  \end{equation*}
  Due to \eqref{eq:Gpikhol0}, \eqref{eq:Tk2} {and the induction
    hypothesis \eqref{eq:indhop}}
  we get
  \begin{equation}\label{eq:step2Tk}
    T_k(\bsx)=
    G_{\measii;k}\big(\underbrace{T_1(x_1)}_{\in\cB_{\delta_{\measii;1}}([0,1])},\dots,
    \underbrace{T_{k-1}(\bsx_{[k-1]})}_{\in\cB_{\delta_{\measii;k-1}}([0,1])}
    ,~~\underbrace{F_{\measi;k}(\bsx)}_{\in\cB_{\frac{\alpha\delta_{\measii;k}}{2}}([-1,1])}\big)
    \in \cB_{\beta\delta_{\measii;k}}\qquad
    \forall \bsx\in \cB_{\bszeta_{[k]}}([-1,1]).
  \end{equation}
  Since $\beta\le 1$ by definition, this shows \eqref{eq:Tkhol} for
  $T_k$.

  {\bf Step 3.} We show \eqref{eq:Rkext1} and first verify that for
  all $k\in\{1,\dots,d\}$
  \begin{equation}\label{eq:Rk}
    R_k(\bsx) = \frac{f_{\measi;k}(\bsx)}{f_{\measii;k}(T_{[k]}(\bsx))}\qquad
    \forall \bsx \in \cB_{\bszeta_{[k]}}([-1,1]).
  \end{equation}
  
  First let $\bsx\in[-1,1]^k$. By \eqref{eq:Tk2}
  \begin{equation}\label{eq:pxkTk}
    \partial_{k} T_k(\bsx)= \partial_{k} G_{\measii;k}(T_{[k-1]}(\bsx_{[k-1]}),F_{\measi;k}(\bsx)) \partial_{k} F_{\measi;k}(\bsx).
  \end{equation}
  Now
  \begin{equation}\label{eq:xkeqGF}
    x_k=G_{\measii;k}(\bsx_{[k-1]},F_{\measii;k}(\bsx))%
  \end{equation}
  and thus %
  \begin{equation*}
    1 = \partial_{k} \left( G_{\measii;k}(\bsx_{[k-1]},F_{\measii;k}(\bsx))\right)=\partial_{k} G_{\measii;k}(\bsx_{[k-1]},F_{\measii;k}(\bsx)) \partial_{k} F_{\measii;k}(\bsx) .%
  \end{equation*}
  Using that $F_{\measii;k}(\bsx_{[k-1]},\cdot):[-1,1]\to [0,1]$ is
  bijective, the substitution $y_k=F_{\measii;k}(\bsx)$ and
  \eqref{eq:xkeqGF} give for all
  $(\bsx_{[k-1]},y_k)\in [-1,1]^{k-1}\times [0,1]$
  \begin{equation*}
    \partial_{k} G_{\measii;k}(\bsx_{[k-1]},y_k)
    = \frac{1}{\partial_{k}F_{\measii;k}(\bsx_{[k-1]},G_{\measii;k}(\bsx_{[k-1]},y_k))}.
  \end{equation*}
  Hence, since $\partial_{k}F_{\measi;k}=f_{\measi;k}$ and
  $\partial_{k}F_{\measii;k}=f_{\measii;k}$, we obtain by \eqref{eq:pxkTk}
  \begin{equation*}
    \partial_{k} T_k(\bsx)=\frac{f_{\measi;k}(\bsx)}{f_{\measii;k}(T_{[k-1]}(\bsx_{[k-1]}),G_{\measii;k}(T_{[k-1]}(\bsx_{[k-1]}),F_{\measi;k}(\bsx)))},
  \end{equation*}
  which by \eqref{eq:Tk2} shows \eqref{eq:Rk} for $\bsx \in [-1,1]^k$.
  The identity theorem for holomorphic functions implies that
  \eqref{eq:Rk} holds for all $\bsx\in\cB_{\bszeta_{[k]}}([-1,1])$.
  
  Now we show \eqref{eq:Rkext1}. By Lemma \ref{lemma:fk}
  \ref{item:fk:tMtL} for $*\in\{\measi,\measii\}$
  \begin{equation*}
    \Re(f_{*;k}(\bsx))\ge \frac{M_*}{4L_*}\quad\text{and}\quad
    |\Im(f_{*;k}(\bsx))|\le \frac{M_*}{8L_*}\qquad
    \forall \bsx\in \cB_{\bsdelta_{*;[k]}}([-1,1]).
  \end{equation*}
  Moreover
  $|f_{\measii;k}(\bsx)|\le\tilde L_\measii=\frac{2
    L_\measii}{M_\measii}$. Thus, for arbitrary
  $\bsx\in \cB_{\bszeta_{[k]}}([-1,1])\subseteq
  \cB_{\bsdelta_{\measii;[k]}}([-1,1])$, writing
  $f_{\measi;k}(\bsx)=a+\ii b$ and
  $f_{\measii;k}(T_{[k]}(\bsx))=c+\ii d$
  \begin{equation*}
    \Re(R_k(\bsx)) = \Re\left(\frac{a+\ii b}{c+\ii d}\right)
    = \Re\left(\frac{a+\ii b}{c+\ii d} \frac{c-\ii d}{c-\ii d}\right)
    = \frac{ac+bd}{c^2+d^2}
    \ge \frac{\frac{M_\measii}{4L_\measii}\frac{M_\measi}{4L_\measi} - \frac{M_\measii}{8L_\measii}\frac{M_\measi}{8L_\measi}}{\tilde L_\measii^2} = \frac{3 M_\measii^3 M_\measi}{256 L_\measii^3 L_\measi}.
  \end{equation*}
  Moreover by Lemma \ref{lemma:fk} \ref{item:fk:tMtL} and
  \eqref{eq:Rk} we have
  $|R_k(\bsx)|\le (\frac{2L_\measi}{
    M_\measi})/(\frac{M_\measii}{2L_\measii})$ for all
  $\bsx\in \cB_{\bszeta_{[k]}}([-1,1])$, which gives
  \eqref{eq:Rkext1}.

  {\bf Step 4.} We show \eqref{eq:Rkext2}. Fix $k\in\{2,\dots,d\}$ and
  let $\xi\in (0,\zeta_k)$ be so small that
  $\epsilon_{\measi;k}+L_\measi\xi<
  \frac{\alpha\min\{1,\delta_{\measii;k}\}}{2}$, which is possible by
  \eqref{eq:epsmeasik}. Let
  \begin{equation}
    \bsx\in\cB_{\bszeta_{[k-1]}}([-1,1])\times \cB_{\xi}([-1,1])
    \subset \cB_{\bszeta_{[k]}}([-1,1])
  \end{equation}
  arbitrary. By \eqref{eq:Tkhol}
  \begin{equation*}
    T_j(\bsx)\in\cB_{\delta_{\measii;j}}([0,1])\qquad\forall j\in\{1,\dots,k\}.
  \end{equation*}
  By
  \eqref{eq:Fetakhol}
  \begin{equation*}
    F_{\measi;k}(\bsx)
    \in\cB_{\epsilon_{\measi;k}+\tilde L_\measi\xi}([0,1])
    \subseteq \cB_{\frac{\alpha\min\{1,\delta_{\measii;k}\}}{2}}([0,1]).
  \end{equation*}
  By \eqref{eq:Gpikhol}, \eqref{eq:Tk2}
  \begin{equation*}
    T_k(\bsx)=
    G_{\measii;k}\big(\underbrace{T_1(x_1)}_{\in\cB_{\delta_{\measii;1}}([-1,1])},\dots,
    \underbrace{T_{k-1}(\bsx_{[k-1]})}_{\in\cB_{\delta_{\measii;k-1}}([-1,1])}
    ,~~\underbrace{F_{\measi;k}(\bsx)}_{\in\cB_{\frac{\alpha\min\{1,\delta_{\measii;k}\}}{2}}([-1,1])}\big)
    \in \cB_{\beta\min\{1,\delta_{\measii;k}\}},
  \end{equation*}
  and since $\beta\le\frac{1}{2}$ (cp.~\eqref{eq:thmalphabeta}),
  $T_k(\bsx)\in \cB_{\frac{\min\{1,\delta_{\measii;k}\}}{2}}([-1,1])$.

  {For the rest of the proof fix}
  $\bsx\in\cB_{\bszeta_{[k-1]}}([-1,1])\times [-1,1]$. Lemma
  \ref{lemma:fk} \ref{item:fk:1-fk} implies
  \begin{equation*}
    f_{\measii;k}(\underbrace{T_1(x_1)}_{\in\cB_{\delta_{\measii;1}}([-1,1])},\dots,
    \underbrace{T_{k-1}(\bsx_{[k-1]})}_{\in\cB_{\delta_{\measii;k-1}}([-1,1])},
    \underbrace{T_k(\bsx)}_{\in\cB_{\frac{\min\{1,\delta_{\measii;k}\}}{2}}([-1,1])})\in \cB_{\frac{2+\min\{1,\delta_{\measii;k}\}/2}{\delta_{\measii;k}}C_2}(1)
    \subseteq \cB_{\frac{3C_2}{\delta_{\measii;k}}}(1),
  \end{equation*}
  with $C_2=C_2(M_\measii,L_\measii)$ as in Lemma \ref{lemma:fk}
  \ref{item:fk:1-fk}.  Lemma \ref{lemma:fk} \ref{item:fk:1-fk} (and
  $\zeta_j\le\delta_{\measi;j}$, $j=1,\dots,k-1$) also gives
  \begin{equation*}
    f_{\measi;k}(\bsx)\in \cB_{\frac{3C_2}{\delta_{\measi;k}}}(1),
  \end{equation*}
  with $C_2=C_2(M_\measi,L_\measi)$ as in Lemma \ref{lemma:fk}
  \ref{item:fk:1-fk}.
  
  Write
  $f_{\measii;k}(T_{[k]}(\bsx))=1+z_1$ and
  $f_{\measi;k}(\bsx)=1+z_2$ for $z_1$, $z_2\in\C$ with
  \begin{equation}\label{eq:xi1xi2}
    |z_1|\le \frac{3C_2(M_\measi,L_\measi)}{\delta_{\measi;k}}\qquad\text{and}\qquad
    |z_2|\le \frac{3C_2(M_\measii,L_\measii)}{\delta_{\measii;k}}.
  \end{equation}
  We distinguish between two cases and assume first that
  $|z_1|\le \frac{1}{2}$.
  Then by \eqref{eq:Rk} (cp.~\eqref{eq:zetak} and \eqref{eq:xi1xi2})
  \begin{align}\label{eq:Rk-1firstcase}
    |R_k(\bsx)-1| &= \left|\frac{1+z_2}{1+z_1}-\frac{1+z_1}{1+z_1}\right|\le \left | \frac{z_1-z_2}{1/2} \right|\nonumber\\
                  &\le
                    4\max\{|z_1|,|z_2|\}
                    \le 4 \frac{3\max\{C_2(M_\measi,L_\measi),C_2(M_\measii,L_\measii)\}}{\min\{|\delta_{\measi;k}|,|\delta_{\measii;k}|\}}.
  \end{align}

  In the second case %
  we have $|z_1|>\frac{1}{2}$. By \eqref{eq:Rkext1}
  $|R_k(\bsx)|\le 1+C_3$ and thus %
  with \eqref{eq:xi1xi2}
  \begin{equation}\label{eq:Rk-1secondcase}
    |R_k(\bsx)-1|\le 2+C_3=\frac{4+2C_3}{2}\le
                   (4+2C_3)|z_1|\le (4+2C_3)\frac{3C_2(M_\measi,L_\measi)}{\delta_{\measi;k}}.
  \end{equation}
  With
  \begin{equation}\label{eq:C5}
    C_5\dfn 3(4+2C_3)\max\{C_2(M_\measi,L_\measi),C_2(M_\measii,L_\measii)\},
  \end{equation}
  \eqref{eq:Rk-1firstcase} and \eqref{eq:Rk-1secondcase} show
  \eqref{eq:Rkext2}.
\end{proof}

We are now in position to prove Thm.~\ref{THM:COR:DINN}.
\begin{proof}[Proof of Thm.~\ref{THM:COR:DINN}]
  Let $0<M\le L<\infty$ and $\bsdelta\in (0,\infty)^d$ be such that
  $f_\measi\in C^1(\cB_\bsdelta([-1,1]);\C)$ and
  $ f_\measii\in C^1(\cB_\bsdelta([-1,1]);\C)$ both satisfy Assumption
  \ref{ass:finite} with these constants. Upon choosing $C_6=C_6(M,L)$
  in Assumption \ref{ass:finite} small enough, we show that
  $f_\measi$, $f_\measii$ satisfy Assumption \ref{ass:fpi} with the
  additional constraint \eqref{eq:epseta}.  This means that we can
  apply Thm.~\ref{THM:DINN}, which immediately implies
  Thm.~\ref{THM:COR:DINN} {with (cp.~\eqref{eq:rrs})
  \begin{equation}\label{eq:C7}
    C_7\dfn \min\left\{1,\frac{\alpha(\tilde M,\tilde L)}{4\tilde L}\right\}\qquad\text{and}\qquad C_8\dfn \max\left\{C_3,\frac{1}{C_4},C_5\right\}
  \end{equation}
  and where $\tilde M\dfn \frac{M}{2L}$, $\tilde L\dfn \frac{2L}{M}$
  and $C_3$, $C_4$ and $C_5$ are as in Thm.~\ref{THM:DINN}.  }

  Assumption \ref{ass:fpi} \ref{item:fpi:1}, \ref{item:fpi:3}
  holds by Assumption \ref{ass:finite} \ref{item:cordinN:1},
  \ref{item:cordinN:2} with $M_\measi=M_\measii=M$ and
  $L_\measi=L_\measii=L$.

  With $\alpha=\frac{M^2}{2M+4L}$ holds
  $\frac{\alpha M}{32L}=\frac{M^3}{64(ML+2L^2)}$. Therefore, spelling out
  Assumption \ref{ass:fpi} \ref{item:fpi:eps} with the additional
  constraint \eqref{eq:epseta}, for $k\in\{1,\dots,d-1\}$ we require
  \begin{subequations}\label{eq:complicated_assumption}
  \begin{equation}\label{eq:complicated_assumption_a}
    \sup_{\bsx\in [-1,1]^d}\sup_{\bsy\in\cB_{\bsdelta_{[k]}}\times\{0\}^{d-k}}
    |f(\bsx+\bsy)-f(\bsx)|
    \le \min\{1,\delta_{k+1}\}
    \min\left\{\frac{M^3}{64(ML+2L^2)},
      \frac{M^5}{512 L^4 + 64L^2M^2}
    \right\}
  \end{equation}
  and
  \begin{equation}
    \sup_{\bsx\in [-1,1]^d}\sup_{\bsy\in\cB_{\bsdelta}}
    |f(\bsx+\bsy)-f(\bsx)|
    \le \frac{M^3}{64L^2}.
  \end{equation}
  \end{subequations}
  
  Define
  \begin{equation}\label{eq:C6}
    C_6\dfn \min\left\{\frac{M^3}{64(ML+2L^2)},
      \frac{M^5}{512 L^4 + 64L^2M^2}
    \right\}\le \frac{M^3}{64L^2}.
  \end{equation}
  Then
  \begin{subequations}\label{eq:simple_assumption}
  \begin{equation}
    \sup_{\bsx\in [-1,1]^d}\sup_{\bsy\in\cB_{\bsdelta_{[k]}}\times\{0\}^{d-k}}
    |f(\bsx+\bsy)-f(\bsx)|\le C_6\min\{1,\delta_{k+1}\}\qquad\forall k\in\{1,\dots,d-1\}
  \end{equation}
  and
  \begin{equation}\label{eq:simple_assumption_b}
    \sup_{\bsx\in [-1,1]^d}\sup_{\bsy\in\cB_{\bsdelta}}
    |f(\bsx+\bsy)-f(\bsx)|\le C_6
  \end{equation}
  \end{subequations}  
  implies \eqref{eq:complicated_assumption}. This concludes the proof,
  as Assumption \ref{ass:finite} \ref{item:cordinN:3},
  \ref{item:cordinN:4} corresponds to \eqref{eq:simple_assumption}
  (note that if $\min\{1,\delta_{k+1}\}=1$,
  then \eqref{eq:complicated_assumption_a} follows
  by \eqref{eq:simple_assumption_b}).
\end{proof}

\section{Proofs of Sec.~\ref{SEC:APPROX}}
\subsection{Lemma \ref{LEMMA:LEGEST}}
{
\begin{proof}[Proof of Lemma \ref{LEMMA:LEGEST}]
  We start with \eqref{eq:lknubound} for $k=1$, and denote the
  Legendre coefficients of a function $g$ by
  $l_{g,j}=\int_{[-1,1]}g(x)L_j(x)\dd\mu(x)$, $j\in\N_0$.

  For $\xi>1$ introduce the Bernstein ellipse
\begin{equation*}
  \cE_\xi\dfn \setc{\frac{z+z^{-1}}{2}}{z\in\C,~1\le|z|\le\xi}\subset\C.
\end{equation*}
If $g\in C^1(\cE_\xi)$ then
$|l_{g,j}|\le
\xi^{-j}(1+2j)^{3/2}\norm[{L^\infty(\cB_\delta([-1,1]))}]{g}\frac{2\xi}{\xi-1}$
for all $j\in\N$. The proof of this classical estimate can be found
for example in \cite{davis} (see equations (12.4.24)-(12.4.26); also
note that our normalization $\norm[{L^2([-1,1],\mu)}]{L_j}=1$ gives
another factor $(1+2j)^{1/2}$ compared to the discussion in
\cite{davis}, cp.~Rmk.~\ref{rmk:WkinfLnu}).

The Bernstein ellipse $\cE_\xi$ has semiaxes $\frac{\xi+\xi^{-1}}{2}$
and $\frac{\xi-\xi^{-1}}{2}$. Solving $\frac{\xi-\xi^{-1}}{2}=\delta$
and $\frac{\xi+\xi^{-1}}{2}-1=\delta$ for $\xi$, we find that the
largest ellipse $\cE_\xi$ contained in $\cB_\delta([-1,1])$ is
obtained for $\xi=\delta+\sqrt{\delta^2+1}\ge 1+\delta$. In particular
$\cE_{1+\delta}\subseteq \cB_\delta([-1,1])$ for all
$\delta>0$. Hence, if $f\in\cB_\delta([-1,1])$, then with
$\varrho\dfn 1+\delta$,
$|l_{f,j}|\le
\varrho^{-k}(1+2k)^{3/2}\norm[{L^\infty(\cB_\delta([-1,1]))}]{f}\frac{\varrho}{\varrho-1}$. This
shows \eqref{eq:lknubound} for $k=1$. The bound \eqref{eq:lknubound}
for general $k$ follows by applying the same argument componentwise,
see for example the appendix in \cite{chkifa} or
\cite[Cor.~B.2.7]{JZdiss}.

It remains to show \eqref{eq:lknubound2} for
$\bsnu\in\N_0^{k-1}\times\{0\}$. Since $L_{\nu_k}(y_k)=L_0(y_k)=1$ we
have
  \begin{align*}
    \left|\int_{[-1,1]^{k}} f(\bsy) L_\bsnu(\bsy)\dd\mu(\bsy) \right|
    &=\left|\int_{-1}^1\int_{[-1,1]^{k-1}} f(\bsy_{[k-1]},y_{k}) L_{\bsnu_{[k-1]}}(\bsy_{[k-1]})\dd\mu(\bsy_{[k-1]})\dd\mu( y_{k}) \right|\nonumber\\
    &\le\max_{y_{k}\in[-1,1]}\left|\int_{[-1,1]^{k-1}} f(\bsy_{[k-1]},y_{k}) L_{\bsnu_{[k-1]}}(\bsy_{[k-1]})\dd \mu (\bsy_{[k-1]})\right|.
  \end{align*}
  Thus \eqref{eq:lknubound2} follows by \eqref{eq:lknubound}.
\end{proof}}%

\subsection{%
  Prop.~\ref{PROP:KAPPROX}}\label{app:legendre}

  {The exponential convergence of Prop.~\ref{PROP:KAPPROX} is verified by
  bounding the approximation error of the truncated Legendre expansion
  $f(\bsy)\simeq \sum_{\bsnu\in\Lambda_{k,\eps}}l_\bsnu L_\bsnu(\bsy)$ %
  by
  $\sum_{\bsnu\in\N_0^k\backslash \Lambda_{k,\eps}} |l_\bsnu|
  \norm[{W^{m,\infty}([-1,1]^k)}]{L_\bsnu}$, which is an upper bound
  on the remainder. As we will see, this classical line of arguments
  requires the following ingredients: (a) an upper bound on the
  Legendre coefficients, (b) an upper bound on the norm of the
  Legendre polynomials $L_\bsnu$, and (c) an upper bound on the
  cardinality of $\Lambda_{k,\eps}$. The first is obtained in Lemma
  \ref{LEMMA:LEGEST} and the second in Rmk.~\ref{rmk:WkinfLnu} ahead.
  Before coming to the proof of Prop.~\ref{PROP:KAPPROX},
  it then remains to discuss the third.}

\begin{remark}\label{rmk:WkinfLnu}
  For all
  $\bsnu\in \N_0^k$ 
  holds
  $\norm[{L^{\infty}([-1,1]^k;\R)}]{L_\bsnu}\le \prod_{j=1}^k
  (1+2\nu_j)^{1/2}$; see \cite[\S 18.2(iii) and \S
  18.3]{nist}. By the Markov brothers' inequality (see
  \cite{MR1511855} and, e.g., the references in
  \cite[p.~228]{cheney1966introduction}), this generalizes to
  $W^{m,\infty}$ via
\begin{equation}\label{eq:WkinfLnu}
  \norm[{W^{m,\infty}([-1,1]^k;\R)}]{L_\bsnu}\le \prod_{j=1}^k (1+2\nu_j)^{\frac 1 2 + 2m}\qquad\forall \bsnu\in\N_0^k,~m\in\N_0.
\end{equation}
\end{remark}

  A first simple bound on $|\Lambda_{k,\eps}|$ in \eqref{eq:Leps} is obtained as follows:
  With $\varrho_{\rm min}\dfn\min_j\varrho_j>1$ it holds
  $\gamma(\bsvarrho,\bsnu)\le \varrho_{\rm min}^{-|\bsnu|}$ and thus
  $\set{\bsnu\in\N_0^k}{\varrho_{\rm
      min}^{-|\bsnu|}\ge\eps}\supseteq \Lambda_{k,\eps}$. Due to
  $|\set{\bsnu\in\N_0^k}{|\bsnu|\le a}|\le (1+a)^k$ for all $a\ge 0$,
  for $\eps\in (0,1)$ we conclude with
  $a=-\frac{\log(\eps)}{\log(\varrho_{\rm min})}$
  \begin{equation}\label{eq:Lepssimple}
    |\Lambda_{k,\eps}|\le \left(1-\frac{\log(\eps)}{\log(\varrho_{\rm min})}\right)^k.
  \end{equation}  
  The following slightly more involved bound holds by \cite[Lemma
  3.3]{OSZ19} and because $\gamma(\bsvarrho,\bsnu)\le \bsvarrho^{-\bsnu}$ for
  all $\bsnu\in\N_0^k$ (also see \cite{MR2280784} and \cite[Lemma
  4.2]{BonitoEtAl2019}).
\begin{lemma}\label{lemma:lepsbound}
  Let $k\in\N$ and $\bsvarrho\in(1,\infty)^k$. It holds for
  $\eps\in (0,1)$
  \begin{equation}\label{eq:Lepsbound0}
    |\Lambda_{k,\eps}|\le \left|\setc{\bsnu\in\N_0^k}{\bsvarrho^{-\bsnu}\ge\eps}\right|\le \frac{1}{k!} \left(-\log(\eps)+\sum_{j=1}^k\log(\varrho_j) \right)^k \prod_{j=1}^k \frac{1}{\log(\varrho_j)}.
  \end{equation}
\end{lemma}

 \proofref{The following proposition is proven in
  Appendix \ref{app:legendre}.}

The next lemma will be required in the proof of Prop.~\ref{PROP:KAPPROX}.
In the following for
$\bsvarrho\in (1,\infty)^k$ and $\theta>0$ set
\begin{equation}\label{eq:Sxi}
  S(\bsvarrho,\theta)\dfn \sum_{\bsmu\in\N_0^k}\bsvarrho^{-\bsmu}\prod_{j=1}^k(1+2\nu_j)^{\theta}.
\end{equation}
For any $\theta>0$ holds $S(\bsvarrho,\theta)<\infty$; {see, e.g.,
\cite[Lemma 3.10]{ZS17}.}

\begin{lemma}\label{LEMMA:ETASUMEST}
  Let $k$, $m\in\N$, $\theta>0$, $C_0>0$ and $\bsvarrho\in
  (1,\infty)^k$. Assume that
  $f(\bsy)=\sum_{\bsnu\in\N_0^k}l_\bsnu L_\bsnu(\bsy)\in
  L^2([-1,1]^k,\mu)$ for certain coefficients $l_\bsnu\in\R$ satisfying
  $|l_\bsnu|\le C_0 \gamma(\bsvarrho,\bsnu)
  \prod_{j=1}^k(1+2\nu_j)^\theta$ for all $\bsnu\in\N_0^k$. Then
  $f\in W^{m,\infty}([-1,1]^k)$ and with
  \begin{equation}\label{eq:Lepsk}
    \Lambda_{k,\eps}\dfn \setc{\bsnu\in\N_0^k}{\gamma(\bsvarrho,\bsnu)\ge\eps},
  \end{equation}
  $\varrho_{\rm min}\dfn\min_j\varrho_j>1$, and
  {$w_\bsnu\dfn\prod_{j=1}^k(1+2\nu_j)^{1/2+2m+\theta}$} it holds for all
  $\eps\in (0,1)$
  \begin{equation}\label{eq:etasumest}
    \normc[{W^{m,\infty}([-1,1]^k)}]{f(\cdot)-\sum_{\bsnu\in\Lambda_{k,\eps}}l_\bsnu L_\bsnu(\cdot)}
    \le  2C_0k S\left(\bsvarrho,\frac{1}{2}+\theta+2m\right)
    \eps\left(3-\frac{2\log(\eps)}{\log(\varrho_{\rm min})}\right)^{k(\frac{3}{2}+\theta+2m)},
  \end{equation}
  where $S(\bsvarrho,\frac{1}{2}+\theta+2m)$ is defined in \eqref{eq:Sxi}.
\end{lemma}
\begin{proof}
  For $\eps\in (0,\varrho_k^{-1})$ set
  \begin{equation*}
    A_\eps\dfn\set{\bsnu\in\N_0^{k}}{\gamma(\bsvarrho,\bsnu)<\eps,~\exists j\in\supp\bsnu\text{ s.t.\ }\gamma(\bsvarrho,\bsnu-\bse_j)\ge\eps}
  \end{equation*}
  so that
  \begin{equation*}
    \set{\bsnu\in\N_0^k}{\gamma(\bsvarrho,\bsnu)<\eps} = \set{\bsnu+\bsmu}{\bsnu\in A_\eps,~\bsmu\in\N_0^k}
  \end{equation*}
  since $\gamma(\bsvarrho,\bsnu)$ is monotonically decreasing in each
  $\nu_j$. Set
  \begin{equation*}
    A_{\eps,1}\dfn \set{\bsnu\in A_\eps}{\nu_k=0},\qquad
    {A_{\eps,2}}\dfn \set{\bsnu\in A_\eps}{\nu_k>0}
  \end{equation*}
  and
  \begin{equation}\label{eq:S1S2}
    S_1\dfn \set{\bsnu+\bsmu}{\bsnu\in A_{\eps,1},~\bsmu\in\N_0^{k-1}\times\{0\}},\quad
    S_2\dfn\set{\bsnu+\bsmu}{\bsnu\in A_{\eps,2},~\bsmu\in\N_0^k}.
  \end{equation}
  We want to show
  \begin{equation}\label{eq:S1cupS2}
    \set{\bsnu\in\N_0^k}{\gamma(\bsvarrho,\bsnu)<\eps} =S_1\cup S_2.
  \end{equation}
  If $\gamma(\bsvarrho,\bseta)<\eps$ and $\eta_k=0$, then $\bseta\in
  S_1$. Now let $\gamma(\bsvarrho,\bseta)<\eps$ and $\eta_k>0$. If there
  exists $\bsnu\in A_{\eps,2}$ (i.e., $\nu_k>0$) such that
  $\bsnu\le\bseta$ (i.e., $\nu_j\le\eta_j$ for all $j$), then
  $\bseta=\bsnu+(\bseta-\bsnu)\in {S_2}$.  Otherwise there exists
  $\bsnu\in A_{\eps,1}$ (i.e., $\nu_k=0$) with $\bsnu\le\bseta$. By
  definition $\gamma(\bsvarrho,\bsnu)=\gamma(\bsvarrho,\bsnu+\bse_k)$ (since
  $\nu_k=0$), and therefore
  $\tilde\bsnu\dfn\bsnu+\bse_k\in A_{\eps,2}$ satisfies
  $\tilde\bsnu\le\bseta$. Hence
  $\bseta=\tilde\bsnu+(\bseta-\tilde\bsnu)\in S_2$, showing
  \eqref{eq:S1cupS2}.

  Note that
  \begin{equation}\label{eq:gammasplits}
    \gamma(\bsvarrho,\bsnu+\bsmu)=\gamma(\bsvarrho,\bsnu)\bsvarrho^{-\bsmu}
  \end{equation}
  for all $\bsnu+\bsmu\in S_1$ or $\bsnu+\bsmu\in S_2$ as in
  \eqref{eq:S1S2}. As mentioned before, $\gamma(\bsvarrho,\bsnu)\ge\eps$
  implies with $\varrho_{\min}\dfn \min_j\varrho_j$ that
  $\varrho_{\rm min}^{-|\bsnu|}\ge\gamma(\bsvarrho,\bsnu)\ge\eps$. Thus if
  $\bsnu\in A_\eps$ and $\gamma(\bsvarrho,\bsnu-\bse_j)\ge\eps$ then
  $|\bsnu|\le-\frac{\log(\eps)}{\log(\varrho_{\rm min})}+1$. Hence
  $w_\bsnu\le (3-\frac{2\log(\eps)}{\log(\varrho_{\rm
      min})})^{k(\frac{1}{2}+\theta+2m)}$ for all $\bsnu\in
  A_\eps$. Using {
  \begin{equation*}
    w_{\bsnu+\bsmu}=\prod_{j=1}^k(1+2(\nu_j+\mu_j))^{1/2+\theta+2m}\le
    w_\bsnu w_\bsmu,
  \end{equation*}
  which is true because $1+2(\nu_j+\mu_j)\le (1+2\nu_j)(1+2\mu_j)$,} as well as \eqref{eq:WkinfLnu}, \eqref{eq:S1cupS2} and
  \eqref{eq:gammasplits}
  \begin{align*}
    \normc[{W^{m,\infty}([-1,1]^k)}]{f(\cdot)-\sum_{\bsnu\in\Lambda_{k,\eps}}l_\bsnu L_\bsnu}
    &\le \sum_{\bsnu\in S_1}|l_\bsnu|\norm[W^{m,\infty}({[-1,1]^k})]{L_\bsnu} + \sum_{\bsnu\in S_2}|l_\bsnu|\norm[W^{m,\infty}({[-1,1]^k})]{L_\bsnu}\nonumber\\
    &\le \sum_{{\bsnu\in A_{\eps,1}}}\sum_{\bsmu\in\N_0^{k-1}\times\{0\}}
      C_0w_{\bsnu+\bsmu} \gamma(\bsvarrho,\bsnu+\bsmu)\nonumber\\
    &\quad+
      \sum_{{\bsnu\in A_{\eps,2}}}\sum_{\bsmu\in\N_0^{k}}
      C_0w_{\bsnu+\bsmu} \gamma(\bsvarrho,\bsnu+\bsmu)\nonumber\\
    &\le 2C_0 \sum_{\bsnu\in A_\eps}w_\bsnu \gamma(\bsvarrho,\bsnu)\sum_{\bsmu\in\N_0^k}w_\bsmu\bsvarrho^{-\bsmu}\nonumber\\
    &\le 2C_0 \eps |A_\eps|\left(3-\frac{2\log(\eps)}{\log(\varrho_{\rm min})}\right)^{k(\frac{1}{2}+\theta+2m)}
      \sum_{\bsmu\in\N_0^k}w_\bsmu\bsvarrho^{-\bsmu}.
  \end{align*}
  Now $|A_\eps|\le k|\Lambda_{k,\eps}|$, due to the fact that for
  every $\bsnu\in A_\eps$ there exists $j\in\{1,\dots,k\}$ such that
  $\bsnu-\bse_j\in\Lambda_{k,\eps}$.  Finally, by
  \eqref{eq:Lepssimple},
  $|A_\eps|\le k|\Lambda_{k,\eps}|\le k( 1 -
  \frac{\log(\eps)}{\log(\varrho_{\rm min})})^k$, and together with
  \eqref{eq:Sxi} we obtain \eqref{eq:etasumest}.
\end{proof}

\begin{proof}[Proof of Prop.~\ref{PROP:KAPPROX}]
  {
  The function $x\mapsto \frac{2x}{x-1}$ is monotonically
  decreasing for $x>1$. Hence, with
  $\varrho_{\min}\dfn \min_{j\in\{1,\dots,k\}}\varrho_j$ we can
  replace the term $\prod_{j\in\supp\bsnu}\frac{2\varrho_j}{\varrho_j-1}$
  occurring in Lemma \ref{LEMMA:LEGEST}
  by
  $\prod_{j\in\supp\bsnu} 2^{\theta-1}$, where
  $\theta \dfn 1+\log_2(\frac{2\varrho_{{\rm min}}}{\varrho_{\rm
      min}-1})$, i.e., $\frac{2\varrho_j}{\varrho_j-1}
  \le \frac{2\varrho_{{\rm min}}}{\varrho_{\rm
      min}-1}=2^{\theta-1}$ for each $j$. Using
  $1+\nu_j\ge 2$ whenever $j\in\supp\bsnu$,
  \eqref{eq:lknubound} thus implies
  \begin{equation}\label{eq:rmklegest}
    \left|\int_{[-1,1]^d} f(\bsy) L_\bsnu(\bsy)\dd\mu(\bsy) \right|
    \le \bsvarrho^{-\bsnu}\norm[{L^\infty(\cB_\bsdelta([-1,1]))}]{f}
    \prod_{j\in\supp\bsnu}(1+2\nu_j)^{\theta}.
  \end{equation}
  Using the bounds on the Legendre coefficients $l_\bsnu$ of the
  function $f$ from Lemma \ref{LEMMA:LEGEST} we then get
  \begin{align*}
    |l_\bsnu|&\le \norm[L^\infty({\cB_{\bsdelta}([-1,1])})]{f}\prod_{j=1}^k(1+2\nu_j)^{\theta}
    \begin{cases}
      \varrho_k^{-1}\prod_{j=1}^{k-1}\varrho_j^{-\nu_j} &      \nu_k=1\\
      \bsvarrho^{-\bsnu} &\nu_k\ge 1
    \end{cases}\nonumber\\
    &=\gamma(\bsvarrho,\bsnu)
    \norm[L^\infty({\cB_{\bsdelta}([-1,1])})]{f} \prod_{j=1}^k(1+2\nu_j)^{\theta}.
  \end{align*}
  Thus by Lemma \ref{LEMMA:ETASUMEST}
  \begin{align*}
    \normc[{W^{m,\infty}([-1,1]^k)}]{f(\cdot)-\sum_{\bsnu\in\Lambda_{k,\eps}}l_\bsnu L_\bsnu(\cdot)}
    &\le \eps\left(3-\frac{2\log(\eps)}{\log(\varrho_{\rm min})}\right)^{k(\frac{3}{2}+\theta+2m)}k
      \norm[L^\infty({\cB_{\bsdelta}([-1,1])})]{f}\nonumber\\
    &\quad\cdot S\left(\bsvarrho,\frac{1}{2}+\theta+2m\right)
  \end{align*}
  Let $\tau\in (0,1)$ such that
  $\tilde\beta=\tau\beta<\beta$. Absorbing the logarithmic $\eps$
  term, %
  we get
  \begin{equation}\label{eq:Wminftyeps}
    \normc[{W^{m,\infty}([-1,1]^k)}]{f(\cdot)-\sum_{\bsnu\in\Lambda_{k,\eps}}l_\bsnu L_\bsnu(\cdot)}\le C \eps^{\tau} \norm[L^\infty({\cB_{\bsdelta}([-1,1])})]{f} S\left(\bsvarrho,\frac{1}{2}+\theta+2m\right),    
  \end{equation} %
  for some $C$ depending on $\varrho_{\rm min}$, $k$, $\tau$ and $m$.}

  Next, \eqref{eq:Lepsbound0} gives
  $(k!|\Lambda_{k,\eps}|\prod_{j=1}^k\log(\varrho_j))^{1/k}-\sum_{j=1}^k\log(\varrho_j)\le-\log(\eps)$
  and thus
  \begin{equation*}
    \eps\le \exp\left(-\beta|\Lambda_{k,\eps}|^{\frac 1 k}\right)
    \prod_{j=1}^{k}\varrho_j.
  \end{equation*}
  Plugging this into \eqref{eq:Wminftyeps} yields \eqref{eq:kapprox}.
\end{proof}

\subsection{Thm.~\ref{THM:POLYD}}\label{app:polyd}

In the proof we will need the following elementary lemma.
\begin{lemma}\label{lemma:Ckm}
  Let $m\in\N_0$, $k\in\N$.  There exists $C=C(k,m)$ such that
  \begin{enumerate}
  \item\label{item:Leibniz}
    $\norm[{W^{m,\infty}([-1,1]^k)}]{fg}\le C
    \norm[{W^{m,\infty}([-1,1]^k)}]{f}
    \norm[{W^{m,\infty}([-1,1]^k)}]{g}$, for all $f$,
    $g\in {W^{m,\infty}([-1,1]^k)}$,
  \item\label{item:square} for all $f$,
    $g\in {W^{m,\infty}([-1,1]^k)}$
    \begin{equation*}
      \norm[{W^{m,\infty}([-1,1]^k)}]{f^2-g^2} \le C (\norm[{W^{m,\infty}([-1,1]^k)}]{f}+\norm[{W^{m,\infty}([-1,1]^k)}]{g})\norm[{W^{m,\infty}([-1,1]^k)}]{f-g}
    \end{equation*}
  \item\label{item:inverse} if
    $\norm[{L^{\infty}([-1,1]^k)}]{1-f}\le \frac{1}{2}$ then
    \begin{equation*}
      \normc[{W^{m,\infty}([-1,1]^k)}]{1-\frac{1}{f}}\le C
      \norm[{W^{m,\infty}([-1,1]^k)}]{1-f}
      {\left(1+\norm[{W^{m,\infty}([-1,1]^k)}]{f}\right)^{\max\{0,m-1\}}}.
    \end{equation*}
  \end{enumerate}

  In case $m=0$, the constant $C=C(k,0)$ is independent of $k$.
\end{lemma}
\begin{proof}
  Item \ref{item:Leibniz} is a consequence of the (multivariate)
  Leibniz rule for weak derivatives, i.e.,
  $\partial_{\bsx}^{\bsnu} (fg) = \sum_{\bsmu\le\bsnu}
  \binom{\bsnu}{\bsmu}\partial_{\bsx}^{\bsmu}f
  \partial_{\bsx}^{\bsnu-\bsmu}g$ for all multi-indices
  $\bsnu\in\N_0^k$. Item \ref{item:square} follows by
  $f^2-g^2=(f-g)(f+g)$
  and \ref{item:Leibniz}.

  Next let us show \ref{item:inverse}.  %
  Due to $\norm[{L^\infty([-1,1]^k)}]{1-f}\le \frac{1}{2}$ it holds
  $\essinf_{\bsx\in [-1,1]^k}f(\bsx)\ge\frac{1}{2}$. Thus
  $\norm[{L^\infty([-1,1]^k)}]{1-\frac{1}{f}}
  =\norm[{L^\infty([-1,1]^k)}]{\frac{1-f}{f}} \le 2
  \norm[{L^\infty([-1,1]^k)}]{1-f}$, which proves the statement in
  case $m=0$.

  For general $m\in\N$, we claim that for any $\bsnu\in\N_0^k$ such
  that $1\le |\bsnu|\le m$ it holds
  $\partial_{\bsx}^\bsnu(\frac{1}{f})=\frac{p_\bsnu}{f^{|\bsnu|+1}}$
  for some $p_\bsnu$ satisfying
  \begin{equation}\label{eq:pnuhypothesis}
    \norm[W^{m-|\bsnu|,\infty}{([-1,1]^k)}]{p_\bsnu}\le C
    \norm[{W^{m,\infty}([-1,1]^k)}]{1-f}
    {(1+\norm[{W^{m,\infty}([-1,1]^k)}]{f})^{|\bsnu|-1}}
  \end{equation}
  for a constant $C$ depending on $\bsnu$ but independent of $f$. For
  $|\bsnu|=1$, i.e., $\bsnu=\bse_j=(\delta_{ij})_{i=1}^k$ for some
  $j$, this holds by
  $\partial_\bsx^{\bse_j}\frac{1}{f}=
  \partial_{j}\frac{1}{f}=\frac{-\partial_{j} f}{f^2}$ with
  $p_{\bse_j}=-\partial_{j} f$ where
  $\norm[{W^{m-1,\infty}([-1,1]^k)}]{p_{\bse_j}}
  =\norm[{W^{m-1,\infty}([-1,1]^k)}]{\partial_{j} f} \le
  \norm[{W^{m,\infty}([-1,1]^k)}]{1-f}$. For the induction step,
  \begin{equation*}
    \partial_{j}\frac{p_\bsnu}{f^{|\bsnu|+1}}
    = \frac{f^{|\bsnu|+1}\partial_{j}p_{\bsnu} - (|\bsnu|+1)p_\bsnu f^{|\bsnu|}\partial_{j}f}{f^{2|\bsnu|+2}}
    =\frac{f\partial_{j}p_{\bsnu} - (|\bsnu|+1)p_\bsnu\partial_{j}f}{f^{|\bsnu|+2}}\dfnn \frac{p_{\bsnu+\bse_j}}{f^{|\bsnu+\bse_j|+1}}.
  \end{equation*}
  Then
  \begin{align*}
    \norm[{W^{m-|\bsnu|-1,\infty}([-1,1]^k)}]{p_{\bsnu+\bse_j}}&=
                                                                 \norm[{W^{m-|\bsnu|-1,\infty}([-1,1]^k)}]{f\partial_{j}p_{\bsnu} - (|\bsnu|+1)p_\bsnu\partial_{j}f}\nonumber\\
                                                               &\le
                                                                 C
                                                                 (|\bsnu|+2)
                                                                 \norm[{W^{m,\infty}([-1,1]^k)}]{f}\norm[{W^{m-|\bsnu|,\infty}([-1,1]^k)}]{p_\bsnu}\nonumber\\
                                                               &\le
                                                                 C
                                                                 \norm[{W^{m,\infty}([-1,1]^k)}]{1-f}{(1+\norm[{W^{m,\infty}([-1,1]^k)}]{f})^{|\bsnu|}}
  \end{align*}
  by the induction hypothesis \eqref{eq:pnuhypothesis}.

  Hence for $1\le|\bsnu|\le m$, due to
  $\norm[L^\infty({[-1,1]^k})]{1-f}\le \frac{1}{2}$
  which implies $\essinf_{\bsx\in [-1,1]^k} f(\bsx)^{|\bsnu|+1}\ge 2^{-|\bsnu|-1}$,
  \begin{align*}
    \normc[{L^\infty([-1,1]^k)}]{\partial_{\bsx}^\bsnu\Big(1- \frac{1}{f}\Big)}
    &\le C
      \frac{\norm[{W^{m,\infty}([-1,1]^k)}]{1-f}
      {(1+\norm[{W^{m,\infty}([-1,1]^k)}]{f})^{m-1}}
      }{{\essinf_{\bsx\in [-1,1]^k} f(\bsx)^{|\bsnu|+1}}}\nonumber\\
    &\le C 2^{{m}+1} \norm[{W^{m,\infty}([-1,1]^k)}]{1-f}{(1+\norm[{W^{m,\infty}([-1,1]^k)}]{f})^{m-1}}
  \end{align*}
  which shows \ref{item:inverse}.
\end{proof}

\begin{lemma}\label{lemma:ptildeR}
  Let $m\in\N_0$, $k\in\N$, $k\ge 2$ and $T:[-1,1]^k\to [-1,1]$ such
  that $T\in W^{m,\infty}([-1,1]^k)$,
  ${\sqrt{\partial_{k}T}}\in C^0\cap W^{m,\infty}([-1,1]^k)$, and
  $T(\bsx_{[k-1]},\cdot):[-1,1]\to [-1,1]$ is an increasing bijection
  for every $\bsx_{[k-1]}\in [-1,1]^{k-1}$. Set
  ${Q}(\bsx)\dfn \sqrt{\partial_{k}T}-1\in C^0\cap W^{m,\infty}\cap
  ([-1,1]^k)$.

  There exist constants $K\in (0,1]$ (independent of $m$ and $k$) and
  $C=C(k,m)>0$, both independent of ${Q}$, with the following
  property: If $p\in W^{m,\infty}([-1,1]^k;\R)$ satisfies
  \begin{equation}\label{eq:lemmaptildeRK}
    \norm[W^{m,\infty}({[-1,1]^k})]{{Q}-p}\le 1\qquad
    \text{and}
    \qquad
    \norm[L^\infty({[-1,1]^k})]{{Q}-p}<\frac{K}{1+\norm[L^\infty({[-1,1]^k})]{{Q}}}
  \end{equation}
  then with
  \begin{equation}\label{eq:tildeT}
    \tilde T(\bsx)\dfn -1+\frac{2}{c_k(\bsx_{[k-1]})}\int_{-1}^{x_k}(p(\bsx_{[k-1]},t)+1)^2\dd t\quad
    \text{where}\quad
    c_k(\bsx_{[k-1]})\dfn \int_{-1}^1 (p(\bsx_{[k-1]},t)+1)^2\dd t
  \end{equation}
  it holds
  \begin{equation}\label{eq:ptildeR1}
    \norm[W^{m,\infty}({[-1,1]^k})]{T-\tilde T}\le C
    (1+\norm[W^{m,\infty}({[-1,1]^k})]{{Q}})^{\max\{3,m+2\}}
    \norm[W^{m,\infty}({[-1,1]^k})]{{Q}-p}
  \end{equation}
  and
  \begin{equation}\label{eq:ptildeR2}
    \norm[W^{m,\infty}({[-1,1]^k})]{\partial_{k} T-\partial_{k} \tilde T}\le C
    (1+\norm[W^{m,\infty}({[-1,1]^k})]{{Q}})^{\max\{3,m+2\}}
    \norm[W^{m,\infty}({[-1,1]^k})]{{Q}-p}.
  \end{equation}
  
  In case $m=0$, $C(k,0)$ is independent of $k$.
\end{lemma}
\begin{proof}
  Throughout this proof the constant $C>0$ (which may change its value
  even within the same equation) will depend on $m$, $k$ but be
  independent of ${{Q}}$. Moreover, it will only depend on $k$
  through the constants from Lemma \ref{lemma:Ckm}, and thus be
  \emph{independent of $k$} in case $m=0$. In the following we use
  several times the fact that
  \begin{equation}\label{eq:normboundint}
    \normc[{W^{m,\infty}([-1,1]^k)}]{\int_{-1}^{x_k}g(\bsx_{[k-1]},t)\dd t}
    \le 2 \norm[{W^{m,\infty}([-1,1]^k)}]{g(\bsx)}
    \qquad\forall
    g\in W^{m,\infty}([-1,1]^k).
  \end{equation}

  Set $\eps\dfn \norm[W^{m,\infty}({[-1,1]^k})]{{{Q}}-p}\le
  1$. Using Lemma \ref{lemma:Ckm} \ref{item:square} we get
  \begin{align}\label{eq:pp1dxT}
    \norm[W^{m,\infty}({[-1,1]^k})]{(p+1)^2-\partial_{k}T}&=\norm[W^{m,\infty}({[-1,1]^k})]{(p+1)^2-({{Q}}+1)^2}\nonumber\\
                                                            &\le C (\norm[W^{m,\infty}({[-1,1]^k})]{p+1}+\norm[W^{m,\infty}({[-1,1]^k})]{{{Q}}+1}) \eps\nonumber\\
                                                            &\le C (2\norm[W^{m,\infty}({[-1,1]^k})]{{{Q}}+1}+\eps)\eps\nonumber\\
                                                            &\le C(1+\norm[W^{m,\infty}({[-1,1]^k})]{{{Q}}})\eps,%
  \end{align}
  where we used the triangle inequality and $\eps\le 1$, which holds
  by assumption.

  Additional to $\tilde T$ in \eqref{eq:tildeT} let
  \begin{equation*}
  \hat T({\bsx})=-1+\int_{-1}^{x_k} (1+p(\bsx_{[k-1]},t))^2\dd
  t.
\end{equation*}
Since
  $T(\bsx)=-1+\int_{-1}^{x_k}\partial_{k} T(\bsx_{[k-1]},t)\dd t$ we
  get with \eqref{eq:normboundint} {and \eqref{eq:pp1dxT}}
  \begin{align}\label{eq:TkhTk}
    \norm[W^{m,\infty}({[-1,1]^k})]{T-\hat T}
    &= \normc[W^{m,\infty}({[-1,1]^k})]{{\int_{-1}^{x_k}}\partial_{k}T(\bsx_{[k-1]},t) -
      (p(\bsx_{[k-1]},t)+1)^2{\dd t}}\nonumber\\
    &\le C(1+\norm[W^{m,\infty}({[-1,1]^k})]{{{Q}}})\eps.
  \end{align}

  Fix $\bsx_{{[k-1]}}\in [-1,1]^{k-1}$. Then
  \begin{equation*}
    c_{k}(\bsx_{[k-1]})=\int_{-1}^1 (p(\bsx_{[k-1]},t)+1)^2\dd t=
    \hat
    T(\bsx_{[k-1]},1)-\hat T(\bsx_{[k-1]},-1).
  \end{equation*}
  Moreover, since $T(\bsx_{[k-1]},\cdot):[-1,1]\to [-1,1]$ is a
  monotonically increasing bijection,
  \begin{equation*}
    \int_{-1}^1\partial_{k} T(\bsx_{[k-1]},t)\dd t =
    T(\bsx_{[k-1]},1)-
    T(\bsx_{[k-1]},-1)=2\qquad \forall \bsx\in [-1,1]^k.
  \end{equation*}
  Thus, using that $T$, $\hat T$ are Lipschitz continuous, and
  restrictions to $[-1,1]^{k-1}\times \{y\}$, $y\in\{-1,1\}$, are
  well-defined,
  \begin{align}\label{eq:oock}
    \norm[W^{m,\infty}({[-1,1]^{k-1}})]{c_k-2}
    &\le\norm[W^{m,\infty}({[-1,1]^{k-1}})]{T(\cdot,1)-\hat T(\cdot,1)}+\norm[W^{m,\infty}({[-1,1]^{k-1}})]{T(\cdot,-1)-\hat T(\cdot,-1)}\nonumber\\
    &\le 2 \norm[W^{m,\infty}({[-1,1]^{k}})]{T-\hat T}
      \le \tilde C(1+\norm[W^{m,\infty}({[-1,1]^{k}})]{{{Q}}})\eps,
  \end{align}
  for some $\tilde C=\tilde C(k,m)>0$, which is independent of $k$ in
  case $m=0$. Additionally by \eqref{eq:oock}
  \begin{equation}\label{eq:oock2}
    \norm[W^{m,\infty}({[-1,1]^{k-1}})]{c_k}\le
    2+\norm[W^{m,\infty}({[-1,1]^{k-1}})]{c_k-2}
    \le 2+C(1+\norm[W^{m,\infty}({[-1,1]^{k}})]{{{Q}}})\eps.
  \end{equation}

  Define $K\dfn \min\{1,\frac{1}{\tilde C(0,k)}\}$ with
  $\tilde C(0,k)$ from \eqref{eq:oock}, (i.e., $m=0$ and $K$ does not
  depend on $k$). By assumption
  $\eps\le \frac{K}{1+\norm[L^{\infty}({[-1,1]^{k}})]{{{Q}}}}$,
  which implies
  $\norm[L^\infty({[-1,1]^k})]{1-\frac{c_k}{2}}\le \frac{1}{2}$ by
  \eqref{eq:oock} (for $m=0$). {This allows to apply
  Lemma \ref{lemma:Ckm}
  \ref{item:inverse}, which together with} \eqref{eq:oock} %
and \eqref{eq:oock2}
  gives
  \begin{align}\label{eq:1-2dc}
    \normc[W^{m,\infty}({[-1,1]^{k-1}})]{1-\frac{2}{c_k}}&\le
                                                           \normc[{W^{m,\infty}([-1,1]^{k-1})}]{1-\frac{c_k}{2}}
                                                           \left({1+}\normc[{W^{m,\infty}([-1,1]^{k-1})}]{\frac{c_k}{2}}\right)^{\max\{0,m-1\}}\nonumber\\
                                                         &\le
                                                           C \eps
                                                           (1+\norm[{W^{m,\infty}([-1,1]^{k-1})}]{{Q}})^{\max\{1,m\}}.
  \end{align}
  Since $\tilde T=-1+\frac{2}{c_k}(\hat T+1)$, \eqref{eq:TkhTk}
  {and \eqref{eq:1-2dc}} yield
  \begin{align*}%
    &\norm[W^{m,\infty}({[-1,1]^k})]{\tilde T-T}\le
      \norm[W^{m,\infty}({[-1,1]^k})]{\hat T-T} + \norm[W^{m,\infty}({[-1,1]^k})]{\hat T-\tilde T}\nonumber\\
    &\qquad\le C(1+\norm[W^{m,\infty}({[-1,1]^{k}})]{{{Q}}}) \eps+ \normc[W^{m,\infty}({[-1,1]^k})]{\left(1-\frac{2}{c_k}\right) + \left(1-\frac{2}{c_k}\right)\hat T}\nonumber\\
    &\qquad\le C(1+\norm[W^{m,\infty}({[-1,1]^{k}})]{{{Q}}})\eps\\
    &\qquad\quad+\normc[W^{m,\infty}({[-1,1]^{k-1}})]{1-\frac{2}{c_k}}(1+ \norm[W^{m,\infty}({[-1,1]^k})]{\hat T-T}+\norm[W^{m,\infty}({[-1,1]^k})]{T})\nonumber\\
    &\qquad\le C\eps(1+\norm[W^{m,\infty}({[-1,1]^k})]{{{Q}}})\nonumber\\
    &\qquad\quad+C\eps(1+\norm[W^{m,\infty}({[-1,1]^k})]{{{Q}}})^{\max\{1,m\}}
      \big(1+C\eps(1+\norm[W^{m,\infty}({[-1,1]^k})]{{{Q}}})+\norm[W^{m,\infty}({[-1,1]^k})]{T}\big)%
  \end{align*}
  for some constant $C$ depending on $m$ and $k$. Due to
  $T(\bsx)=-1+\int_{-1}^{x_k}({{Q}}(\bsx_{[k-1]},t)+1)^2\dd x_k$,
  by \eqref{eq:normboundint}
  \begin{equation*}
    \norm[{W^{m,\infty}([-1,1]^k)}]{T}\le C(1+\norm[{W^{m,\infty}([-1,1]^k)}]{({{Q}}+1)^2})\le C(1+\norm[{W^{m,\infty}([-1,1]^k)}]{{{Q}}})^2,
  \end{equation*}
  by Lemma \ref{lemma:Ckm} \ref{item:Leibniz}.  In all this shows
  \eqref{eq:ptildeR1}.

  To show \eqref{eq:ptildeR2} we proceed similarly and obtain via
  \eqref{eq:pp1dxT}
  \begin{align*}
    &\norm[W^{m,\infty}({[-1,1]^k})]{\partial_{k} T-\partial_{k} \tilde T} =\normc[W^{m,\infty}({[-1,1]^k})]{%
      \partial_{k}T
      -\frac{2(p+1)^2}{{c_k}}}\nonumber\\
    &\qquad\le \normc[W^{m,\infty}({[-1,1]^k})]{\partial_{k}T
      -(p+1)^2}
      +\normc[W^{m,\infty}({[-1,1]^k})]{(p+1)^2\left(1-\frac{2}{{c_k}}\right)}\nonumber\\
    &\qquad\le C(1+\norm[W^{m,\infty}({[-1,1]^k})]{{{Q}}})\eps
      +C(1+\normc[W^{m,\infty}({[-1,1]^k})]{p})^2
      \normc[W^{m,\infty}({[-1,1]^{k-1}})]{1-\frac{2}{{c_k}}},
  \end{align*}
  where we used Lemma \ref{lemma:Ckm} \ref{item:Leibniz} to bound
  $\norm[W^{m,\infty}({[-1,1]^k})]{(1+p)^2}$. %
  {Since $\norm[W^{m,\infty}({[-1,1]^k})]{p-{{Q}}}\le 1$}
  \begin{equation*}
    \norm[W^{m,\infty}({[-1,1]^k})]{p} \le
    \norm[W^{m,\infty}({[-1,1]^k})]{p-{{Q}}}+
    \norm[W^{m,\infty}({[-1,1]^k})]{{Q}}
    \le 1+
    \norm[W^{m,\infty}({[-1,1]^k})]{{{Q}}},
  \end{equation*}
  and thus using \eqref{eq:1-2dc}
  \begin{equation*}
    \norm[W^{m,\infty}({[-1,1]^k})]{\partial_{k} T-\partial_{k} \tilde T}\le C(1+\norm[W^{m,\infty}({[-1,1]^k})]{{{Q}}})^{\max\{3,m+2\}}\eps.
    \qedhere
  \end{equation*}
\end{proof}

\begin{proof}[Proof of Thm.~\ref{THM:POLYD}]
  Wlog in the following we assume $m\ge 1$, since the statement for
  $m\ge 1$ trivially implies the statement for $m=0$.
  {Throughout we omit writing the $\eps>0$ index and
  write $\tilde T_{k}\dfn \tilde T_{k,\eps}$ etc.}

  {\bf Step 1.} Fix $k\in\{1,\dots,d\}$ and define
  ${Q_k}(\bsx)\dfn \sqrt{\partial_{k}T_k}-1$. We now construct
  $p_k$ such that $\norm[{W^{m,\infty}([-1,1]^k)}]{{Q_k}-p_k}$ is
  small.

  According to Thm.~\ref{THM:COR:DINN} with
  $\bszeta=(C_7\delta_j)_{j=1}^d$ it holds
  \begin{equation}\label{eq:cEbsxik_1}
    \partial_{k}T_k\in
    C^1(\cB_{\bszeta_{[k]}}([-1,1]);\cB_{C_8}(1))
  \end{equation}
  and if $k\ge 2$
  \begin{equation}\label{eq:cEbsxik_2}
    \partial_{k}T_k:
    \cB_{\bszeta_{[k-1]}}([-1,1])\times [-1,1]\to \cB_{\frac{C_8}{\max\{1,\delta_k\}}}(1).
  \end{equation}
  
  We have $\Re(\partial_{k}T_k(\bsx)) > 0$ for all
  $\bsx\in \cB_{\bszeta_{[k]}}$ by Thm.~\ref{THM:COR:DINN}
  \ref{item:cordinN:a}, so that \eqref{eq:cEbsxik_1} implies
  \begin{equation}\label{eq:sqrtK-1} %
    \sqrt{\partial_{k}T_k}-1:
    {\cB_{\bszeta_{[k]}}([-1,1])}\to \cB_{\sqrt{C_8}}(1)\subseteq
    \cB_{\sqrt{C_8}+1}.
  \end{equation}
  If $\delta_k<1$, then $\frac{1+C_7}{1+C_7\delta_k}\ge 1$ and by
  \eqref{eq:sqrtK-1}
  \begin{equation*}
    \sqrt{\partial_{k}T_k}-1:
    \cB_{\bszeta_{[k-1]}}([-1,1])\times [-1,1]\to \cB_{\sqrt{C_8}+1}
    \subseteq \cB_{\frac{(1+C_7)(\sqrt{C_8}+1)}{1+C_7\delta_k}}.
  \end{equation*}
  Next, we consider the case $\delta_k\ge 1$, so that
  $\max\{1,\delta_k\}=\delta_k$. For any $z\in \C$ with
  $|z|\le \frac{1}{2}$ one checks that $|\sqrt{z}+1|\ge \frac{1}{2}$.
  Thus $|\sqrt{z}-1|\le 2|\sqrt{z}-1||\sqrt{z}+1|=2|z-1|$. Hence, if
  $\frac{C_8}{\max\{1,\delta_k\}}\le \frac{1}{2}$, then by
  \eqref{eq:cEbsxik_2}
  \begin{equation*}
    \sqrt{\partial_{k}T_k}-1:
    \cB_{\bszeta_{[k-1]}}([-1,1])\times [-1,1]\to \cB_{2\frac{C_8}{\max\{1,\delta_k\}}}
    =\cB_{\frac{2C_8}{\delta_k}}
    \subseteq \cB_{\frac{2C_8(1+C_7)}{1+C_7\delta_k}},
  \end{equation*}
  and if $\frac{C_8}{\max\{1,\delta_k\}}>\frac{1}{2}$, then
  $2C_8>\delta_k$ and by \eqref{eq:sqrtK-1}
  \begin{equation*}
    \sqrt{\partial_{k}T_k}-1:
    \cB_{\bszeta_{[k-1]}}([-1,1])\times [-1,1]\to \cB_{\sqrt{C_8}+1}
    \subseteq \cB_{\frac{(1+2C_7C_8)(\sqrt{C_8}+1)}{1+C_7\delta_k}}.
  \end{equation*}
  
  In all with
  \begin{equation*}
    r\dfn \max\left\{1+\sqrt{C_8},2C_8(1+C_7),(1+C_7)(\sqrt{C_8}+1),(1+2C_7C_8)(\sqrt{C_8}+1)\right\}
  \end{equation*}
  holds
  \begin{equation*}
    {Q_k}\in
    C^1({\cB_{\bszeta_{[k]}}([-1,1])};\cB_{r})
  \end{equation*}
  and if $k\ge 2$
  \begin{equation*}
    {Q_k}:
    \cB_{\bszeta_{[k-1]}}([-1,1])\times [-1,1]\to \cB_{\frac{r}{1+C_7\delta_k}}.
  \end{equation*}
  
  Therefore, for $m\in\N_0$ fixed, Prop.~\ref{PROP:KAPPROX} gives for
  $\eps\in (0,1)$ and with
  \begin{equation*}
    p_{k,\eps}{(\bsx)}=\sum_{\bsnu\in\Lambda_{k,\eps}} L_\bsnu{(\bsx)}
    \int_{[-1,1]^k}{Q_k}(\bsy)L_\bsnu(\bsy)\dd\mu(\bsy)\in \bbP_{\Lambda_{k,\eps}}
  \end{equation*}
  that {with $\tau:=\frac{\tilde\beta}{\beta}$}
  \begin{equation}\label{eq:WminftyRkpkeps}
    \norm[{W^{m,\infty}([-1,1]^k)}]{{Q_k}-p_{k,\eps}}\le
    \tilde C\eps^{\tau},
  \end{equation}
  for a constant $\tilde C$ depending on $r$, $m$, $\tau$, $\bsdelta$
  and $T$ but independent of $\eps\in (0,1)$.

  {\bf Step 2.} Fix $k\in\{1,\dots,d\}$. %
  Let the constant $K>0$ be as in Lemma \ref{lemma:ptildeR}. We
  distinguish between two cases, first assuming
  $\tilde C \eps^\tau<\max\{1,\frac{K}{1+\norm[L^\infty({[-1,1]^k})]{{Q_k}}}\}$.
  , %
  Then \eqref{eq:WminftyRkpkeps} and Lemma \ref{lemma:ptildeR} imply
  \begin{subequations}\label{eq:tTk-Tkleeps2}
    \begin{equation}
      \norm[{W^{m,\infty}([-1,1]^k)}]{T_k-\tilde T_k}\le C \eps^\tau,
    \end{equation}
    for a constant $C$ independent of $\eps$.

    In the second case where
    $\tilde C \eps^\tau>\max\{1,\frac{K}{1+\norm[L^\infty({[-1,1]^k})]{{Q_k}}}\}$, we simply redefine $p_k\dfn 0$, so that
    $\tilde T_k(\bsx)=x_k$ (cp.~\eqref{eq:ck}). Then
    \begin{align}
      \norm[{W^{m,\infty}([-1,1]^k)}]{T_k-\tilde T_k}
      &\le\underbrace{\frac{\norm[{W^{m,\infty}([-1,1]^k)}]{T_k}+
        \norm[{W^{m,\infty}([-1,1]^k)}]{x_k}}{\max\{1,\frac{K}{1+\norm[L^\infty({[-1,1]^k})]{{Q_k}}}\}\frac{1}{\tilde
          C}}}_{{\dfnn C}} \max\Big\{1,\frac{K}{1+\norm[L^\infty({[-1,1]^k})]{{Q_k}}}\Big\}\frac{1}{\tilde
        C}\nonumber\\
      &\le C \eps^\tau,
    \end{align}
  \end{subequations}
  with $C>0$ depending on $K$, $\tilde C$,
  $\norm[{W^{m,\infty}([-1,1]^k)}]{T_k}$ and
  $\norm[{W^{m,\infty}([-1,1]^k)}]{{Q_k}}$ but not on $\eps$.

  {\bf Step 3.} We estimate the error in terms of $N_\eps$. By Lemma
  \ref{lemma:lepsbound}
  \begin{equation*}
    N_\eps = \sum_{k=1}^d|\Lambda_{k,\eps}|
    \le \frac{d}{d!} \left(-\log(\eps)+\sum_{j=1}^d\log(\xi_j) \right)^d \prod_{j=1}^d \frac{1}{\log(\xi_j)},
  \end{equation*}
  which implies
  \begin{equation*}
    \eps\le \exp\left(-\left(N_\eps (d-1)!\prod_{j=1}^d\log(\xi_j)\right)^{1/d}\right)\prod_{j=1}^d\xi_j.
  \end{equation*}
  Together with \eqref{eq:tTk-Tkleeps2} this shows
  \begin{equation*}
    \norm[{W^{m,\infty}([-1,1]^k)}]{T_k-\tilde T_k}\le C \exp(-\tau\beta N_\eps^{1/d}){=C \exp(-\tilde\beta N_\eps^{1/d})}\qquad\forall k\in\{1,\dots,d\}.\qedhere
  \end{equation*}
\end{proof}

\section{Proofs of Sec.~\ref{sec:relu}}

\subsection{Thm.~\ref{THM:RELU}}\label{app:relu}
\begin{proof}[Proof of Thm.~\ref{THM:RELU}]
  Fix $k\in\{1,\dots,d\}$. By \cite[Thm.~3.6]{OSZ19}, there exists
  $\beta>0$ such that for any $N\in\N$ there exists a ReLU NN
  $\tilde\Phi_{N,k}:[-1,1]^k\to\R$ such that
  $\norm[{W^{1,\infty}([-1,1]^k)}]{T_k-\tilde\Phi_{N,k}}\le C
  \exp(-\beta N^{1/(k+1)})$ and ${\rm size}(\tilde\Phi_{N,k})\le {N}$,
  ${\rm depth}(\tilde\Phi_{N,k})\le C\log(N) N^{1/({k+1})}$.

  We use Lemma \ref{LEMMA:RELU} to correct $\tilde\Phi_{N,k}$ and
  guarantee that it is a bijection of $[-1,1]^d$ onto itself. For
  $k\ge 2$ define $f_{-1}(\bsx)\dfn -1-\tilde\Phi_{N,k}(\bsx,-1)$ and
  $f_{1}(\bsx)\dfn 1-\tilde\Phi_{N,k}(\bsx,1)$ for
  $\bsx\in [-1,1]^{k-1}$. In case $k=1$ set
  $f_{-1}\dfn -1-\tilde\Phi_{N,1}(-1)\in\R$ and
  $f_{1}\dfn 1-\tilde\Phi_{N,1}(1)\in\R$.  With the notation from
  Lemma \ref{LEMMA:RELU} for $\bsx\in [-1,1]^{k-1}$ and
  $x_k\in [-1,1]$ set
  \begin{equation*}
    \Phi_{N,k}(\bsx,x_k)\dfn \tilde\Phi_{N,k}(\bsx,x_k)
    +g_{f_1}(\bsx,x_k)+g_{f_{-1}}(\bsx,-x_k).
  \end{equation*}
  Then by Lemma \ref{LEMMA:RELU} for any $\bsx\in [-1,1]^{k-1}$
  \begin{align*}
    \Phi_{N,k}(\bsx,1)&= \tilde\Phi_{N,k}(\bsx,1)
                        +g_{f_1}(\bsx,1)+g_{f_{-1}}(\bsx,-1)\nonumber\\
                      &= \tilde\Phi_{N,k}(\bsx,1)+f_1(\bsx)+0\nonumber\\
                      &= \tilde\Phi_{N,k}(\bsx,1)+1-\tilde\Phi_{N,k}(\bsx,1)\nonumber\\
                      &=1.
  \end{align*}
  Similarly $\Phi_{N,k}(\bsx,-1)=-1$. Clearly $\Phi_{N,k}$ is a ReLU
  NN, and with Lemma \ref{LEMMA:RELU}
  \begin{align*} {\rm size}(\Phi_{N,k})&\le C(1+{\rm
      size}(\tilde\Phi_{N,k})+{\rm size}(g_{f_1})
                                         +{\rm size}(g_{f_{-1}}))\nonumber\\
                                       &\le C(1+{\rm
                                         size}(\tilde\Phi_{N,k}))\le
                                         C(1+N)\le C N,
  \end{align*}
  for some suitable constant $C$ independent of $N\in\N$. Similarly
  ${\rm depth}(\Phi_{N,k})\le C N^{1/(1+k)}\log(N)$.

  By Lemma \ref{LEMMA:RELU},
  $\norm[{W^{1,\infty}([-1,1]^k)}]{g_{f_1}}\le C
  \norm[{W^{1,\infty}([-1,1]^k)}]{f_1}$ and the same inequality holds
  for {$f_{-1}$}. Hence
  \begin{align*}
    \norm[{W^{1,\infty}([-1,1]^k)}]{\Phi_{N,k}-T_k}
    &\le     \norm[{W^{1,\infty}([-1,1]^k)}]{\tilde\Phi_{N,k}-T_k}
      + C \norm[{W^{1,\infty}([-1,1]^k)}]{f_1}
      + C \norm[{W^{1,\infty}([-1,1]^k)}]{f_{-1}}\nonumber\\
    &\le     \norm[{W^{1,\infty}([-1,1]^k)}]{\tilde\Phi_{N,k}-T_k}
      + C \norm[{W^{1,\infty}([-1,1]^{k-1})}]{\tilde\Phi_{N,k}(\cdot,1)-1}\nonumber\\
    &\quad+ C \norm[{W^{1,\infty}([-1,1]^{k-1})}]{\tilde\Phi_{N,k}(\cdot,-1)+1}\nonumber\\
    & = \norm[{W^{1,\infty}([-1,1]^k)}]{\tilde\Phi_{N,k}-T_k}
      + C \norm[{W^{1,\infty}([-1,1]^{k-1})}]{\tilde\Phi_{N,k}(\cdot,1)-T_k(\cdot,1)}\nonumber\\
    &\quad+ C \norm[{W^{1,\infty}([-1,1]^{k-1})}]{\tilde\Phi_{N,k}(\cdot,-1)+T_k(\cdot,-1)}      \nonumber\\
    &\le C \norm[{W^{1,\infty}([-1,1]^k)}]{\tilde\Phi_{N,k}-T_k}.
  \end{align*}
  The last term is bounded by $C_0\exp(-\beta N^{1/(k+1)})$ by
  definition of $\tilde \Phi_{N,k}$, and where $C_0$ and $\beta$ are
  independent of $N\in\N$ {(but depend on $d$)}.
  Choosing $N$ large
  enough, it holds
  \begin{equation*}
    \inf_{\bsx\in [-1,1]^k}\partial_{k}\Phi_{N,k}(\bsx) \ge
    \frac{1}{2}\inf_{\bsx\in [-1,1]^k}\partial_{k}T_k(\bsx)>0.
  \end{equation*}
    Then
  for every $\bsx\in [-1,1]^{k-1}$,
  $\Phi_{N,k}(\bsx,\cdot):[-1,1]\to [-1,1]$ is continuous,
  monotonically increasing and satisfies $\Phi_{N,k}(\bsx,-1)=-1$
  $\Phi_{N,k}(\bsx,1)=1$ as shown above.

  In all, $\Phi_N=(\Phi_{N,k})_{k=1}^d:[-1,1]^d\to [-1,1]^d$ is
  monotone, triangular and satisfies the claimed {error bound}
  $\norm[{W^{1,\infty}([-1,1]^d)}]{\Phi_N-T}\le C\exp(-\beta
  N^{1/(d+1)})$ for some $C>0$, $\beta>0$ and all $N\in\N$. Moreover,
  ${\rm size}(\Phi_N)\le\sum_{k=1}^d {\rm size}(\Phi_{N,k}) \le C N$
  for a constant $C$ depending on $d$. Finally
  ${\rm depth}(\Phi_N)\le\max_{k\in\{1,\dots,d\}}{\rm
    depth}(\Phi_{N,k}) \le C N^{1/2}\log(N)$. To guarantee that all
  $\Phi_{N,k}$, $k=1,\dots,d$ have the same depth $CN^{1/2}\log(N)$ we
  can concatenate $\Phi_{N,k}$ (at most $CN^{1/2}\log(N)$ times) with
  the identity network $x=\varphi(x)-\varphi(-x)$. This does not
  change the size and error bounds.
\end{proof}

\subsection{Lemma \ref{LEMMA:RELU}}\label{app:relulem}
\begin{proof}[Proof of Lemma \ref{LEMMA:RELU}]
  Set $a\dfn \min_{\bsx\in [-1,1]^d}f(\bsx)$ and
  $b\dfn \max_{\bsx\in [-1,1]^d}f(\bsx)-a$. Define
  \begin{equation*}
    g_f(\bsx,t)\dfn \varphi\left(\frac{f(\bsx)-a}{b}
      +(t-1)\right)b+a\varphi(t).
  \end{equation*}
  Using the identity network $t=\varphi(t)-\varphi(-t)$, we may carry
  the value of $t$ from layer $0$ to layer $L={\rm depth}(f)$ with a
  network of size $C{\rm depth}(f)\le C{\rm size}(f)$. This implies
  that $g_f$ can be written as a ReLU NN satisfying the claimed size and
  depth bounds.

  The definition of $a$ and $b$ implies
  $0\le \frac{f(\bsx)-a}{b}\le 1$ for all $\bsx\in [-1,1]^d$. Hence
  $g_f(\bsx,1)=f(\bsx)$ and $g_f(\bsx,s)=0$ for all $s\in [-1,0]$.
  Thus $\frac{d}{dt}g_f(\bsx,t)=0$ for $t\le 0$ and either
  $\frac{d}{dt}g_f(\bsx,t)=a$ or $\frac{d}{dt}g_f(\bsx,t)=b+a$ for
  $t\in (0,1]$. Now $|a|\le\max_{\bsy\in [-1,1]^d}|f(\bsx)|$ and
  $|b+a|\le\max_{\bsy\in [-1,1]^d}|f(\bsx)|$ imply
  the bound on $|\frac{d}{dt}g_f(\bsx,t)|$. Next fix $t\in [-1,1]$.
  At all points where $\nabla_\bsx g_f(\bsx,t)$ is well-defined, it
  either holds $\nabla_\bsx g_f(\bsx,t)=\nabla f(\bsx)$ or
  $\nabla_\bsx g_f(\bsx,t)=0$, which concludes the proof.
\end{proof}

\section{Proofs of Sec.~\ref{SEC:MEASURES}}

{
\subsection{Lemma \ref{LEMMA:WINFTY}}
\begin{proof}[Proof of Lemma \ref{LEMMA:WINFTY}]
  First,
  \begin{equation*}
    \sup_{\bsx\in [-1,1]^d}\norm{S(\bsx)-\tilde S(\bsx)}=
    \sup_{\bsx\in [-1,1]^d}\norm{S(\tilde T(\bsx))-\tilde S(\tilde T(\bsx))}
    =
    \sup_{\bsx\in [-1,1]^d}\norm{S(\tilde T(\bsx))-S(T(\bsx))},
  \end{equation*}
  which is bounded by the Lipschitz constant of $S$ times
  $\norm[{L^\infty([-1,1]^d)}]{T-\tilde T}$.

  For the second bound, let first $A$, $B\in\R^{d\times d}$. Then
  \begin{equation*}
    \norm{A-B}=\norm{A(B^{-1}-A^{-1})B}
    \le \norm{A}\norm{B}\norm{A^{-1}-B^{-1}}.
  \end{equation*}
  Since $S\circ T$ is the identity we have $dS(\bsx)=dT(S(\bsx))^{-1}$
  and similarly $d\tilde S(\tilde T(\bsx))=d\tilde T(\bsx)^{-1}$. Thus
  \begin{align*}
    &\norm[{L^\infty([-1,1]^d)}]{dS-d\tilde S}\nonumber\\
    &\qquad\le \norm[{L^\infty([-1,1]^d)}]{dS} \norm[{L^\infty([-1,1]^d)}]{d\tilde S}
      \esssup_{\bsx\in [-1,1]^d}\norm{dT(S(\bsx))-d\tilde T(\tilde S(\bsx))}.
  \end{align*}
  The essential supremum is bounded by
  \begin{align*}
    &%
      \esssup_{\bsx\in [-1,1]^d}\big(\norm{dT(S(\bsx))-d T(\tilde S(\bsx))} + \norm{dT(\tilde S(\bsx))-d\tilde T(\tilde S(\bsx))}\big)\nonumber\\
    &\qquad\qquad\le L_{dT}\norm[{L^\infty([-1,1]^d)}]{S-\tilde S}+\norm[{L^\infty([-1,1]^d)}]{dT-d\tilde T}.
  \end{align*}
  Together with the first statement this concludes the proof.
\end{proof}
}
\subsection{Thm.~\ref{THM:HTVKL}}

{We'll start by bounding $|\det dS-\det d\tilde S|$.  To this end
  we need the following lemma.}
\begin{lemma}\label{LEMMA:DET}%
  {Let $(a_j)_{j=1}^d$, $(b_j)_{j=1}^d\subseteq (0,\infty)$.
  Then with
  $a_{\rm min}=\min_{j}a_j$, $b_{\rm min}=\min_{j}b_j$
  and
  \begin{equation*}
    C\dfn\frac{\exp\left(\sum_{j=1}^d\frac{|a_j-b_j|}{a_{\rm min}}\right)
      \prod_{j=1}^d a_j}{\min\{a_{\rm min},b_{\rm min}\}}
  \end{equation*}
  holds
  \begin{equation}\label{eq:proddiff}
    \left|\prod_{j=1}^d a_j-\prod_{j=1}^d b_j\right|
    \le C \sum_{j=1}^d|a_j-b_j|.
  \end{equation}}
\end{lemma}
\begin{proof}
  Since $\log(1+x)\le x$ for all $x\ge 0$,
  \begin{align}\label{eq:logajepsjdiff}
    |\log(a_j)-\log(b_j)|
    &=|\log(\max\{a_j,b_j\})-\log(\min\{a_j,b_j\})|\nonumber\\
    &=\log\left( 1+\frac{\max\{a_j,b_j\}-\min\{a_j,b_j\}}{\min\{a_j,b_j\}}\right)\nonumber\\
    &\le
      \frac{|a_j-b_j|}{\min\{a_j,b_j\}}\qquad\forall {j\le d}.
  \end{align}
  For every {$x\in\R$, $\exp:(-\infty,x]\to\R$} has Lipschitz
  constant $\exp(x)$. Thus, since
  {$\max\{a_j,b_j\}\le a_j+|a_j-b_j|$},
  \begin{align}\label{eq:proddiffajepsj}
    \left|\prod_{j=1}^d a_j-\prod_{j=1}^d b_j\right|
    &= {\left|\exp\left(\sum_{j=1}^d\log(a_j)\right)-\exp\left(\sum_{j=1}^d\log(b_j)\right)\right|}\nonumber\\
    &\le \exp\left(\sum_{j=1}^d\log(a_j+|a_j-b_j|)\right)\sum_{j=1}^d \frac{|a_j-b_j|}{\min\{a_j,b_j\}}.
  \end{align}
  {Set $a_{\rm min}=\min_{j\le d}a_j>0$.}
  Using again $\log(1+x)\le x$ so that
  \begin{equation*}
    \log(a_j+|a_j-b_j|)=\log\left(a_j\left(1+\frac{|a_j-b_j|}{a_j}\right)\right)
    \le\log(a_j)+\frac{|a_j-b_j|}{a_{\rm min}}
  \end{equation*}
  we get
  {
  \begin{align*}
\exp\left(\sum_{j=1}^d\log(a_j+|a_j-b_j|)\right)
    &\le \exp\left(\sum_{j=1}^d\Bigg(\log(a_j)+\frac{|a_j-b_j|}{a_{\rm min}}\Bigg)\right)\nonumber\\
    &\le \exp\left(\sum_{j=1}^d\frac{|a_j-b_j|}{a_{\rm min}}\right)
      \prod_{j=1}^da_j<\infty.\qedhere
  \end{align*}}
\end{proof}

Let
\begin{equation}\label{eq:sminmax}
  S_{\rm min}\dfn \min_{j=1,\dots,d}\min_{\bsx\in [-1,1]^j}\partial_{j}S_j(\bsx),\qquad
  S_{\rm max}\dfn \max_{\bsx\in [-1,1]^d}|\det dS(\bsx)|.
\end{equation}
Then
Lemma \ref{LEMMA:DET} implies for any triangular
$\tilde S\in C^1([-1,1]^d;[-1,1]^d)$ (i.e.,
$\det d\tilde S=\prod_{j=1}^d\partial_{j}\tilde S_j$)
\begin{equation}\label{eq:detdist}
  |\det dS(\bsx)-\det d\tilde S(\bsx)|
  \le \frac{\exp\left(\frac{\sum_{j=1}^d|\partial_{j}S_j(\bsx_{[j]})-\partial_{j}\tilde S_j(\bsx_{[j]})|}{S_{\rm min}}\right)S_{\rm max}}{\min\{S_{\rm min},\tilde S_{\rm min}\}}
  \sum_{j=1}^d|\partial_{j}S_j(\bsx_{[j]})-\partial_{j}\tilde S_j(\bsx_{[j]})|.
\end{equation}

{Thus we can show the following:}
\begin{lemma}\label{LEMMA:FLIP}
  Let $f_\measi$, $f_\measii\in C^1([-1,1]^d;\R)$ be two
  positive probability
  densities on $[-1,1]^d$ (w.r.t.\ $\mu$), such that $f_\measi$ has
  Lipschitz constant $L$. Let $S:[-1,1]^d\to [-1,1]^d$
  {satisfy $S^\sharp \measi=\measii$}, and let
  $\tilde S:[-1,1]^d\to [-1,1]^d$ be triangular and monotone with
  $\norm[{W^{1,\infty}([-1,1]^d)}]{S-\tilde S}\le 1$.
  
  Then there exists $C>0$ solely depending on $d$,
  {and the positive constants}
  $\tilde S_{\rm min}$, $S_{\rm min}$, $S_{\rm max}$ in
  \eqref{eq:sminmax}, such that
  \begin{equation}\label{eq:diffdensities}
    \norm[L^\infty({[-1,1]^d})]{f_\measii - f_\measi\circ \tilde S\det d\tilde S}\le C (\norm[L^\infty({[-1,1]^d})]{f_\measi}+L) \norm[W^{1,\infty}({[-1,1]^d})]{S-\tilde S},
  \end{equation}
  and
  \begin{equation}\label{eq:diffsqrtdensities}
    \normc[L^\infty({[-1,1]^d})]{\sqrt{f_\measii} - \sqrt{f_\measi\circ \tilde S\det d\tilde S}}\le C \frac{\norm[L^\infty({[-1,1]^d})]{f_\measi}+L}{\inf_{\bsx\in [-1,1]^d} \sqrt{f_\measii(\bsx)}} \norm[W^{1,\infty}({[-1,1]^d})]{S-\tilde S}.
  \end{equation}  
\end{lemma}

\begin{proof}%
  For $\bsx\in[-1,1]^d$
  \begin{equation}\label{eq:diffdet0}
    |f_\measi(S(\bsx))-f_\measi(\tilde S(\bsx))|\le
    L\norm[L^\infty({[-1,1]^d})]{S-\tilde S}.
  \end{equation}
  Next, \eqref{eq:detdist} yields
  \begin{equation}\label{eq:diffdet}
    \norm[L^\infty({[-1,1]^d})]{\det dS-\det d\tilde S}\le
    C \sum_{j=1}^d|\partial_{j}S_j(\bsx_{j})-\partial_{j}\tilde S_j(\bsx_{j})|
    {\le C \norm[{W^{1,\infty}([-1,1]^d)}]{S-\tilde S}.}
  \end{equation}
  Since $S^\sharp\measi=\measii$, it holds
  $f_\measi\circ S \det dS=f_\measii$ and hence
  \begin{align*}
    \norm[L^\infty({[-1,1]^d})]{f_\measii - f_\measi\circ \tilde S\det d\tilde S}%
      &\le\norm[{L^\infty([-1,1]^d)}]{\det d S}
        \norm[{L^\infty([-1,1]^d)}]{f_\measi\circ S-f_\measi\circ\tilde S}\nonumber\\
                                                                                   &\quad +\norm[{L^\infty([-1,1]^d)}]{f_\measi\circ \tilde S}\norm[{L^\infty([-1,1]^d)}]{\det dS-\det d\tilde S}.
  \end{align*}
  Together with \eqref{eq:diffdet0} and \eqref{eq:diffdet} we obtain
  \eqref{eq:diffdensities}.

  For any $x$, $y\ge 0$ it holds
  $|x-y|=|x^{1/2}-y^{1/2}||x^{1/2}+y^{1/2}|$.
  Thus for
  $\bsx\in [-1,1]^d$
  \begin{equation*}
    |f_\measii(\bsx)^{1/2} - (f_\measi(\tilde S(\bsx))\det d\tilde S(\bsx))^{1/2}|\le \frac{|f_\measii(\bsx) - {f_\measi(\tilde S(\bsx))\det d\tilde S(\bsx)}|}{f_\measii(\bsx)^{1/2}},
  \end{equation*}
  which shows \eqref{eq:diffsqrtdensities}.
\end{proof}

{
\begin{lemma}\label{lemma:klbound}
  For all $x$, $y>0$ holds
  $x|\log(x)-\log(y)|\le (1+\frac{|x-y|}{y})|x-y|$.
\end{lemma}
\begin{proof}
For any $a>0$ the map $\log:[a,\infty)\to\R$ has Lipschitz constant
$\frac{1}{a}$ so that
    \begin{equation*}
      x|\log(x)-\log(y)|\le \frac{x}{\min\{x,y\}}|x-y|.
    \end{equation*}
    If $x>y$
    \begin{equation}\label{eq:xlogxlogy}
      x|\log(x)-\log(y)|\le \frac{x}{y}|x-y|
      = \frac{y+|x-y|}{y}|x-y|
      = \left(1+\frac{|x-y|}{y} \right)|x-y|.
    \end{equation}    
    If $x\le y$ then $\frac{x}{\min\{x,y\}}=1\le 1+\frac{|x-y|}{y}$ so
    that the claimed bound also holds.
\end{proof}
}

\begin{proof}[Proof of Thm.~\ref{THM:HTVKL}]
  The assumptions on $S$ and $T$ imply
  $\det dT(\bsx) = \frac{1}{\det dS(T(\bsx))}<\infty$ for all
  $\bsx\in [-1,1]^d$ so that $S_{\rm min}>0$.  Denote
  $\measii\dfn T_\sharp\measi$ with pushforward density
  $f_\measii(\bsx)\dfn f_\measi(S(\bsx))\det dS(\bsx)$.  We have
  \begin{equation*}
    \essinf_{x\in [-1,1]^d}f_\measii(\bsx)\ge
    \essinf_{x\in [-1,1]^d}f_\measi(\bsx)
    S_{\rm min}(\bsx)^d>0.
  \end{equation*}

  By Lemma \ref{LEMMA:WINFTY}
  \begin{equation}\label{eq:S-SepsT-Teps}
    \norm[{W^{1,\infty}([-1,1]^d)}]{S-\tilde S}
    \le (1+{\rm Lip}(S){\rm Lip}(dT))
      \norm[{L^\infty([-1,1]^d)}]{dS}
      \norm[{L^\infty([-1,1]^d)}]{d\tilde S}
          \norm[{W^{1,\infty}([-1,1]^d)}]{T-\tilde T}.
  \end{equation}
    Since $f_\measi:[-1,1]^d\to\R$ is
    Lipschitz continuous, we may apply Lemma
    \ref{LEMMA:FLIP}, and obtain together with \eqref{eq:S-SepsT-Teps}
    for $\alpha\in\{\frac{1}{2},1\}$
    \begin{equation}\label{eq:diffdensitiesKL}
      \norm[L^\infty({[-1,1]^d})]{f_\measii^\alpha - (f_\measi\circ \tilde S\det d\tilde S)^{\alpha}}\le C \norm[{W^{1,\infty}([-1,1]^d)}]{S-\tilde S}
    \end{equation}
    for a suitable constant $C$ depending on $S_{\rm min}$,
    $S_{\rm max}$, $\tilde S_{\rm min}$,
    $\norm[{L^\infty([-1,1]^d)}]{d\tilde S}$,
    $\norm[{L^\infty([-1,1]^d)}]{f_\measi}$,
    $\essinf_{\bsx\in [-1,1]^d}f_\measi(\bsx)$,
    $\essinf_{\bsx\in [-1,1]^d}f_\measii(\bsx)$,
    ${\rm Lip}(f_\measi)$,
    ${\rm Lip}(S)$ and ${\rm Lip}(dT)$.  Hence $C$ depends on
    $\tilde T$ only through $\tilde S_{\rm min}\ge \tau_0$ and
    $\norm[{L^\infty([-1,1]^d)}]{d\tilde S}\le \frac{1}{\tau_0}$.
    This implies \eqref{eq:measdiffd}
    for the Hellinger and total variation distance.

    Next we show \eqref{eq:measdiffdKL}.
    It holds $\inf_{\bsx\in [-1,1]^d}f_\measii(\bsx)\in (0,1]$ since
    $f_\measii$ is a positive probability density. We obtain
    using Lemma \ref{lemma:klbound}
    \begin{align*}
      {\rm KL}((\tilde T)_\sharp \measi\|\measii)
      &\le \int_{[-1,1]^d}|f_\measi(\tilde S(\bsx))\det d\tilde S(\bsx)||
        \log(f_\measi(\tilde S(\bsx))\det d\tilde S(\bsx))-\log(f_\measii(\bsx))|\dd\mu(\bsx)\nonumber\\
      &\le \left(1+\frac{\norm[L^\infty{([-1,1]^d)}]{f_\measii-f_\measi\circ \tilde S\det d\tilde S}}{\inf_{\bsx\in [-1,1]^d}f_\measii(\bsx)}\right) {\norm[L^\infty{([-1,1]^d)}]{f_\measii-f_\measi\circ \tilde S\det d\tilde S}}.
    \end{align*}
    Finally \eqref{eq:diffdensitiesKL} with $\alpha=1$ implies
    \eqref{eq:measdiffd} for the KL divergence, which concludes the
    proof.
\end{proof}

\subsection{Prop.~\ref{PROP:MEASCONVD}}
\begin{proof}[Proof of Prop.~\ref{PROP:MEASCONVD}]
  By Thm.~\ref{THM:POLYD}
  \begin{equation}\label{eq:T-Tepsprop}
    \norm[{W^{1,\infty}([-1,1]^k)}]{T-\tilde T_\eps}\le C \exp(-\tilde\beta N_\eps^{1/d}).
  \end{equation}

  For the Wasserstein distance,
  \cite[Theorem 2]{MR4120535} (cp.~\eqref{eq:sagiv})
  implies together with \eqref{eq:T-Tepsprop}
implies
  \begin{equation*}
    W_p((\tilde T)_\sharp\measi,\measii)
    =    W_p((\tilde T)_\sharp\measi,T_\sharp \measi)
    \le \norm[{L^\infty([-1,1]^d)}]{T-\tilde T}
    \le C \exp(-\tilde\beta N^{1/d}).
  \end{equation*}
  
  As shown in Sec.~\ref{SEC:ANT}, $T$ and $S=T^{-1}$ are analytic, and
  in particular $T$, $S\in W^{2,\infty}([-1,1]^d)$.  Moreover $T$ and
  $\tilde T_\eps:[-1,1]^d\to [-1,1]^d$ are bijective by
  Thm.~\ref{THM:POLYD}.

  We wish to apply Thm.~\ref{THM:HTVKL}, which together with
  \eqref{eq:T-Tepsprop} concludes the proof.  To do so, it remains to
  show that we can find $\tau_0>0$ such that for all $\eps>0$ holds
  (cp.~\eqref{eq:sminmax})
  \begin{equation}\label{eq:tau0}
    \tilde S_{\eps,{\rm min}}>\tau_0\qquad\text{and}
    \qquad \norm[{L^\infty([-1,1]^d)}]{d\tilde S_\eps}\le \frac{1}{\tau_0}.
  \end{equation}
  
  Since $f_\measi$ and $f_\measii$ are uniformly positive and
  $f_\measii\circ T\det dT=f_\measi$,
  $\essinf_{\bsx\in [-1,1]^d}\det dT(\bsx)>0$. By
  \eqref{eq:T-Tepsprop} we have
  $\lim_{\eps\to 0}\essinf_{\bsx\in [-1,1]^d}\det d\tilde T_\eps>0$.
  Together with
  $\lim_{\eps\to 0}\norm[{L^\infty([-1,1]^d)}]{d\tilde
    T_\eps}<\infty$, this implies with
  $\tilde S_\eps=\tilde T_\eps^{-1}$ that
  $\lim_{\eps\to 0}\norm[{L^\infty([-1,1]^d)}]{d\tilde S_\eps}<\infty$.
  This can be seen by using $A^{-1}=\frac{1}{\det A} A^{\rm adj}$ for
  a matrix $A\in\R^{d\times d}$, where $A^{\rm adj}$ denotes the
  adjugate matrix, which is equal to the transpose of the cofactor
  matrix, for which each entry can be bounded in terms of
  $\norm[]{A}$. By a similar argument
  $\lim_{\eps\to 0}\tilde S_{\eps,\rm min}>0$.
  Hence \eqref{eq:tau0} holds for some $\tau_0>0$
  and all $\eps\le\eps_0$ for some $\eps_0>0$.
  Finally note that $\eps\in [\eps_0,\infty)$
  corresponds to only finitely many choices of sets $(\Lambda_{k,\eps})_{k=1}^d$
  in Thm.~\ref{THM:POLYD}. Hence, by decreasing $\tau_0>0$
  if necessary, \eqref{eq:tau0} holds for all $\eps>0$.
\end{proof}

\section{Proofs of Sec.~\ref{SEC:EXAMPLE}}
\subsection{Lemma \ref{LEMMA:EXAMPLE}}
\begin{lemma}\label{LEMMA:TAU}
  Let $\bsb=(b_j)_{j=1}^d\subset (0,\infty)$,
  $C_0\dfn \max\{1,\norm[\ell^\infty]{\bsb}\}$ and $\gamma>0$. There
  exists $(\kappa_j)_{j=1}^d\in (0,1]$ monotonically increasing and
  $\tau>0$ (both depending on $\gamma$, $\tau$, $d$, $C_0$ but
  otherwise independent of $\bsb$), such that with
  $\delta_j\dfn \kappa_j+\frac{\tau}{b_j}$ holds
  $\sum_{j=1}^kb_j\delta_j\le\gamma \delta_{k+1}$ for all
  $k\in\{1,\dots,d-1\}$, and $\sum_{j=1}^d b_j\delta_j\le \gamma$.
\end{lemma}
\begin{proof}
  Set
  $\tilde\gamma\dfn\frac{\min\{\gamma,1\}}{C_0}\le 1$,
  $\kappa_j\dfn (\frac{\tilde\gamma}{4})^{d+1-j}$ and
  $\tau\dfn (\frac{\tilde\gamma}{4})^{d+1}$. For $k\in\{1,\dots,d\}$
  \begin{equation}\label{eq:sumbjdeltaj}
    \sum_{j=1}^k b_j\delta_j = \sum_{j=1}^k b_j\Big(\kappa_j+\frac{\tau}{b_j}\Big)
    \le C_0 \sum_{j=1}^k \Big(\frac{\tilde\gamma}{4}\Big)^{d+1-j}  + \sum_{j=1}^k \tau
    \le C_0 \frac{(\frac{\tilde\gamma}{4})^{d+1-k}}{1-\frac{\tilde\gamma}{4}}+k\tau
    \le C_02\kappa_k+ k\tau.
  \end{equation}
  For $k<d$ we show
  $2\kappa_k\le
  \frac{\tilde\gamma}{2}(\kappa_{k+1}+\frac{\tau}{{b_{k+1}}})=\frac{\tilde\gamma}{2}\delta_{k+1}$
  and
  $k\tau\le
  \frac{\tilde\gamma}{2}(\kappa_{k+1}+\frac{\tau}{b_{k+1}})=\frac{\tilde\gamma}{2}\delta_{k+1}$,
  which then implies
  $\sum_{j=1}^k b_j\delta_j\le
  C_0\tilde\gamma\delta_{k+1}\le\gamma\delta_{k+1}$.  First
  \begin{equation*}
    2\kappa_k = 2 \left(\frac{\tilde\gamma}{4}\right)^{d+1-k}
    =2 \frac{\tilde\gamma}{4} \left(\frac{\tilde\gamma}{4}\right)^{d-k}
    = \frac{\tilde\gamma}{2}\left(\frac{\tilde\gamma}{4}\right)^{d+1-(k+1)}
    = \frac{\tilde\gamma}{2}\kappa_{k+1}\le
    \frac{\tilde\gamma}{2}\Big(\kappa_{k+1}+\frac{\tau}{b_{k+1}}\Big).
\end{equation*}
Furthermore
\begin{equation*}
  k\tau = k\Big(\frac{\tilde\gamma}{4}\Big)^{d+1}\le \Big(\frac{\tilde\gamma}{4}\Big)^{d+1-k}=\frac{\tilde\gamma}{4}
  \Big(\frac{\tilde\gamma}{4}\Big)^{d+1-(k+1)}
  \le \frac{\tilde\gamma}{2}\kappa_{k+1}.
\end{equation*}
Here we used $k\alpha^{d+1}\le \alpha^{d+1-k}$, which is equivalent to
$k\alpha^k\le 1$: due to $(x\alpha^x)'=\alpha^x(\log(\alpha)x+1)$,
this holds in particular for any $k=1,\dots,d$ and
$\alpha\in(0,\exp(-1)]$ since then $\log(\alpha)x+1\le 0$ for all
$x\ge 1$ and $1\cdot \alpha^1\le 1$. The claim follows with
$\alpha=\frac{\tilde\gamma}{4}\le \frac{1}{4}\le\exp(-1)$.

Finally, $\sum_{j=1}^db_j\delta_j\le \gamma$ follows by
\eqref{eq:sumbjdeltaj} with $k=d$:
$2\kappa_d=2\frac{\tilde\gamma}{4}=\frac{\tilde\gamma}{2}$ and
$d\tau = d(\frac{\tilde\gamma}{4})^{d+1}\le\frac{\tilde\gamma}{4}$ (by
the same argument as above), and thus
$C_02\kappa_d+d\tau\le C_0\frac{\tilde\gamma}{2}+\frac{\tilde\gamma}{4}\le C_0\frac{3\tilde\gamma}{4}\le\gamma$.
\end{proof}
\begin{proof}[Proof of Lemma \ref{LEMMA:EXAMPLE}]
  The solution operator $\scr{u}$ associated to \eqref{eq:pde} is
  complex Fr\'echet differentiable as a mapping from {an open
    subset of $L^\infty(\D;\C)$ containing}
  $\set{a\in L^\infty(\D;\R)}{\essinf_{x\in\D}a(x)>0}$ to the complex
  Banach space $H_0^1(\D;\C)$; see, e.g., \cite[Example
  1.2.39]{JZdiss}.  By assumption
  $\inf_{\bsy\in [-1,1]^d}\essinf_{x\in\D}a(\bsy,x)>0$. Thus there
  exists $r>0$ such that with
  \begin{equation}\label{eq:Sr}
    S\dfn \setc{a+b}{a=\sum_{j=1}^dy_j\psi_j,~\bsy\in [-1,1]^d,~\norm[{L^\infty(\D;\C)}]{b}\le r}
  \end{equation}
  the map $\scr{u}:S\to H_0^1(\D;\C)$ is complex (Fr\'echet)
  differentiable. Then also
  \begin{equation*}
    {\scr{f}_\measii(a)\dfn \frac{1}{Z}\exp\left(-\frac{(A(\scr{u}(a))-\bsvarsigma)^\top\Sigma^{-1}(A(\scr{u}(a))-\bsvarsigma)}{2}\right)}
  \end{equation*}
  (see Sec.~\ref{sec:priorposterior} for the notation) is complex
  differentiable from $S\to \C$. %
  Here we
  used that the bounded linear operator $A:H_0^1(\D;\R)\to\R$ allows a
  natural (bounded) linear extension $A:H_0^1(\D;\C)\to\C$. Using
  compactness of $[-1,1]^d$, by choosing $r>0$ small enough it holds
  \begin{equation}\label{eq:exampleML}
    M\dfn {\inf_{a\in S}|\scr{f}_\measii(a)|\le \sup_{a\in S}|\scr{f}_\measii(a)|} \dfnn L<\infty,
  \end{equation}
  and $\scr{f}_\measii:S\to \C$ is Lipschitz continuous with
  {some Lipschitz constant $m>0$.}

  With $C_6=C_6(M,L)$ as in Thm.~\ref{THM:COR:DINN}, set
  \begin{equation*}
    \gamma\dfn \min\Big\{r,\frac{C_6}{m}\Big\}.
  \end{equation*}
  Let $b_j\dfn \norm[L^\infty(\D)]{\psi_j}$ and
  $\delta_j=\kappa_j+\frac{\tau}{b_j}$ as in Lemma \ref{LEMMA:TAU}.

  Fix $\bsy\in [-1,1]^d$ and $\bsz\in\cB_{\bsdelta}$. Then by Lemma
  \ref{LEMMA:TAU}
  \begin{equation}\label{eq:djbjsumd}
    \normc[L^\infty(\D)]{\sum_{j=1}^d(y_j+z_j)\psi_j-\sum_{j=1}^dy_j\psi_j}
    \le \sum_{j=1}^d\delta_jb_j\le\gamma=\min\Big\{r,\frac{C_6}{m}\Big\}
  \end{equation}
  and similarly for $k=1,\dots,d-1$
  \begin{equation}\label{eq:djbjsumk}
    \normc[L^\infty(\D)]{\left(\sum_{j=1}^dy_j\psi_j{+\sum_{j=1}^kz_j\psi_j}\right)-\sum_{j=1}^dy_j\psi_j}
    \le \sum_{j=1}^k\delta_jb_j\le\gamma\delta_{k+1}
    \le \frac{C_6}{m}\delta_{k+1}.
  \end{equation}
  
  We verify Assumption \ref{ass:finite} {for
  $\bsy\mapsto f_\measii(\bsy)=\scr{f}_\measii(\sum_{j=1}^d
    y_j\psi_j)$:}
  \begin{enumerate}[label=(\alph*)]
  \item By definition
    {$f_\measii$} is a positive probability density on
    $[-1,1]^d$. As mentioned $\scr{f}_\measii:S\to\C$ is
    differentiable, and \eqref{eq:djbjsumd} shows
    {$\norm[L^\infty(\D)]{\sum_{j=1}^dz_j\psi_j}\le r$ for all
      $\bsz\in\cB_\bsdelta$, so that
      $\sum_{j=1}^d(y_j+z_j)\psi_j\in S$ for all $\bsy\in [-1,1]^d$.
      Hence $f_\measii(\bsx)=\scr{f}_\measii(\sum_{j=1}^dx_j\psi_j)$
      is differentiable for $\bsx\in {\cB_\bsdelta([-1,1])}$.}
  \item $M\le |f_\measii(\bsy+\bsz)|\le L$ for all $\bsy\in [-1,1]^d$
    and $\bsz\in\cB_\bsdelta$ according to \eqref{eq:exampleML},
    \eqref{eq:Sr} and \eqref{eq:djbjsumd}.
  \item $\scr{f}_\measii:S\to\C$ has Lipschitz constant $m$.  By
    \eqref{eq:djbjsumd} for any $\bsy\in [-1,1]^d$
    \begin{equation*}
      \sup_{\bsz\in\cB_\bsdelta}|f_\measii(\bsy+\bsz)-f_\measii(\bsy)|=
      \sup_{\bsz\in\cB_\bsdelta}\left|\scr{f}_\measii\left(\sum_{j=1}^d(y_j+z_j)\psi_j\right)-\scr{f}_\measii\left(\sum_{j=1}^dy_j\psi_j\right)\right|
      \le m\frac{C_6}{m}=C_6.
    \end{equation*}    
  \item Similarly, by
    \eqref{eq:djbjsumk} for any $\bsy\in [-1,1]^d$ and any
    $k\in\{1,\dots,d-1\}$
    \begin{equation*}
      \sup_{\bsz\in\cB_{\bsdelta_{[k]}}\times\{0\}^{d-k}}|f_\measii(\bsy+\bsz)-f_\measii(\bsy)|
      \le m\frac{C_6}{m}\delta_{k+1}=C_6\delta_{k+1}.
    \end{equation*}
  \end{enumerate}
\end{proof}

\bibliographystyle{abbrv} \bibliography{main}

\def\cprime{$'$} \def\cprime{$'$}
\begin{thebibliography}{10}

\bibitem{berg2018sylvester}
R.~v.~d. Berg, L.~Hasenclever, J.~M. Tomczak, and M.~Welling.
\newblock Sylvester normalizing flows for variational inference.
\newblock {\em arXiv preprint arXiv:1803.05649}, 2018.

\bibitem{MR3810512}
A.~Beskos, A.~Jasra, K.~Law, Y.~Marzouk, and Y.~Zhou.
\newblock Multilevel sequential {M}onte {C}arlo with dimension-independent
  likelihood-informed proposals.
\newblock {\em SIAM/ASA J. Uncertain. Quantif.}, 6(2):762--786, 2018.

\bibitem{MR2566594}
M.~Bieri, R.~Andreev, and C.~Schwab.
\newblock Sparse tensor discretization of elliptic {SPDE}s.
\newblock {\em SIAM J. Sci. Comput.}, 31(6):4281--4304, 2009/10.

\bibitem{transportmapslibrary}
D.~Bigoni.
\newblock {TransportMaps} library, 2016--2020.
\newblock \url{http://transportmaps.mit.edu}.

\bibitem{doi:10.1080/01621459.2017.1285773}
D.~M. Blei, A.~Kucukelbir, and J.~D. McAuliffe.
\newblock Variational inference: A review for statisticians.
\newblock {\em Journal of the American Statistical Association},
  112(518):859--877, 2017.

\bibitem{bogachevtri}
V.~I. Bogachev, A.~V. Kolesnikov, and K.~V. Medvedev.
\newblock Triangular transformations of measures.
\newblock {\em Mat. Sb.}, 196(3):3--30, 2005.

\bibitem{BonitoEtAl2019}
A.~Bonito, R.~DeVore, D.~Guignard, P.~Jantsch, and G.~Petrova.
\newblock Polynomial approximation of anisotropic analytic functions of several
  variables, 2019.
\newblock ArXiv:1904.12105.

\bibitem{brennan2020greedy}
M.~Brennan, D.~Bigoni, O.~Zahm, A.~Spantini, and Y.~Marzouk.
\newblock Greedy inference with structure-exploiting lazy maps.
\newblock {\em Advances in Neural Information Processing Systems}, 33, 2020.

\bibitem{MR3939383}
A.~Buchholz and N.~Chopin.
\newblock Improving approximate {B}ayesian computation via quasi-{M}onte
  {C}arlo.
\newblock {\em J. Comput. Graph. Statist.}, 28(1):205--219, 2019.

\bibitem{MR3706335}
P.~Chen and C.~Schwab.
\newblock Adaptive sparse grid model order reduction for fast {B}ayesian
  estimation and inversion.
\newblock In {\em Sparse grids and applications---{S}tuttgart 2014}, volume 109
  of {\em Lect. Notes Comput. Sci. Eng.}, pages 1--27. Springer, Cham, 2016.

\bibitem{cheney1966introduction}
E.~Cheney.
\newblock {\em Introduction to approximation theory}.
\newblock International series in pure and applied mathematics. McGraw-Hill
  Book Co., 1966.

\bibitem{chkifa}
A.~Chkifa.
\newblock Sparse polynomial methods in high dimension: application to
  parametric {PDE}, 2014.
\newblock Ph.D.\ thesis, UPMC, Universit\'e Paris 06, Paris, France.

\bibitem{CCS13}
A.~Chkifa, A.~Cohen, and C.~Schwab.
\newblock High-dimensional adaptive sparse polynomial interpolation and
  applications to parametric pdes.
\newblock {\em Journ. Found. Comp. Math.}, 14(4):601--633, 2013.

\bibitem{CSZ18}
A.~Cohen, {\relax Ch}.~Schwab, and J.~Zech.
\newblock Shape {H}olomorphy of the {S}tationary {N}avier--{S}tokes
  {E}quations.
\newblock {\em SIAM J. Math. Anal.}, 50(2):1720--1752, 2018.

\bibitem{cotter2013mcmc}
S.~L. Cotter, G.~O. Roberts, A.~M. Stuart, and D.~White.
\newblock {MCMC} methods for functions: modifying old algorithms to make them
  faster.
\newblock {\em Statistical Science}, pages 424--446, 2013.

\bibitem{DILI}
T.~Cui, K.~J.~H. Law, and Y.~M. Marzouk.
\newblock Dimension-independent likelihood-informed {MCMC}.
\newblock {\em Journal of Computational Physics}, 304:109--137, 2016.

\bibitem{MR3839555}
M.~Dashti and A.~M. Stuart.
\newblock The {B}ayesian approach to inverse problems.
\newblock In {\em Handbook of uncertainty quantification. {V}ol. 1, 2, 3},
  pages 311--428. Springer, Cham, 2017.

\bibitem{davis}
P.~Davis.
\newblock {\em Interpolation and Approximation}.
\newblock Dover Books on Mathematics. Dover Publications, 1975.

\bibitem{3327546.3327591}
G.~Detommaso, T.~Cui, A.~Spantini, Y.~Marzouk, and R.~Scheichl.
\newblock A {S}tein variational {N}ewton method.
\newblock In {\em Proceedings of the 32nd International Conference on Neural
  Information Processing Systems}, NIPS18, pages 9187--9197, Red Hook, NY, USA,
  2018. Curran Associates Inc.

\bibitem{MR3907407}
J.~Dick, R.~N. Gantner, Q.~T. Le~Gia, and C.~Schwab.
\newblock Higher order quasi-{M}onte {C}arlo integration for {B}ayesian {PDE}
  inversion.
\newblock {\em Comput. Math. Appl.}, 77(1):144--172, 2019.

\bibitem{DLS16_1336}
J.~Dick, Q.~T. LeGia, and C.~Schwab.
\newblock Higher order quasi monte carlo integration for holomorphic,
  parametric operator equations.
\newblock {\em SIAM Journ. Uncertainty Quantification}, 4(1):48--79, 2016.

\bibitem{MR4065222}
S.~Dolgov, K.~Anaya-Izquierdo, C.~Fox, and R.~Scheichl.
\newblock Approximation and sampling of multivariate probability distributions
  in the tensor train decomposition.
\newblock {\em Stat. Comput.}, 30(3):603--625, 2020.

\bibitem{1912.00894}
A.~Duncan, N.~Nuesken, and L.~Szpruch.
\newblock On the geometry of {S}tein variational gradient descent, 2019.

\bibitem{MR3856963}
W.~E and Q.~Wang.
\newblock Exponential convergence of the deep neural network approximation for
  analytic functions.
\newblock {\em Sci. China Math.}, 61(10):1733--1740, 2018.

\bibitem{MR2972870}
T.~A. El~Moselhy and Y.~M. Marzouk.
\newblock Bayesian inference with optimal maps.
\newblock {\em J. Comput. Phys.}, 231(23):7815--7850, 2012.

\bibitem{2002.02798}
C.~Finlay, J.-H. Jacobsen, L.~Nurbekyan, and A.~M. Oberman.
\newblock How to train your neural ode, 2020.

\bibitem{MR1669959}
T.~Gerstner and M.~Griebel.
\newblock Numerical integration using sparse grids.
\newblock {\em Numer. Algorithms}, 18(3-4):209--232, 1998.

\bibitem{GS2002}
A.~L. Gibbs and F.~E. Su.
\newblock On choosing and bounding probability metrics.
\newblock {\em International Statistical Review}, 70(3):419--435, 2002.

\bibitem{1810.01367}
W.~Grathwohl, R.~T.~Q. Chen, J.~Bettencourt, I.~Sutskever, and D.~Duvenaud.
\newblock Ffjord: Free-form continuous dynamics for scalable reversible
  generative models, 2018.

\bibitem{GOE16}
M.~Griebel and J.~Oettershagen.
\newblock On tensor product approximation of analytic functions.
\newblock {\em J. Approx. Theory}, 207:348--379, 2016.

\bibitem{herve89}
M.~Herv{\'e}.
\newblock {\em Analyticity in infinite-dimensional spaces}, volume~10 of {\em
  de Gruyter Studies in Mathematics}.
\newblock Walter de Gruyter \& Co., Berlin, 1989.

\bibitem{MR3640629}
B.~Hosseini and N.~Nigam.
\newblock Well-posed {B}ayesian inverse problems: priors with exponential
  tails.
\newblock {\em SIAM/ASA J. Uncertain. Quantif.}, 5(1):436--465, 2017.

\bibitem{pmlr-v80-huang18d}
C.-W. Huang, D.~Krueger, A.~Lacoste, and A.~Courville.
\newblock Neural autoregressive flows.
\newblock In J.~Dy and A.~Krause, editors, {\em Proceedings of the 35th
  International Conference on Machine Learning}, volume~80 of {\em Proceedings
  of Machine Learning Research}, pages 2078--2087. PMLR, 10--15 Jul 2018.

\bibitem{jaini2019sum}
P.~Jaini, K.~A. Selby, and Y.~Yu.
\newblock Sum-of-squares polynomial flow.
\newblock {\em ICML}, 2019.

\bibitem{MR2102218}
J.~Kaipio and E.~Somersalo.
\newblock {\em Statistical and computational inverse problems}, volume 160 of
  {\em Applied Mathematical Sciences}.
\newblock Springer-Verlag, New York, 2005.

\bibitem{kobyzev2019normalizing}
I.~Kobyzev, S.~Prince, and M.~A. Brubaker.
\newblock Normalizing flows: Introduction and ideas.
\newblock {\em IEEE Transactions on Pattern Analysis and Machine Intelligence},
  2020.

\bibitem{pmlr-v108-kong20a}
Z.~Kong and K.~Chaudhuri.
\newblock The expressive power of a class of normalizing flow models.
\newblock In S.~Chiappa and R.~Calandra, editors, {\em Proceedings of the
  Twenty Third International Conference on Artificial Intelligence and
  Statistics}, volume 108 of {\em Proceedings of Machine Learning Research},
  pages 3599--3609. PMLR, 26--28 Aug 2020.

\bibitem{MR4075344}
J.~Latz.
\newblock On the well-posedness of {B}ayesian inverse problems.
\newblock {\em SIAM/ASA J. Uncertain. Quantif.}, 8(1):451--482, 2020.

\bibitem{CiCP-27-379}
B.~Li, S.~Tang, and H.~Yu.
\newblock Better approximations of high dimensional smooth functions by deep
  neural networks with rectified power units.
\newblock {\em Communications in Computational Physics}, 27(2):379--411, 2019.

\bibitem{NIPS2017_6904}
Q.~Liu.
\newblock Stein variational gradient descent as gradient flow.
\newblock In I.~Guyon, U.~V. Luxburg, S.~Bengio, H.~Wallach, R.~Fergus,
  S.~Vishwanathan, and R.~Garnett, editors, {\em Advances in Neural Information
  Processing Systems 30}, pages 3115--3123. Curran Associates, Inc., 2017.

\bibitem{NIPS2016_6338}
Q.~Liu and D.~Wang.
\newblock Stein variational gradient descent: A general purpose {B}ayesian
  inference algorithm.
\newblock In D.~D. Lee, M.~Sugiyama, U.~V. Luxburg, I.~Guyon, and R.~Garnett,
  editors, {\em Advances in Neural Information Processing Systems 29}, pages
  2378--2386. Curran Associates, Inc., 2016.

\bibitem{MR3919409}
J.~Lu, Y.~Lu, and J.~Nolen.
\newblock Scaling limit of the {S}tein variational gradient descent: the mean
  field regime.
\newblock {\em SIAM J. Math. Anal.}, 51(2):648--671, 2019.

\bibitem{2004.08867}
Y.~Lu and J.~Lu.
\newblock A universal approximation theorem of deep neural networks for
  expressing probability distributions, 2020.

\bibitem{MR1511855}
W.~Markoff and J.~Grossmann.
\newblock \"{U}ber {P}olynome, die in einem gegebenen {I}ntervalle
  m\"{o}glichst wenig von {N}ull abweichen.
\newblock {\em Math. Ann.}, 77(2):213--258, 1916.

\bibitem{MR3821485}
Y.~Marzouk, T.~Moselhy, M.~Parno, and A.~Spantini.
\newblock Sampling via measure transport: an introduction.
\newblock In {\em Handbook of uncertainty quantification. {V}ol. 1, 2, 3},
  pages 785--825. Springer, Cham, 2017.

\bibitem{MR1230251}
H.~N. Mhaskar.
\newblock Approximation properties of a multilayered feedforward artificial
  neural network.
\newblock {\em Adv. Comput. Math.}, 1(1):61--80, 1993.

\bibitem{Morzfeld}
M.~Morzfeld, X.~T. Tong, and Y.~M. Marzouk.
\newblock Localization for {MCMC}: sampling high-dimensional posterior
  distributions with local structure.
\newblock {\em Journal of Computational Physics}, 380:1--28, 2019.

\bibitem{nist}
F.~W.~J. Olver, D.~W. Lozier, R.~F. Boisvert, and C.~W. Clark, editors.
\newblock {\em N{IST} handbook of mathematical functions}.
\newblock U.S. Department of Commerce, National Institute of Standards and
  Technology, Washington, DC; Cambridge University Press, Cambridge, 2010.

\bibitem{OSZ19}
J.~A.~A. Opschoor, C.~Schwab, and J.~Zech.
\newblock Exponential {ReLU} {DNN} expression of holomorphic maps in high
  dimension.
\newblock Technical Report 2019-35, Seminar for Applied Mathematics, ETH
  Z{\"u}rich, Switzerland, 2019.

\bibitem{papamakarios2019normalizing}
G.~Papamakarios, E.~Nalisnick, D.~J. Rezende, S.~Mohamed, and
  B.~Lakshminarayanan.
\newblock Normalizing flows for probabilistic modeling and inference.
\newblock {\em Journal of Machine Learning Research}, 22:1--64, 2021.

\bibitem{papamakarios2017masked}
G.~Papamakarios, T.~Pavlakou, and I.~Murray.
\newblock Masked autoregressive flow for density estimation.
\newblock In {\em Advances in Neural Information Processing Systems}, pages
  2338--2347, 2017.

\bibitem{MR3800241}
M.~D. Parno and Y.~M. Marzouk.
\newblock Transport map accelerated {M}arkov chain {M}onte {C}arlo.
\newblock {\em SIAM/ASA J. Uncertain. Quantif.}, 6(2):645--682, 2018.

\bibitem{MR1616049}
J.~O. Ramsay.
\newblock Estimating smooth monotone functions.
\newblock {\em J. R. Stat. Soc. Ser. B Stat. Methodol.}, 60(2):365--375, 1998.

\bibitem{pmlr-v37-rezende15}
D.~Rezende and S.~Mohamed.
\newblock Variational inference with normalizing flows.
\newblock In F.~Bach and D.~Blei, editors, {\em Proceedings of the 32nd
  International Conference on Machine Learning}, volume~37 of {\em Proceedings
  of Machine Learning Research}, pages 1530--1538, Lille, France, 07--09 Jul
  2015. PMLR.

\bibitem{10.5555/1051451}
C.~P. Robert and G.~Casella.
\newblock {\em Monte Carlo Statistical Methods (Springer Texts in Statistics)}.
\newblock Springer-Verlag, Berlin, Heidelberg, 2005.

\bibitem{MR49525}
M.~Rosenblatt.
\newblock Remarks on a multivariate transformation.
\newblock {\em Ann. Math. Statistics}, 23:470--472, 1952.

\bibitem{rudolf2018generalization}
D.~Rudolf and B.~Sprungk.
\newblock On a generalization of the preconditioned crank--nicolson metropolis
  algorithm.
\newblock {\em Foundations of Computational Mathematics}, 18(2):309--343, 2018.

\bibitem{MR4120535}
A.~Sagiv.
\newblock The {W}asserstein distances between pushed-forward measures with
  applications to uncertainty quantification.
\newblock {\em Commun. Math. Sci.}, 18(3):707--724, 2020.

\bibitem{santambrogio}
F.~Santambrogio.
\newblock {\em Optimal transport for applied mathematicians}, volume~87 of {\em
  Progress in Nonlinear Differential Equations and their Applications}.
\newblock Birkh\"{a}user/Springer, Cham, 2015.
\newblock Calculus of variations, PDEs, and modeling.

\bibitem{MR3640631}
R.~Scheichl, A.~M. Stuart, and A.~L. Teckentrup.
\newblock Quasi-{M}onte {C}arlo and multilevel {M}onte {C}arlo methods for
  computing posterior expectations in elliptic inverse problems.
\newblock {\em SIAM/ASA J. Uncertain. Quantif.}, 5(1):493--518, 2017.

\bibitem{MR3056084}
C.~Schillings and C.~Schwab.
\newblock Sparse, adaptive {S}molyak quadratures for {B}ayesian inverse
  problems.
\newblock {\em Inverse Problems}, 29(6):065011, 28, 2013.

\bibitem{MR3580124}
C.~Schillings and C.~Schwab.
\newblock Scaling limits in computational {B}ayesian inversion.
\newblock {\em ESAIM Math. Model. Numer. Anal.}, 50(6):1825--1856, 2016.

\bibitem{MR4125981}
C.~Schillings, B.~Sprungk, and P.~Wacker.
\newblock On the convergence of the {L}aplace approximation and
  noise-level-robustness of {L}aplace-based {M}onte {C}arlo methods for
  {B}ayesian inverse problems.
\newblock {\em Numer. Math.}, 145(4):915--971, 2020.

\bibitem{MR2903278}
C.~Schwab and A.~M. Stuart.
\newblock Sparse deterministic approximation of {B}ayesian inverse problems.
\newblock {\em Inverse Problems}, 28(4):045003, 32, 2012.

\bibitem{spantini2019coupling}
A.~Spantini, R.~Baptista, and Y.~Marzouk.
\newblock Coupling techniques for nonlinear ensemble filtering.
\newblock {\em arXiv preprint arXiv:1907.00389}, 2019.

\bibitem{spantini2018inference}
A.~Spantini, D.~Bigoni, and Y.~Marzouk.
\newblock Inference via low-dimensional couplings.
\newblock {\em The Journal of Machine Learning Research}, 19(1):2639--2709,
  2018.

\bibitem{stuartacta}
A.~M. Stuart.
\newblock Inverse problems: a {B}ayesian perspective.
\newblock {\em Acta Numer.}, 19:451--559, 2010.

\bibitem{2006.11469}
T.~Teshima, I.~Ishikawa, K.~Tojo, K.~Oono, M.~Ikeda, and M.~Sugiyama.
\newblock Coupling-based invertible neural networks are universal
  diffeomorphism approximators, 2020.

\bibitem{2012.02414}
T.~Teshima, K.~Tojo, M.~Ikeda, I.~Ishikawa, and K.~Oono.
\newblock Universal approximation property of neural ordinary differential
  equations, 2020.

\bibitem{Tong}
X.~T. Tong, M.~Morzfeld, and Y.~M. Marzouk.
\newblock {MALA}-within-{Gibbs} samplers for high-dimensional distributions
  with sparse conditional structure.
\newblock {\em SIAM Journal on Scientific Computing}, in press, 2020.

\bibitem{MR2459454}
C.~Villani.
\newblock {\em Optimal transport}, volume 338 of {\em Grundlehren der
  Mathematischen Wissenschaften [Fundamental Principles of Mathematical
  Sciences]}.
\newblock Springer-Verlag, Berlin, 2009.
\newblock Old and new.

\bibitem{wehenkel2019unconstrained}
A.~Wehenkel and G.~Louppe.
\newblock Unconstrained monotonic neural networks.
\newblock {\em arXiv preprint arXiv:1908.05164}, 2019.

\bibitem{yarotsky}
D.~Yarotsky.
\newblock Error bounds for approximations with deep {ReLU} networks.
\newblock {\em Neural Netw.}, 94:103--114, 2017.

\bibitem{MR2280784}
S.~T. Yau and L.~Zhang.
\newblock An upper estimate of integral points in real simplices with an
  application to singularity theory.
\newblock {\em Math. Res. Lett.}, 13(5-6):911--921, 2006.

\bibitem{JZdiss}
J.~Zech.
\newblock Sparse-{G}rid {A}pproximation of {H}igh-{D}imensional {P}arametric
  {PDE}s, 2018.
\newblock Dissertation 25683, ETH Z\"urich,
  \url{http://dx.doi.org/10.3929/ethz-b-000340651}.

\bibitem{2006.06994}
J.~Zech and Y.~Marzouk.
\newblock Sparse approximation of triangular transports on bounded domains,
  2020.
\newblock arXiv:2006.06994v1.

\bibitem{zm2}
J.~Zech and Y.~Marzouk.
\newblock Sparse approximation of triangular transports. {P}art {II}: the
  infinite dimensional case, 2021.

\bibitem{ZS17}
J.~Zech and C.~Schwab.
\newblock Convergence rates of high dimensional {S}molyak quadrature.
\newblock {\em ESAIM Math. Model. Numer. Anal.}, 54(4):1259--1307, 2020.

\end{thebibliography}
\end{document}